\documentclass[preprint]{imsart}
\usepackage{fullpage}

\RequirePackage{amsthm,amsmath,amsfonts,amssymb}
\RequirePackage[numbers]{natbib}
\RequirePackage[colorlinks,citecolor=blue,urlcolor=blue]{hyperref}
\usepackage{graphicx, xcolor}
\usepackage{amsmath, amssymb, amsfonts, amsthm}
\usepackage{booktabs}
\usepackage{tikz}
\usetikzlibrary{arrows,patterns}

\startlocaldefs
\theoremstyle{plain}
\newtheorem{theorem}{Theorem}[section]
\newtheorem{lemma}[theorem]{Lemma}
\newtheorem{proposition}[theorem]{Proposition}

\theoremstyle{definition}

\newtheorem{example}[theorem]{Example}
\newtheorem{remark}[theorem]{Remark}

\definecolor{pythonblue}{RGB}{31 119 180}
\definecolor{pythongreen}{RGB}{44 160 44}
\definecolor{pythonred}{RGB}{214 39 40}

\newcommand{\Rbb}{\mathbb{R}}
\newcommand{\Zbb}{\mathbb{Z}}
\newcommand{\Nbb}{\mathbb{N}}
\newcommand{\Cbb}{\mathbb{C}}

\DeclareSymbolFont{bbold}{U}{bbold}{m}{n}
\DeclareSymbolFontAlphabet{\mathbbold}{bbold}
\newcommand{\ind}{\mathbbold{1}}

\newcommand{\Pbb}{\mathbb{P}}
\newcommand{\Ebb}{\mathbb{E}}
\newcommand{\Var}{\textnormal{Var}}
\newcommand{\Cov}{\textnormal{Cov}}

\newcommand{\Ncal}{\mathcal{N}}
\newcommand{\Pcal}{\mathcal{P}}

\newcommand{\Mcal}{\mathcal{M}}
\newcommand{\Hcal}{\mathcal{H}}
\newcommand{\diff}{\textnormal{d}}
\newcommand{\Span}{\textnormal{span}}
\newcommand{\trace}{\operatorname{Tr}}

\DeclareMathOperator{\Tan}{Tan}
\newcommand{\vol}{\textnormal{vol}}
\newcommand{\B}{\mathcal{B}}
\newcommand{\N}{\mathcal{N}}
\newcommand{\T}{\mathcal{T}}

\DeclareMathOperator{\Log}{Log}
\DeclareMathOperator{\Exp}{Exp}

\DeclareMathOperator{\diam}{diam}
\DeclareMathOperator{\SO}{SO}

\DeclareMathOperator{\KL}{D_{\textnormal{KL}}}

\endlocaldefs

\begin{document}

\begin{frontmatter}
\title{Nonparametric Riemannian Empirical Bayes,~\\and Denoising Measurements on Manifolds}
\runtitle{Nonparametric Riemannian Empirical Bayes}

\begin{aug}
\author[A]{\fnms{Adam Quinn}~\snm{Jaffe}\ead[label=e1]{a.q.jaffe@columbia.edu}}
\author[B]{\fnms{Leonardo}~V.~\snm{Santoro}\ead[label=e2]{leonardo.santoro@epfl.ch}}
\author[A]{\fnms{Bodhisattva}~\snm{Sen}\ead[label=e3]{b.sen@columbia.edu}}
\address[A]{Department of Statistics, Columbia University, New York, NY, USA\printead[presep={,\ }]{e1,e3}}

\address[B]{Institut de Math\'ematiques, \'Ecole Polytechnique F\'ed\'erale de Lausanne, CH\printead[presep={,\ }]{e2}}
\end{aug}

\begin{abstract}
    We initiate the study of nonparametric empirical Bayes denoising methods in the setting where both the latent variables and their measurements lie on a compact Riemannian manifold, and where the likelihood is a Riemannian Gaussian distribution.
    Our starting point is a novel Tweedie-Eddington formula for Riemannian Gaussian mixture models which identifies a certain surrogate oracle denoiser in terms of the marginal distribution of the measurements; it avoids the explicit computation of the posterior Fr\'echet mean (as required by the Bayes denoiser) via a first-order approximation, hence we refer to it as the ``tangential'' Bayes denoiser.
    We show that this surrogate oracle achieves nearly the Bayes risk in a low-noise regime, we construct a fully data-driven approximation of it using the spectral theory of the Laplace-Beltrami operator, and we establish finite-sample rates of convergence for the distance between the the surrogate oracle and its approximation.
    Contrasting the nearly-parametric rates from the Euclidean setting, the rates in the Riemannian setting are slower due to the singularities of the Riemannian Gaussian density at the cut locus of its Fr\'echet mean; in the special case of the circle we establish matching lower bounds which show that our proposed denoiser is minimax-optimal, and that the denoising problem exhibits a genuinely nonparametric rate of convergence.
    Our theory and simulations reveal that pooling information across different measurements---namely, by shrinking towards high-density regions and stretching away from low-density regions---can substantially outperform the na\"ive denoiser that treats each measurement separately.
    Lastly, we implement our methodology in two scientific applications: in astronomy, the sphere-valued problem of denoising the locations of gamma ray bursts; in structural biology, the torus-valued problem of denoising pairs of torsion angles of adjacent amino acids in a protein (i.e., the Ramachandran plot).
\end{abstract}



\end{frontmatter}
\setcounter{tocdepth}{1}
\tableofcontents

    
\section{Introduction}

    Let $\Mcal$ be a compact connected Riemannian manifold with geodesic distance $d$, and consider the problem of estimating a \textit{latent variable} $\Theta\in\Mcal$ given a \textit{measurement} $X\in\Mcal$ thereof. This problem arises naturally in many scientific settings, and in this paper we are particularly motivated by two applications:
    the first takes place on the sphere $\Mcal:=\mathbb{S}^2$ and requires denoising the positions of a catalog of observed gamma ray bursts, which is a long-standing problem in astronomy known as \textit{gamma-ray burst localization} \cite{UlyssesGRB, Connaughton_2015, GRB_errors};
    the second takes place on the torus $\mathcal{M}:=\mathbb{S}^1\times \mathbb{S}^1$ and requires denoising the pairs of torsion angles formed by adjacent amino acids along a protein, usually referred to as a \textit{Ramachandran plot} \cite{Ramachandran,Hollingsworth,Pertsemlidis,Park}.
    We are also motivated by potential applications to the \textit{orientation estimation} problem from cryogenic electron microscopy \cite{Harpaz_Shkolnisky_2023, Mendez, BayesianOrientationCryoEM} which takes place on the Lie group of rotations, $\Mcal=\SO(3)$.
    
    To formulate our problem more precisely, suppose that $G$ is a fixed but unknown probability distribution on $\Mcal$, and consider the following model for a pair of random variables $(\Theta,X)\in\Mcal\times\Mcal$:
    \begin{equation}\label{eqn:pop-model}
        \Theta \sim G, \qquad \mbox{and} \qquad (X \mid \Theta = \theta) \sim P_{\theta,\sigma^2}.
    \end{equation}
Here, $\{P_{\theta,\sigma^2}\}_{\theta\in\Mcal}$ is the family of \textit{Riemannian Gaussian distributions} \cite{PennecRiemNorm} on $\Mcal$, defined so that $P_{\theta,\sigma^2}$ has a density with respect to the volume form $\diff x$ of $\Mcal$ given by
\begin{equation}\label{eqn:Likelihood}
    p_{\theta,\sigma^2}(x) \propto \exp\left(-\frac{d^2(x,\theta)}{2\sigma^2}\right)
\end{equation}
for all $\theta,x\in\Mcal$, where $\sigma^2>0$ is a fixed and known quantity representing the scale of the noise of the measurement $X$ of $\Theta$.
We emphasize that both $\Theta$ and $X$ are elements of the Riemannian manifold $\Mcal$.
We also write
\begin{equation*}
    f(x):= \int_{\Mcal}p_{\theta,\sigma^2}(x)\,\diff G(\theta), \qquad \mbox{for } x \in \Mcal
\end{equation*}
for the marginal density of $X$. Further, we consider the setting that $(\Theta_1,X_1),\ldots ,(\Theta_n,X_n)$ are i.i.d.~from model~\eqref{eqn:pop-model}, meaning $\Theta_1,\ldots, \Theta_n$ is an ensemble of unobserved latent variables and $X_1,\ldots, X_n$ are measurements thereof.
Our goal is to estimate (or predict) the latent variables $\Theta_1,\ldots,\Theta_n$ from the measurements $X_1,\ldots, X_n$, which can be seen as the analog of a standard \textit{denoising} problem but for data living on a Riemannian manifold.

Mathematically, we pose the denoising problem as aiming to find a function $\delta:\Mcal\to\Mcal$, called a \textit{denoiser}, achieving a small value of the following risk, which is the average expected squared geodesic distance when integrated over the joint distribution of $(\Theta_1,X_1),\ldots, (\Theta_n,X_n)$:
\begin{equation}\label{eqn:emp-prob}
		\begin{cases}
			\textnormal{minimize} &\Ebb\left[\frac{1}{n}\sum_{i=1}^{n}d^2(\delta(X_i),\Theta_i)\right] \\
            \textnormal{over} &\delta:\Mcal\to\Mcal.
		\end{cases}
	\end{equation}
    By exchangeability, problem~\eqref{eqn:emp-prob} is equivalent to the following, in which we integrate over the joint distribution of $(\Theta,X)$:
	\begin{equation}\label{eqn:pop-prob}
		\begin{cases}
			\textnormal{minimize} &\Ebb[d^2(\delta(X),\Theta)] \\
			\textnormal{over} &\delta:\Mcal\to\Mcal,
		\end{cases}
	\end{equation}
    and standard Bayesian arguments show that the solution to \eqref{eqn:pop-prob} is the minimizer of the posterior loss, namely the \textit{posterior Fr\'echet mean}
    \begin{equation}\label{eqn:pop-FM}
        \delta_{\B}(x) \; \in \;\underset{\vartheta\in\Mcal}{\arg\min}\;\Ebb[d^2(\vartheta,\Theta)\,|\,X=x].
    \end{equation}
    (In fact, $\delta_{\B}$ is also optimal in a generalization of problem~\eqref{eqn:emp-prob} in which a different denoiser $\delta_i$ can be applied to each coordinate and each $\delta_i$ can depend on all of $X_1,\ldots, X_n$.)    
    We refer to $\delta_{\B}$ as the \textit{oracle Bayes denoiser}, since it depends on the unknown distribution $G$ of $\Theta$ and hence cannot be implemented in most statistical settings.
    
    Nonetheless, the \textit{empirical Bayes} approach to denoising is a collection of methods for constructing a fully data-driven denoiser $\hat{\delta}$ which achieves nearly the same risk as the oracle denoiser $\delta_{\B}$; such $\hat{\delta}$ often have a substantially smaller risk than the \textit{na\"ive denoiser}
    \begin{equation}
        \delta_{\N}(x):=x
    \end{equation}
    which is implicitly used in works that do not employ techniques from empirical Bayes.
    In the standard Euclidean setting, there is a large literature on statistical theory and methodology for empirical Bayes approaches to denoising, e.g., \cite{Zhang1997, JiangZhang, SahaGuntuboyina, Soloff, Robbins1951, Robbins1956, efron2011tweedie, Efron2014, efron2019bayes}.

    In this paper, we develop a nonparametric empirical Bayes approach to denoising on compact Riemannian manifolds.
    For the Riemannian Gaussian mixture model in~\eqref{eqn:pop-model} and~\eqref{eqn:Likelihood}, we derive an intrinsic version of the Tweedie-Eddington formula~\cite{dyson1926method, Robbins1956, IgnatiadisSen} which does not directly characterize the oracle Bayes denoiser $\delta_{\B}$ but rather characterizes a surrogate denoiser that we call the \emph{oracle tangential Bayes denoiser} $\delta_\T$.
    This suggests denoising by a plug-in procedure, leading to a fully data-driven denoiser that we call the \textit{empirical tangential Bayes denoiser} $\hat{\delta}_{\T}$: first, estimate the marginal density $f$ of $X$ by $\hat f_n$, then compute its score $\nabla \log \hat f_n$, and finally adjust each data point along the resulting vector field (hence shrinking the data towards high-density regions and stretching the data away from low-density regions).
    In the terminology of \cite{Efron2014,efron2019bayes}, this is an instance of \textit{$f$-modeling}, and the resulting estimator is simple to implement and computationally tractable.
    It turns out that $\delta_{\T}$ is provably close to $\delta_{\B}$ when $\sigma^2$ is small, and that $\hat{\delta}_{\T}$ is provably close to $\delta_{\T}$ when $n$ is large, so our proposed denoiser can achieve nearly the Bayes risk in the asymptotic regime of interest.

    There are two reasons for using $\delta_\T$ instead of $\delta_{\B}$ in this work.
    The first is that $\delta_{\T}$ is directly identifiable in terms of the marginal density $f$ of $X$ via our Tweedie-Eddington formula (see Lemma~\ref{lem:Tweedie}), while $\delta_{\B}$ is not.
    The second is that $\delta_{\T}$ is much easier to compute than $\delta_{\B}$, which typically must be numerically approximated using iterative procedures from Riemannian optimization (e.g., \cite{FrechetCircle, Chakraborty_2015_ICCV, StiefelStatistics, FrechetSphere, GeomStats, TreeOptimality}); even worse, such computations must be performed separately for each evaluation $\delta_{\B}(X_1),\ldots, \delta_{\B}(X_n)$, which can be expensive in practice.    
    However, in order to reasonably change our target from $\delta_{\B}$ to $\delta_{\T}$, we must establish that $\delta_{\B}$ and $\delta_{\T}$ are indeed close in some sense, and this is a technical challenge; the oracle Bayes denoiser $\delta_{\B}$ is exactly a (posterior) Fr\'echet mean, and existing analyses of Fr\'echet means require strict assumptions that exclude our setting of interest.
    For example, some results (e.g., \cite{BrunelSerresI,SchoetzHadamard,YunPark}) restrict to nonpositively curved Riemannian manifolds $\Mcal$, but this excludes the basic setting of the sphere $\mathbb{S}^2$ which is fundamental in applications of directional statistics.
    Other results (e.g., \cite{Escande,BrunelSerresII,EltznerHuckemann,HundrieserEltznerHuckemann,SchoetzQuadruple}) allow general Riemannian manifolds $\Mcal$ and instead require the distribution of the data to satisfy some small-support condition, with respect to the curvature of $\Mcal$; however, this excludes the empirical Bayes setting when $G$ can be  fully-supported.
    Thus, in order to establish quantitative results on the closeness of $\delta_{\B}$ and $\delta_{\T}$ for general compact Riemannian manifolds, we develop a new technique for analyzing the posterior Fre\'chet mean in the small-noise asymptotic regime, $\sigma^2\to 0$.
    
    Another technical difficulty in this work is that Riemannian Gaussian mixture densities have worse smoothness than Euclidean Gaussian mixture densities, owing to the fact that the squared geodesic distance in~\eqref{eqn:Likelihood} can be nonsmooth at its cut locus.
    Thus, in the Riemannian setting our denoiser $\hat{\delta}_{\T}$ exhibits genuinely nonparametric rates of convergence depending on the smoothness of the density $g$ of $G$, rather than the nearly-parametric behavior familiar from Euclidean empirical Bayes~\cite{JiangZhang, brown2009nonparametric, SahaGuntuboyina}.
    While we can explicitly describe these rates of convergence in several special cases of interest (and provide matching lower bounds in the case of the circle $\Mcal=\mathbb{S}^1$), a complete analysis would require precise quantitative control on the smoothing properties of Riemannian Gaussian mixture densities and their singular behavior near cut loci.
    Such a theory is not yet available in full generality, except in some special cases.

    \begin{figure}[t]
        \centering
        \includegraphics[width=1\linewidth]{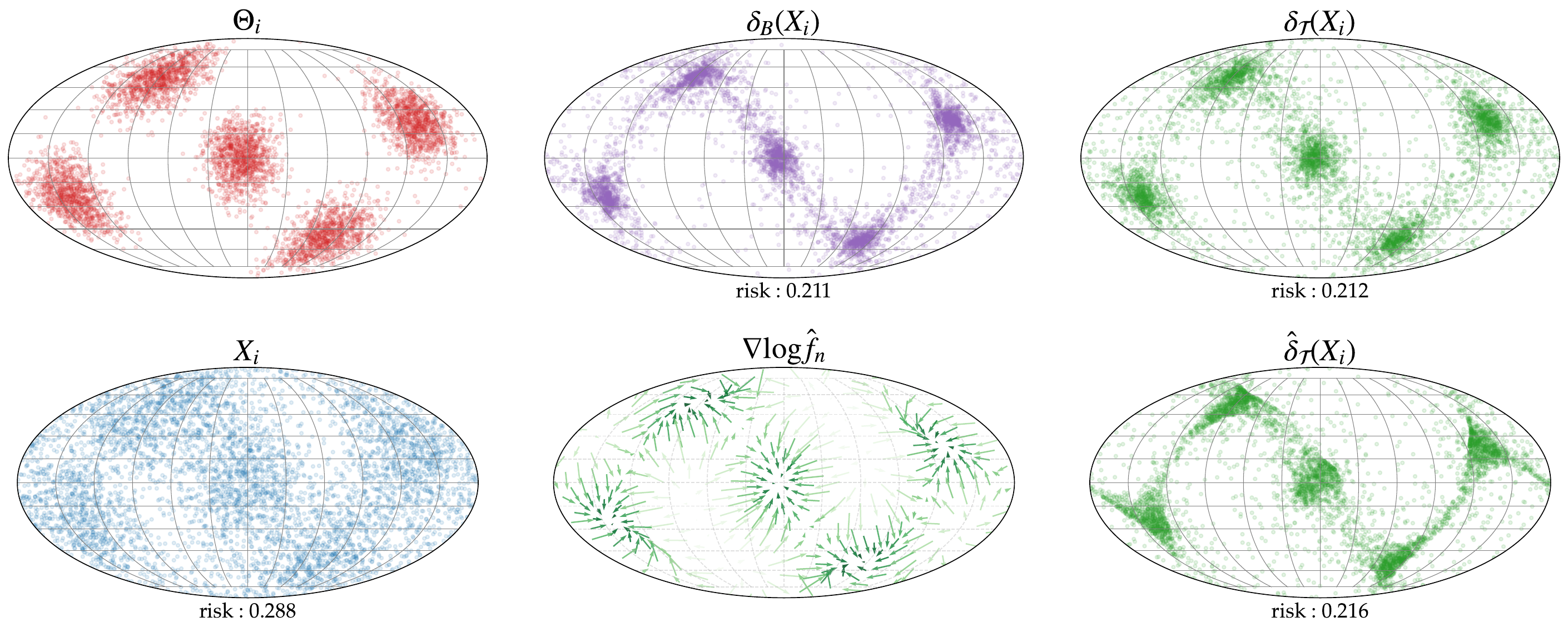}
        \caption{
        An illustration of the oracle and empirical denoisers, in the sphere $\mathbb{S}^2$.
        The latent variables $\Theta_1,\ldots,\Theta_n$  (top left) come from a five-component Riemannian Gaussian mixture model, and the measurements $X_1,\ldots,X_n$ (bottom left) come from a Riemannian Gaussian likelihood.
        The oracle Bayes denoiser $\delta_{\B}(X_1),\ldots, \delta_{\B}(X_n)$ (top center) and the oracle tangential Bayes denoiser $\delta_{\T}(X_1),\ldots, \delta_{\T}(X_n)$ (top right) recover the latent structure.
        The empirical tangential Bayes denoiser $\hat{\delta}_{\T}(X_1),\ldots, \hat{\delta}_{\T}(X_n)$ (bottom right) achieves nearly the Bayes risk by adjusting each measurement along an estimated score field $\nabla \log \hat{f}_n$ (bottom center).
        Below the plot of each denoiser, we display a Monte Carlo estimation of its risk.}
    \label{fig:intro}
    \end{figure}

    As a preview of our methodology, we show Figure~\ref{fig:intro} depicting a simulation with {$n=5000$} points on the sphere, $\Mcal = \mathbb{S}^2$.
    Here, $G$ is a mixture of 5 Riemannian Gaussian distributions, at roughly equally-spaced positions where the squared scale parameter of each component is $\tau^2=0.05$, and the likelihood is Riemannian Gaussian with squared scale parameter equal to $\sigma^2=0.25$.
    Concretely, we note that the Bayes risk is $0.211$ while our empirical Bayes denoiser $\hat{\delta}_{\T}$ achieves a risk of $0.216$ and the na\"ive denoiser $\delta_{\N}$ achieves a risk of $0.288$; in other words, the risk of $\hat{\delta}_{\T}$ is only 2.4\% larger than the Bayes risk, and it is 25.0\% smaller than the risk of the denoiser that treats each measurement separately and does not use ideas from empirical Bayes.    

    In this work we chose the likelihood \eqref{eqn:Likelihood} to be a particular generalization of Gaussian distributions to Riemannian manifolds that is often attributed to Pennec \cite{PennecRiemNorm}, but it is worth noting that there are multiple possible generalizations, e.g., wrapped Gaussians \cite{ExponentialWrappedGaussian} and the heat-kernel Gaussians \cite{hsustochastic}.
    Our choice of Pennec's Riemannian Gaussian \eqref{eqn:Likelihood} is precisely because it admits a Tweedie-Eddington formula (see Lemma~\ref{lem:Tweedie}) which is the starting point of all empirical Bayes analysis, as it connects the unknown $G$ with the observable $f$.
    An interesting avenue for future work is to study denoising problems with these other likelihoods, but we do not know whether they are amenable to empirical Bayes analysis.
    
    \subsection{Summary of Results}

    We recall some basic elements of Riemannian manifolds needed to state our results, but we direct the reader to Section~\ref{subsec:background} for further detail.
    Recall that each $x\in\Mcal$ admits a Hilbert space $\Tan_x(\Mcal)$ called the \textit{tangent space} of $\Mcal$ at $x$ which represents the best linear approximation of $\Mcal$, and $\Tan_x(\Mcal)$ is endowed with an inner product $\langle\cdot,\cdot\rangle_x$ and a norm $\|\cdot\|_x$.
    Further, the \emph{exponential map} $\Exp_x:\Tan_x(\Mcal)\to\Mcal$ for each $x\in\Mcal$ is a smooth map sending each tangent vector $v\in\Tan_x(\Mcal)$ to the point on $\Mcal$ which results from ``shooting'' distance $\|v\|_x$ in the direction $v/\|v\|_x$ from $x$.
    Conversely, the \textit{logarithmic map} $\Log_x:\Mcal\to\Tan_x(\Mcal)$ for each $x\in\Mcal$ sends a point $\theta\in\Mcal$ to a vector $v\in\Tan_x(\Mcal)$ satisfying $\Exp_x(v) = \theta$; in other words, $\Log_x$ is a smooth right inverse of $\Exp_x$, wherever it is well-defined.
    When $\Mcal=\Rbb^m$, the exponential and logarithmic maps are simply given by $\Exp_x(v) = v+x$ and $\Log_x \theta = \theta - x$, hence $\Log_x\theta$ can be regarded as a natural, intrinsic notion of displacement of $\theta$ from $x$ in $\Mcal$.

    The first main result of this paper, as is common in most works on nonparametric empirical Bayes \cite{dyson1926method, Robbins1956, IgnatiadisSen}, is to establish a Tweedie-Eddington formula which identifies a suitable oracle denoiser in terms of the marginal density $f$ of $X$.
    In the Riemannian setting, we use the simple geometric fact $\nabla_xd^2(x,\theta) = -2\Log_x\theta$ to derive (Lemma~\ref{lem:Tweedie}) the following Tweedie-Eddington formula
    \begin{equation}\label{eqn:Tweedie}
        \Ebb[\Log_{x}\Theta\,|\,X=x] = \sigma^{2}\nabla_{x}\log f(x).
    \end{equation}
    While this does not directly characterize the Bayes denoiser $\delta_{\B}$ in terms of $f$, it suggests considering the following oracle-level denoiser as a surrogate:
    \begin{equation}\label{eq:Tan-Bayes}
        \delta_{\T}(x) :=\Exp_x\left(\sigma^2 \nabla_x \log f(x)\right) =  \Exp_x\left(\Ebb[\Log_x\Theta\,|\,X=x]\right).
    \end{equation}
    To gain some intuition for this denoiser, note for each $x\in\Mcal$ that $\delta_{\T}(x)$ is precisely the result of approximately solving the optimization problem~\eqref{eqn:pop-FM} with a single step of Riemannian gradient descent, with stepsize $1/2$ and with initialization $\delta_{\N}(x) =x$.
    Indeed, at the point $\vartheta\in\Mcal$, the objective function $\Phi_x(\vartheta):=\Ebb[d^2(\vartheta,\Theta)\,|\,X=x]$ satisfies $\nabla_\vartheta\Phi_x(\vartheta) = -2\Ebb[\Log_{\vartheta}\Theta\,|\,X=x]$, hence we have
    \begin{equation}\label{eq:Tan-Bayes-GD}
        \delta_{\T}(x) = \Exp_x\left(-\frac{1}{2}\nabla_{\vartheta}\Phi_x(x)\right).
    \end{equation}
    (As explained in \cite[Section~4.6]{PennecRiemNorm} on general numerical algorithms for Fr\'echet means, this can also be seen as a single step of an approximate Riemannian Newton's method, where the Hessian of $\Phi_x$ is approximated by ignoring certain curvature effects.)
    Since $\delta_{\T}$ is thus a first-order approximation of $\delta_{\B}$, we refer to $\delta_{\T}$ as the \textit{oracle tangential Bayes denoiser}.
    Note in the Euclidean case of $\Mcal=\Rbb^m$ that \eqref{eq:Tan-Bayes} coincides with the usual Tweedie-Eddington formula and yields $\delta_{\T} = \delta_{\B}$.

    The second main result of the paper establishes that $\delta_{\T}$ has nearly the same risk as $\delta_{\B}$ as $\sigma^2\to 0$, which is the regime of interest in our scientific applications where $\sigma^2>0$ is known to be small.
    (The regime $\sigma^2\to 0$ has been studied in some recent empirical Bayes works, e.g., \cite{LiangShrinkageI,LiangShrinkageII}.)
    To state this, let us define the risk of any denoiser $\delta:\Mcal\to\Mcal$ as
    \begin{equation}\label{eqn:def-risk}
        R(\delta)=R(\delta,\sigma^2) := \Ebb\left[d^2(\delta(X),\Theta)\right] = \int_{\Mcal}\int_{\Mcal}d^2(\delta(x),\theta)p_{\theta,\sigma^2}(x)\diff x\,\diff G(\theta)
    \end{equation}
    where we emphasize the (squared) scale $\sigma^2>0$ of the likelihood in~\eqref{eqn:Likelihood} when necessary.
    Then, our main oracle-level result (Theorem~\ref{thm:uncond-main}) states that for any compact Riemannian manifold $\Mcal$ we have
    \begin{equation}\label{eqn:oracle-risk-thm}
        0 \le R(\delta_{\T},\sigma^2)-R(\delta_{\B},\sigma^2)  \le O(\sigma^5),
    \end{equation}
    meaning $\delta_{\B}$ and $\delta_{\T}$ have approximately the same risk as each other (relative to the usual scaling of $\sigma^2$).
    This result may be expected in a heuristic sense since, if $\sigma^2>0$ is small enough so that the conditional distribution of $\Theta$ given $\{X = x\}$ is highly concentrated near $x$ where $\Mcal$ is well-approximated by $\Tan_x(\Mcal)$, then we have
    \begin{equation}\label{eqn:deltaT-heuristic}
        \delta_{\T}(x) = \Exp_x(\Ebb[\Log_x\Theta\,|\,X=x]) \approx \Ebb[\Exp_x(\Log_x\Theta)\,|\,X=x]) = \Ebb[\Theta\,|\,X=x] \approx \delta_{\B}(x).
    \end{equation}
    However, rigorously establishing~\eqref{eqn:oracle-risk-thm} is difficult, because existing techniques for quantitatively analyzing Fr\'echet means do not apply in our setting where the conditional distribution of $\Theta$ given $\{X=x\}$ has full support for every $\sigma^2>0$.
    To resolve this, we introduce a new technique which uses the fixed-point equation for Fr\'echet means (i.e., the first-order optimality conditions for~\eqref{eqn:pop-FM}) to recursively bound $\Log_x(\delta_{\T}(x))-\Log_x(\delta_{\B}(x))$ in terms of $\Log_x(\delta_{\B}(x))$; in order to prevent this recursive bound from diverging, we need very sharp quantitative control on $\Log_x(\delta_{\T}(x))$ and related quantities, uniformly in $x\in\Mcal$.
    We also point out that, in the simulation studies of Section~\ref{sec:sim}, the denoisers $\delta_{\B}$ and $\delta_{\T}$ have extremely close risks for even moderate values of $\sigma^2>0$, to the point that the resulting curves are nearly indistinguishable.
  
    Our last contribution is to develop an empirical Bayes approximation $\hat{\delta}_{\T}$ of $\delta_{\T}$ and provide statistical guarantees for it; the approximation is based on an estimate $\hat{f}_n$ of the density $f$ which we then plug-in for $f$ in expression~\eqref{eq:Tan-Bayes}.
    Following the works \cite{Hendriks1990, InverseProbManifold}, we form a density estimate by utilizing the eigendecomposition of the Laplace-Beltrami operator $\Delta:L^2(\Mcal)\to L^2(\Mcal)$, which is a useful method for adapting Fourier analysis to Riemannian manifolds.
    More precisely, we consider the density estimator defined, for $\{\phi_j\}_{j=0}^{\infty}$ being a countable basis of eigenfunctions of $\Delta$, via
    \begin{equation}
        \hat{f}_n(x) := \frac{1}{n}\sum_{i=1}^{n}K(x,X_i,\ell_n), \qquad \textnormal{ where}\qquad K(x,y,\ell) = \sum_{j=0} ^{\ell}\phi_j(x)\overline{\phi_j(y)}
    \end{equation}
    for all $x\in\Mcal$.
    Then, we define the \textit{empirical tangential Bayes denoiser} as
    \begin{equation}
        \hat{\delta}_{\T}(x):=\Exp_x\left(\frac{\sigma^2\nabla_x \hat{f}_n(x)}{\max\{\hat{f}_n(x),\rho\}}\right)
    \end{equation}
    for some carefully chosen $\ell_n\ge 0$ and $\rho>0$. 
    We emphasize that $\hat{\delta}_{\T}(x)$ is very easy to compute; after computing the density estimate $\hat{f}_n$ and its gradient $\nabla_x\hat{f}_n(x)$, all subsequent evaluations of $\hat{\delta}_{\T}(x)$ require only computing the exponential map $\Exp_x(\cdot)$, which relies on the local geometry of $\Mcal$ near $x$.

    We establish explicit rates of convergence for the empirical tangential Bayes denoiser $\hat{\delta}_{\T}$.
    To state this, let us define the distance of an arbitrary denoiser $\delta:\Mcal\to\Mcal$ from $\delta_{\T}$ as
    \begin{equation}\label{eqn:def-dist}
        D_{\T}(\delta) := \Ebb\left[d^2(\delta(X),\delta_{\T}(X))\right] =\int_{\Mcal}d^2(\delta(x),\delta_{\T}(x))f(x)\diff x.
    \end{equation}
    While in the Euclidean setting of $\Mcal=\Rbb^m$ we have that $D_{\T}(\delta)$ exactly equals the excess risk relative to the tangential Bayes (equivalently, Bayes) denoiser, i.e., $R(\delta)=R(\delta_{\T})+D_{\T}(\delta)$, in the Riemannian setting we have only
    $R(\delta)\le (1+\varepsilon)R(\delta_{\T})+C_{\varepsilon}D_{\T}(\delta)$ for all $\varepsilon>0$ (Lemma~\ref{lem:risk-to-dist}); as such, bounding $D_{\T}$ leads to a non-sharp oracle inequality for the risk, but which can be made arbitrarily close to sharp.
    Using this natural notion of distance, we show (Theorem~\ref{thm:EB-risk}) that, under suitable conditions, there exists $0<\gamma<1$, such that choosing $\ell_n$ and $\rho$ appropriately leads to the bound
    \begin{equation}\label{eqn:EB-result}
        \Ebb\left[D_{\T}\left(\hat{\delta}_{\T}\right)\right] = \Ebb\left[\int_{\Mcal}d^2\left(\hat{\delta}_{\T}(x),\delta_{\T}(x)\right)f(x)\diff x\right] \lesssim_{\Mcal,G,\sigma^2,\gamma} \frac{1}{n^{\gamma}}
    \end{equation}
    with constants that do not depend on $n$.
    Along the way, we also develop (Proposition~\ref{prop:kde-error}) some results on the rate of convergence of density estimation and gradient density estimation for Riemannian Gaussian mixtures, which slightly extend and sharpen some results from the earlier work~\cite{Hendriks1990}.

    It turns out that the rate of convergence of $\hat{\delta}_{\T}$ (i.e., the exponent $0<\gamma<1$ above) depends directly on the smoothness of the marginal density $f$ of $X$, which can suffer from the fact that the squared geodesic distance functional is not globally smooth and its Hessian exhibits singular behaviors near the cut locus (Example~\ref{ex:sphere}).
    This is a phenomenon not seen in the Euclidean setting, where the density of a Euclidean Gaussian mixture is always infinitely differentiable (in fact, analytic) irrespective of the regularity of the mixing measure.
    So, while our theory holds under direct assumptions on the smoothness of $f$, it is natural to ask whether smoothness of $f$ can be instead guaranteed from suitable smoothness of $g$.   It turns out that this is a hard question which depends on the geometry of $\Mcal$ and must be worked out on a case-by-case basis; we address this problem in detail (Subsection~\ref{subsec:sobolev}) for several geometries of interest (e.g., $\mathbb{S}^1,\mathbb{S}^2,\mathbb{S}^1\times\mathbb{S}^1$, and $\SO(3)$).
    
    Consequently, our rates of convergence for $\hat{\delta}_{\T}$ are polynomial but slower than parametric.
    In the case of the circle $\Mcal=\mathbb{S}^1$, we show (Example~\ref{ex:empirical-S1-smooth} and Theorem~\ref{thm:S1-lower-bd}) matching upper and lower bounds; if $\mathcal{G}_s(\alpha,\beta,L)$ denotes the set of densities $g:\mathbb{S}^1\to\Rbb$ with $\alpha\le g\le \beta$ and $\|g^{(s)}\|_{L^2(\mathbb{S}^1)}\le L$ and we infimize over all measurable functions $\hat{\delta}:\mathbb{S}^1\times(\mathbb{S}^1)^n\to \mathbb{S}^1$, then (including the dependence on $g$ for emphasis) we have
    \begin{equation}\label{eqn:S1-minimax}
        \inf_{\hat{\delta}}\sup_{g\in\mathcal{G}_s(\alpha,\beta,L)}\Ebb_g\left[\int_{\mathbb{S}^1}d^2\left(\hat{\delta}(x;X_1,\ldots, X_n),\delta_{\T,g}(x)\right)f_g(x)\diff x\right]\asymp_{\sigma^2,s,\alpha,\beta,L} \frac{1}{n^{\frac{2s+2}{2s+5}}}
    \end{equation}
    with constants that do not depend on $n$, and where $\hat{\delta}_{\T}$ achieves the infimal rate up to constants.
    In other words, our proposed empirical tangential Bayes denoiser is minimax-optimal (with no additional logarithmic factors), and the minimax rate is genuinely nonparametric.
    The rate should be understood as $n^{-2p/(2p+1)}$ for $p=(s+2)-1$ which reflects the fact that the difficulty of denoising reduces to the difficulty of estimating the gradient of $f$, which is $(s+2)$-smooth whenever $g$ is $s$-smooth (see Lemma~\ref{lem:Sobolev-bd-S1}); our precise control of this smoothing effect in the special case of $\Mcal=\mathbb{S}^1$ comes from the fact that $f$ is given by a convolution of $g$ with a kernel that is differentiable but not continuously differentiable.


    These main results highlight the practical value of the Riemannian empirical Bayes approach to denoising.
    That is, we have shown $\delta_{\T}$ has comparable risk to $\delta_{\B}$ when $\sigma^2$ is small and that $\hat{\delta}_{\T}$ has comparable risk to $\delta_{\T}$ when $n$ is large; hence, we conclude that $\hat{\delta}_{\T}$ has comparable risk to $\delta_{\B}$ when both $\sigma^2$ is small and $n$ is large.
    Moreover, the empirical Bayes denoiser $\hat{\delta}_{\T}$ is easy to compute.
    While our theory determines an oracle choice of the hyperparameters $\ell_n$ and $\rho$, we also emphasize that they can be selected in a fully data-driven way via cross-validation with an adaptation of Hyv\"arinen’s score-matching objective \cite{hyvarinen2005,vincent2011connection} to the Riemannian setting 
    \cite{de2022riemannian}.
    We further explore these practical considerations in several detailed simulations in Section~\ref{sec:sim}, and in two scientific applications in Section~\ref{sec:appl}, which highlight the remarkable ability of empirical Bayes denoisers to adapt to latent structure in the data.
    The proofs of all results (and the statements and proofs of some auxiliary results) are given in Appendix~\ref{App:Proofs}, detail on the implementation of our denoisers is given in Appendix~\ref{app:implementation}, and additional figures are given in Appendix~\ref{app:additional-figs}.
    All figures in this paper can be reproduced with the code available at \url{https://github.com/aqjaffe/RiemannianEB}, which also contains an easy-to-use Python package integrated with the \texttt{Geomstats} paradigm \cite{GeomStats}.
    
\subsection{Related Literature}

The primary context for our work is the large and growing literature on nonparametric empirical Bayes, which has been summarized in the recent tutorials \cite{IgnatiadisSen,Koenker_Gu_2026}.
In particular, we provide the first extension to the Riemannian setting of the principle that there exist fully data-driven denoisers that can be computed efficiently and achieve a risk that is nearly equal to the Bayes risk, even in the setting where the distribution of the latent variables is not assumed to have any particular parametric form; analogous results in the Euclidean setting have been established for Gaussian \cite{Zhang1997, JiangZhang, SahaGuntuboyina, Soloff} and Poisson \cite{JanaPolyanskiyWu, ShenWu} likelihoods.
Such empirical Bayes methods are often traced back to the work of Robbins \cite{Robbins1956}, although they are known to have a longer history in astronomy \cite{dyson1926method,eddington1940correction}; recent interest in nonparametric empirical Bayes (and much of our modern terminology) was popularized by Efron \cite{efron2011tweedie,Efron2014, efron2019bayes}.

There is also a large literature on parametric empirical Bayes methods, where, typically, the assumption of a conjugate prior allows one to directly approximate the Bayes denoiser by generalized method of moments.
In the Euclidean case, this was initiated by the empirical Bayes interpretation of the James-Stein shrinkage estimator due to Efron and Morris \cite{EfronMorris1973}, and the very many subsequent developments can be found in textbook treatments \cite{CarlinLouis, Maritz}.
In the Riemannian case, we are only aware of one previous work developing empirical Bayes methodology, namely \cite{YangDossVemuri} which provides an empirical Bayes interpretation for a particular shrinkage estimator on the manifold of positive semi-definite matrices.
There is also the closely-related is work \cite{SteinFrechet} which shows, in the setting of metric spaces with non-positive (Alexandrov) curvature, that a certain geodesic James-Stein estimator strictly dominates the Fr\'echet mean, although without utilizing any conjugate prior structure directly.

Another context for our work is the field of statistics on Riemannian manifolds, which includes theory and methodology in the general setting \cite{ManifoldsI, ManifoldsII,GeomStats} and in specific geometries of interest; of particular importance are the Riemannian manifolds studied in directional statistics, outlined in \cite{JuppMardia}.
The phenomena in this work that the singularity of a Riemannian Gaussian density at the cut locus of its Fr\'echet mean can deteriorate the rate of convergence of empirical Bayes denoising is closely related to the ``smeariness'' for Fr\'echet mean estimation (which has been established in spheres \cite{HotzHuckemann} and in other geometries, e.g., \cite{EltznerHuckemann,HundrieserEltznerHuckemann}).
There is some literature on Bayesian statistics on Riemannian manifolds \cite{BhattacharyaDunsonI, BhattacharyaDunsonII, RiemannianBayesianDTI}, which is typically focused on particular models tailored to each geometry of interest.
Overall, we emphasize that our assumptions on the geometry of $\Mcal$ are much more general than what is typically assumed in this field; we require only that $\Mcal$ is a compact connected Riemannian manifold, while other works occasionally require much further structure that can be limiting in some applications (e.g., non-positive curvature, constant curvature, homogeneity, Lie group structure).

Many aspects of this paper are closely related to existing literature on density estimation, deconvolution, and score matching on Riemannian manifolds.
Regarding density estimation, two popular methods are based on kernel smoothing using kernels that are made either from the geodesic distance \cite{Pelletier, HenryRodriguez} or the eigenfunctions of Laplace-Beltrami operator \cite{Hendriks1990}.
Regarding deconvolution, a few methods have been developed in particular geometries of interest, e.g., hyperspheres \cite{HealyKim}, the special orthogonal group \cite{Kim}, and the hyperbolic plane \cite{MobiusDeconvolution} (but all with likelihoods that are not Riemannian Gaussian); these works develop methods based on Fourier analysis (or generalization thereof), and they exhibit slow rates of convergence also seen in the Euclidean setting \cite{FanDeconvolution}.
There is also a recent interest in score-matching on Riemannian manifolds \cite{de2022riemannian} which studies Riemannian versions of the generative modeling procedures that have become popular in the Euclidean setting \cite{vincent2011connection, hyvarinen2005}.

\section{Geometric Background}
\label{subsec:background}

In this section we collect the basic geometric notions required throughout the paper.
We will provide references to specialized results along the way, but we direct the reader to standard sources like \cite{Aubin, Chavel,Petersen} for further details; for a particularly accessible introduction to these ideas, we recommend consulting \cite{SommerFletcherPennec,Boumal}.

It may be helpful for the reader to keep in mind a few concrete examples in order to build intuition throughout this section.
Towards this end, we will occasionally remark on examples like the usual Euclidean space $\Mcal=\Rbb^m$, the circle $\Mcal=\mathbb{S}^1$ (which may be parameterized either as $\Rbb/2\pi\Zbb$ or as $\{x\in\Rbb^2: \|x\|=1\}$), the two-dimensional sphere $\Mcal=\mathbb{S}^2:=\{x\in\Rbb^3: \|x\|=1\}$, and the space of $3\times 3$ special orthogonal matrices $\Mcal=\SO(3):=\{X\in\Rbb^{3\times 3}: XX^{\top} = X^{\top}X = I_3 \textnormal{ and } \det(X)=1\}$.
We also include in Figure~\ref{fig:background} a cartoon of a compact Riemannian manifold of dimension $m=2$ in order to visualize these notions.

\begin{figure}
    \begin{center}
        \begin{tikzpicture}[scale=1.3]
			\fill[smooth, color=gray!20] (2.35,0) to[out=90,in=210] (3,1) to[out=30,in=150] (5,1) to[out=-30,in=40] (5,-0.6) to[out=220,in=-20] (3,-0.85) to[out=160,in=-90] (2.35,0);
			\draw[thin, smooth] (2.35,0) to[out=90,in=210] (3,1) to[out=30,in=150] (5,1) to[out=-30,in=40] (5,-0.6) to[out=220,in=-20] (3,-0.85) to[out=160,in=-90] (2.35,0);
			\fill[smooth, color=white] (3,-0.31) .. controls (3.2,-0.1) and (3.8,-0.1) .. (4,-0.31) .. controls (3.8,-0.51) and (3.2,-0.51) .. (3.0,-0.31);
			\draw[smooth] (2.9,-0.2) .. controls (3.1,-0.55) and (3.9,-0.55) .. (4.1,-0.2);
			\draw[smooth] (3,-0.31) .. controls (3.2,-0.1) and (3.8,-0.1) .. (4,-0.31);
            \node[color=black] at (3, 1.45) {$\mathcal{M}$};
			
			\draw[very thick, smooth, color=gray,-stealth] (4.5,0.65) to[out=210, in=60] (2.7, 0.3) to[out=240,in=120] (2.7, -0.35);					
			
			\filldraw[thin, opacity=0.75, color=pythongreen!30, xscale=1, yscale=0.6,shift={(3,0.65)}] (0, 0) to (2, -1) to (3,1) to (1, 2) to cycle;
			\draw[thin, color=pythongreen, xscale=1, yscale=0.6,shift={(3,0.65)}] (0, 0) to (2, -1) to (3,1) to (1, 2) to cycle;
            \draw[thick, color=pythongreen,-stealth] (4.5, 0.65) -- ++(-0.7125, -0.3);
			\filldraw[pythonblue] (4.5, 0.65) circle (0.04) node[above right] {$x$};
		    \node[color=pythongreen] at (3.5, 0.424) {$v$};
			\node[color=pythongreen] at (5.5, 1.45) {$\textnormal{Tan}_x(\mathcal{M})$};
            \filldraw[pythonblue] (2.725, -0.375) circle (0.04);
            \draw[pythonblue] node at (3.35, -0.65) {$\Exp_x(v)$};
		\end{tikzpicture}
		\begin{tikzpicture}[scale=1.3]
			\fill[smooth, color=gray!20] (2.35,0) to[out=90,in=210] (3,1) to[out=30,in=150] (5,1) to[out=-30,in=40] (5,-0.6) to[out=220,in=-20] (3,-0.85) to[out=160,in=-90] (2.35,0);
			\draw[thin, smooth] (2.35,0) to[out=90,in=210] (3,1) to[out=30,in=150] (5,1) to[out=-30,in=40] (5,-0.6) to[out=220,in=-20] (3,-0.85) to[out=160,in=-90] (2.35,0);
			\fill[smooth, color=white] (3,0.-0.31) .. controls (3.2,-0.1) and (3.8,-0.1) .. (4,-0.31) .. controls (3.8,-0.51) and (3.2,-0.51) .. (3.0,-0.31);
			\draw[smooth] (2.9,-0.2) .. controls (3.1,-0.55) and (3.9,-0.55) .. (4.1,-0.2);
			\draw[smooth] (3,0.-0.31) .. controls (3.2,-0.1) and (3.8,-0.1) .. (4,-0.31);
			
			\draw[very thick, color=gray, domain=2.775:4, smooth,stealth-,cm={cos(-100) ,-sin(-100) ,sin(-100) ,cos(-100) ,(5.31,-3.26)}] plot (\x,{0.5*(\x-3.5)^2});
			\filldraw[pythonblue] (4.555, -0.6) circle (0.04) node[above left] {$y$};
			
			\filldraw[thin, opacity=0.75, color=pythongreen!30, xscale=1, yscale=0.6,shift={(3,0.65)}] (0, 0) to (2, -1) to (3,1) to (1, 2) to cycle;
			\draw[thin, color=pythongreen, xscale=1, yscale=0.6,shift={(3,0.65)}] (0, 0) to (2, -1) to (3,1) to (1, 2) to cycle;
            \draw[thick, color=pythongreen,-stealth] (4.5, 0.65) -- ++(0.225, -0.3375);
			\filldraw[pythonblue] (4.5, 0.65) circle (0.04) node[above right] {$x$};
			\node[color=pythongreen] at (4.05, 0.4) {$\textnormal{Log}_xy$};
		\end{tikzpicture}
        \begin{tikzpicture}[scale=1.3]
			\fill[smooth, color=gray!20] (2.35,0) to[out=90,in=210] (3,1) to[out=30,in=150] (5,1) to[out=-30,in=40] (5,-0.6) to[out=220,in=-20] (3,-0.85) to[out=160,in=-90] (2.35,0);
			\draw[thin, smooth] (2.35,0) to[out=90,in=210] (3,1) to[out=30,in=150] (5,1) to[out=-30,in=40] (5,-0.6) to[out=220,in=-20] (3,-0.85) to[out=160,in=-90] (2.35,0);

            \draw[very thick, color=gray, domain=2.75:4, smooth,-,cm={cos(-100) ,-sin(-100) ,sin(-100) ,cos(-100) ,(5.31,-3.26)}] plot (\x,{0.5*(\x-3.5)^2});
            \draw[very thick, smooth, color=gray,-stealth] (4.555,-0.6) to[out=-120, in=-20] (2.955, -0.555);
			\draw[very thick, smooth, color=gray,-stealth] (4.5,0.65) to[out=210, in=60] (2.7, 0.3) to[out=240,in=90] (2.6, -0.1) to[out=-90] (2.925, -0.525);
            \node[color=gray] at (2.9, 0.2) {$\gamma_1$};
            \node[color=gray] at (4., -0.6) {$\gamma_2$};
            
            \draw[very thick, smooth, color=pythonblue,-] (3.2, -0.4) to[out=180,in=120] (2.9, -0.8125);
            \node[color=pythonblue] at (2.3, -0.65) {$C_x$};

            \fill[smooth, color=white] (3,0.-0.31) .. controls (3.2,-0.1) and (3.8,-0.1) .. (4,-0.31) .. controls (3.8,-0.51) and (3.2,-0.51) .. (3.0,-0.31);
            \draw[smooth] (2.9,-0.2) .. controls (3.1,-0.55) and (3.9,-0.55) .. (4.1,-0.2);
			\draw[smooth] (3,0.-0.31) .. controls (3.2,-0.1) and (3.8,-0.1) .. (4,-0.31);
			
			\filldraw[thin, opacity=0.75, color=pythongreen!30, xscale=1, yscale=0.6,shift={(3,0.65)}] (0, 0) to (2, -1) to (3,1) to (1, 2) to cycle;
			\draw[thin, color=pythongreen, xscale=1, yscale=0.6,shift={(3,0.65)}] (0, 0) to (2, -1) to (3,1) to (1, 2) to cycle; 
            \draw[thick, color=pythongreen,-stealth] (4.5, 0.65) -- ++(-0.95, -0.4);
            \draw[thick, color=pythongreen,-stealth] (4.5, 0.65) -- ++(0.45, -0.675);
            
            \filldraw[pythonblue] (4.5, 0.65) circle (0.04) node[above right] {$x$};
		\end{tikzpicture}
	\end{center}
    \caption{Cartoon of some elements of the geometric background from Section~\ref{subsec:background}.
    For a compact Riemannian manifold $\Mcal$ of dimension $m=2$,
    we show the exponential map (left), the logarithmic map (center), and the cut locus (right).}
    \label{fig:background}
\end{figure}
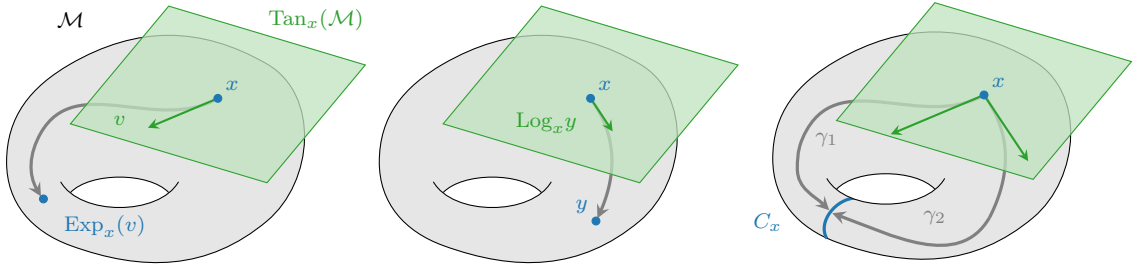

\medskip

\noindent \textbf{Basics.} Throughout, we assume that $\Mcal$ is a Riemannian manifold of dimension $m\in\Nbb=\{1,2,\ldots\}$.
This means each point $x\in\Mcal$ has a \textit{tangent space} $\Tan_x(\Mcal)$, which is an $m$-dimensional vector space representing the best linear approximation of $\Mcal$ near $x$, as well as an inner product $\langle\,\cdot\,,\,\cdot\,\rangle_x$ (hence also a norm $\|\,\cdot\,\|_x$) on $\Tan_x(\Mcal)$; we refer to $\{\Tan_x(\Mcal)\}_{x\in\Mcal}$ as the \textit{tangent bundle} of $\Mcal$ and $\{\langle\,\cdot\,,\,\cdot\,\rangle_x\}_{x\in\Mcal}$ as the \textit{metric tensor} of $\Mcal$.
There also exists a $\sigma$-finite measure on $\Mcal$ called the \textit{volume form} which is a generalization of the Lebesgue measure on $\Rbb^m$, and we write $\diff x$ for integration against the volume form (e.g., $\int_{\Mcal}h(x)\diff x$ for any bounded, measurable function $h:\Mcal\to\Rbb$).
When $\Mcal=\Rbb^m$ the tangent bundle is $\Tan_x(\Mcal) = \Rbb^m$ and the volume form is just the Lebesgue measure on $\Rbb^m$; when $\Mcal=\mathbb{S}^2$ the tangent bundle is $\Tan_x(\Mcal) = (\Span(x))^{\bot}$ and the volume form is the uniform measure on $\mathbb{S}^2$; when $\Mcal=\SO(3)$, the tangent bundle is $\Tan_X(\Mcal) =\{XA\in\Rbb^{3\times 3}:A=-A^{\top}\}$ is the space of $3\times 3$ antisymmetric matrices multiplied on the left by $X$, the metric tensor is $\langle V,W\rangle_X = -\frac{1}{2}\trace(X^{-1}VX^{-1}W)$, and the volume form is the Haar measure.

\medskip

\noindent \textbf{Metric.}  One may also view $\Mcal$ as a metric space $(\Mcal,d)$, where $d$ is a metric on $\Mcal$ called the \textit{geodesic distance} defined as $d(x,y) := \inf_{\gamma(0)=x,\gamma(1)=y}\mathcal{L}(\gamma)$ where the infimum is taken over all curves $\gamma:[0,1]\to\Mcal$ which are infinitely differentiable in the locally Euclidean structure of $\Mcal$ and where $\mathcal{L}(\gamma):=\int_{0}^{1}\|\gamma'(t)\|_{\gamma(t)}\diff t$; we always have $\gamma'(0)\in\Tan_x(\Mcal)$, and $\Tan_x(\Mcal)$ can in fact be defined as the closure of the span of all $\gamma'(0)$, i.e., tangent vectors at $x$ exactly correspond to velocities of curves starting at $x$.
Note that $\mathcal{L}(\gamma)$ is the \textit{length} of the curve $\gamma$, so that the geodesic distance $d(x,y)$ is just the length of the shortest smooth curve connecting $x$ to $y$.
A curve $\gamma:[0,1]\to \Mcal$ is called a \textit{constant-speed geodesic} if it satisfies $d(\gamma(s),\gamma(t)) = (t-s)d(\gamma(0),\gamma(1))$ for all $0\le s\le t \le 1$; if $\gamma$ is a constant-speed geodesic then it is infinitely differentiable and we have $\|\gamma'(t)\|_{\gamma(t)} = d(\gamma(0),\gamma(1))$, for all $t \in [0,1]$.
We say that $\Mcal$ is \textit{compact} if the metric space $(\Mcal,d)$ is compact, and that $\Mcal$ is \textit{connected} if $(\Mcal,d)$ is connected.
We define the \textit{diameter} of $\Mcal$ as $\diam(\Mcal):=\sup_{x,y\in\Mcal}d(x,y)$.

\medskip

\noindent \textbf{Exponential and Logarithm.} 
For each point $x\in\Mcal$, there is a natural correspondence between $\Mcal$ and its tangent bundle given as follows.
The \textit{exponential map} $\Exp_x:\Tan_x(\Mcal)\to\Mcal$ is defined via $\Exp_x(v):=\gamma(1)$ for sufficiently small $\|v\|_x$ where $\gamma$ is the unique constant-speed geodesic with $\gamma(0)=x$ and $\gamma'(0)=v$ (which is guaranteed to exist and be unique under suitable assumptions on $\Mcal$, e.g. compactness) and then extended to all $v\in\Tan_x(\Mcal)$; in other words, $\Exp_x(v)$ is the result of ``shooting'' from $x$ in direction $v/\|v\|_x$ for length $\|v\|_x.$ The \textit{logarithmic map} is defined as $\Log_x(y) = \gamma'(0)$ where $\gamma$ is the constant-speed geodesic with $\gamma(0)=x$ and $\gamma(1) = y$,  if it exists and is unique; consequently, we have $\|\Log_x(y)\|_x= d(x,y)$ whenever the left side is well-defined. Also, it is known that the exponential map is infinitely-differentiable on all of $\Tan_x(\Mcal)$; in fact the map $(x,v)\mapsto \Exp_x(v)$ is smooth (i.e., infinitely-differentiable) as a function of both arguments \cite[Chapter~5.5.1]{Petersen}.
In the Euclidean case $\Mcal=\Rbb^m$, we simply have $\Log_x(y) = y-x$ and $\Exp_x(v) = v+x$, so the logarithmic map in general provides a natural notion of displacement in $\Mcal$; in $\Mcal=\SO(3)$ we have $\Exp_X(V) = X\exp(X^{-1}V)$ where $\exp(\cdot)$ denotes the matrix exponential (defined via the usual power series) which explains the notation.
See Figure~\ref{fig:background} for an illustration of the logarithmic and exponential maps.

\medskip

\noindent \textbf{Cut Locus.} 
For $x\in\Mcal$, the \textit{cut locus} of $x$, denoted $C_x\subseteq\Mcal$, is (roughly speaking) given by the closure of the set of all $y\in\Mcal$ such that there is more than one constant-speed geodesic from $x$ to $y$; see \cite[Definition~1.3.9]{Aubin} for a precise definition.
Then, we may view the logarithmic map as a function $\Log_x:\Mcal\setminus C_x\to\Tan_x(\Mcal)$, and this is in fact its maximal domain of well-definition.
It is also known that the cut locus $C_x$ has measure zero under the volume form, for each $x\in\Mcal$.
If $\Mcal$ is compact, then the cut locus $C_x$ is non-empty for all $x\in\Mcal$.
For example, if $\Mcal=\mathbb{S}^m$ is a hypersphere then the cut locus of any point $x$ is exactly the antipode of $x$.
See Figure~\ref{fig:background} for an illustration of the cut locus, which need not be a single point.

\medskip

\noindent \textbf{Gradients.} If $h:\Mcal\to\Rbb$ is a sufficiently regular function, its gradient $\nabla h$ is a vector field on $\Mcal$, meaning it assigns a vector $\nabla h(x)\in\Tan_x(\Mcal)$ for each $x\in\Mcal$.
More precisely, we say that $h$ is differentiable if we have
\begin{equation}\label{eqn:Gradient}
    \frac{h(\gamma(t))-h(\gamma(0))}{t} \to \langle\nabla h(x),\gamma'(0)\rangle_x
\end{equation}
as $t\to 0$, for all sufficiently regular paths $\gamma$ with $\gamma(0)=x$.
For example, it is known that the squared-distance function has gradient
\begin{equation}\label{eqn:squared-dist-grad}
    \nabla_xd^2(x,y) = -2\Log_x(y)
\end{equation}
for all $x\in\Mcal\setminus C_y$; additionally, it is known that $x\mapsto d(x,y)$ is Lipschitz on all $\Mcal$, but it may fail to be differentiable on $C_y \cup \{y\}$.
Higher-order derivatives of $h$ are defined similarly, namely for $r\in\Nbb$ the derivative $\nabla^{r}h(x)$ is an $r$-linear function from $\Tan_x(\Mcal)$ to $\Rbb$, meaning it is a linear functional of $r$ many vectors in $\Tan_x(\Mcal)$; the space of all $r$-linear functions is a finite-dimensional vector space that can be endowed with the norm $\|H\|_{x,r}:=\max_{\|v_1\|_x=\cdots=\|v_r\|_x=1}|H(v_1,\ldots, v_r)|$ (or, alternatively, many other norms since they are all equivalent up to constants). 

\medskip

\noindent \textbf{Fr\'echet Means.} For a probability measure $P$ on a Riemannian manifold $\Mcal$, its \textit{Fr\'echet variance} is the infimal value of the functional $y\mapsto \int_{\Mcal}d^2(x,y)\diff P(x)$, and a \textit{Fr\'echet mean} is any minimizer, if it exists.
It is known \cite[Lemma~2.13]{EvansJaffeSLLN} that a Fr\'echet mean always exists if $\int_{\Mcal}d^2(x,y)\diff P(x)<\infty$ for some (equivalently, all) $y\in\Mcal$ (but it may be non-unique if $P$ is not sufficiently concentrated); in particular, on a compact and connected Riemannian manifold $\Mcal$, every probability measure admits a Fr\'echet mean.
If $P$ possesses a density with respect to the volume form, then the first-order optimality conditions for $y$ to be a Fr\'echet mean of $P$ are exactly $\int_{\Mcal}\Log_y(x)\diff P(x) =0$, but this condition is not sufficient in general.
If $\Mcal=\Rbb^m$, then the Fr\'echet mean of $P$ is unique and equals the expectation $\int_{\Rbb^m}x\diff P(x)$ of $P$, but in general there is no closed-form expression for Fr\'echet means.

\medskip

\noindent \textbf{Laplace-Beltrami Operator.}
We write $L^2(\Mcal)$ for the space of square-integrable (complex) functions with respect to the volume form of $\Mcal$, and write $L^{\infty}(\Mcal)$ for the space of essentially bounded functions with respect to the volume form of $\Mcal$.
There exists a unique nonnegative linear operator $\Delta:L^2(\Mcal)\to L^2(\Mcal)$ called the \textit{Laplace-Beltrami operator} which acts like a second-order differential operator when restricted to smooth functions; in other words, it is a generalization of the usual Laplace operator $\Delta = -(\partial^2/\partial x_1^2+\cdots +\partial^2/\partial x_m^2)$ from $\Rbb^m$, where, for smooth $f$, the value $\Delta f(x)$ represents the difference between the value of $f(x)$ compared to the average value of $f$ on infinitesimally small balls around $x$.
If $\Mcal$ is compact, then $L^2(\Mcal)$ admits a countable basis of eigenfunctions $\{\phi_j\}_{j=0}^{\infty}$ of $\Delta$ and corresponding eigenvalues $\{\lambda_j\}_{j=0}^{\infty}$, which are infinitely-differentiable functions.
The smallest eigenvalue is always $\lambda_0=0$ and corresponds to a (suitably normalized) constant function $\phi_0$.
As $j\to\infty$, it is known (see \cite[Section~VI.4]{Chavel}) that the eigenvalues satisfy Weyl's law $\lambda_j\sim 4\pi^2(\omega_{m}\vol(\Mcal))^{-2/m}j^{2/m}$ where $\omega_m$ denotes the Lebesgue measure of the unit ball in $\Rbb^m$, and $\vol(\Mcal):=\int_{\Mcal}1\diff x$ denotes the volume of $\Mcal$; in particular, the asymptotic behavior of the eigenvalues depends only on coarse geometric features of $\Mcal$, namely its dimension and volume.
The spectral theory of the Laplace-Beltrami operator provides a natural generalization of Fourier analysis to functions on compact Riemannian manifolds; for example, if $\Mcal=\mathbb{S}^1$ then $\{\phi_{j}\}_{j=0}^{\infty}$ are the complex exponentials and the resulting basis is the Fourier series, and if $\Mcal=\mathbb{S}^2$ then $\{\phi_{j}\}_{j=0}^{\infty}$ are the spherical harmonics.

\medskip

\noindent \textbf{Homogeneity.} An \textit{isometry} of a Riemannian manifold is a map $\psi:\Mcal\to\Mcal$ such that for all $x,y\in\Mcal$ we have $d(\psi(x),\psi(y)) = d(x,y)$, and the set of all isometries of $\Mcal$ is a group under composition denoted $\textnormal{Iso}(\Mcal)$ and called the \textit{isometry group}.
We say that a Riemannian manifold is \textit{homogeneous} if for all $x,x'\in\Mcal$ there exists $\psi\in \textnormal{Iso}(\Mcal)$ such that $\psi(x)= x'$ (i.e., the natural action of $\textnormal{Iso}(\Mcal)$ on $\Mcal$ is transitive).
None of the results in this paper require homogeneity; however, most examples in applications are indeed homogeneous (e.g., $\mathbb{S}^1, \mathbb{S}^1\times \mathbb{S}^1, \mathbb{S}^2, \SO(3)$) and many of our proofs could be slightly simplified under the additional assumption of homogeneity.

\medskip

\noindent \textbf{Lie Groups.} 
Many of the Riemannian manifolds in this paper can also be endowed with a group
structure so that the multiplication map $\Mcal\times\Mcal\to\Mcal$ and inversion
map $\Mcal\to\Mcal$ are smooth; examples include $\mathbb S^1$,
$\mathbb S^1\times\mathbb S^1$, and $\SO(3)$, but not $\mathbb S^2$.
When a Lie group is equipped with a bi-invariant Riemannian metric tensor, many of these geometric notions simplify; the
Riemannian volume form is proportional to Haar measure and is invariant under
left and right translations, the tangent space $\Tan_x(\Mcal)$ can be identified with the Lie algebra
$\Tan_e(\Mcal)$ by left or right translation where $e$ is the identity element of $\Mcal$, and the geodesic distance between $x,y\in\Mcal$ can be
written as a function of the group-theoretic difference $d(x,y)=d(e,x^{-1}y)=d(e,y^{-1}x)$.
In general there are many possible Riemannian metric tensors on a Lie group; whenever we use Lie-group structure in this paper, we assume that the group has been equipped with a fixed bi-invariant Riemannian metric tensor.
In particular, such a Lie group is homogeneous.

\section{Statistical Theory}\label{sec:theory}

In this section we state and prove the main results of the paper.
Subsection~\ref{subsec:prelim} contains preliminary results on Riemannian Gaussian mixture models (including our version of the Tweedie-Eddington formula), Subsection~\ref{subsec:oracle} provides a comparison of the risks of $\delta_{\N},\delta_{\B}$ and $\delta_{\T}$ at the oracle level, Subsection~\ref{subsec:sobolev} introduces some bounds on the Sobolev norms of densities of Riemannian Gaussian mixtures, and Subsection~\ref{subsec:empirical} provides a comparison of the risks of $\hat{\delta}_{\T}$ with $\delta_{\T}$ at the empirical level.

\subsection{Riemannian Gaussian Mixtures}\label{subsec:prelim}

We first give some introductory results related to Riemannian Gaussian distributions and mixtures of Riemannian Gaussian distributions.
The proofs of all results in this subsection appear in Appendix~\ref{subsec:proofs-prelim}, except that we give a sketch of proof of the Tweedie-Eddington formula (Lemma~\ref{lem:Tweedie}) since it is simple and instructive.
Throughout this section we will assume that $\Mcal$ is a compact connected Riemannian manifold.

We begin by defining a particular notion of Gaussian distributions on Riemannian manifolds, which is usually attributed to Pennec \cite{PennecRiemNorm}.
That is, for $\theta\in\Mcal$ and $\sigma^2>0$, the \textit{Riemannian Gaussian distribution with location parameter 
$\theta$ and scale parameter $\sigma$} is the probability measure $P_{\theta,\sigma^2}$ on $\Mcal$ whose density with respect to the volume form  $\diff x$ is 

\begin{equation}\label{eqn:ptheta}
    p_{\theta,\sigma^2}(x) := \exp\left(-\frac{d^2(x,\theta)}{2\sigma^2}-A(\sigma^2,\theta)\right) \qquad \textnormal{where}\qquad
    A(\sigma^2,\theta):=\log \int_{\Mcal}\exp\left(-\frac{d^2(x,\theta)}{2\sigma^2}\right)\diff x.
\end{equation}
As explained in \cite{PennecRiemNorm}, the Riemannian Gaussian distribution $P_{\theta,\sigma^2}$ is centered in the sense that $\theta$ is always a local minimizer (but not necessarily a global minimizer) of the Fr\'echet functional $y\mapsto \int_{\Mcal}d^2(x,y)\diff P_{\theta,\sigma^2}(x)$.
If $\Mcal$ is homogeneous, then (see~\cite[Theorem~2]{EquivariantFrechet}) the point $\theta$ is the unique global minimizer (i.e, the Fr\'echet mean) of $P_{\theta,\sigma^2}$, and $A$ does not depend on $\theta$. 
We also emphasize that it is not a restriction to assume that $\Mcal$ is connected; $P_{\theta,\sigma^2}$ assigns zero probability to all connected components of $\Mcal$ which do not contain $\theta$.


Now we recall the setting from the Introduction, where $\Theta$ is an $\Mcal$-valued random variable with distribution $G$, and the conditional distribution of $X$ given $\{\Theta=\theta\}$ is $P_{\theta,\sigma^2}$; in this setting, the marginal distribution of $X$ is a mixture of Riemannian Gaussian distributions, hence its density is $f(x) =\int_{\Mcal}p_{\theta,\sigma^2}(x)\diff G(\theta)$, for $x \in \Mcal$.
The starting point for this entire paper is the following version of a Tweedie-Eddington formula for Riemannian Gaussian mixture models.

\begin{lemma}\label{lem:Tweedie}
    If $\Mcal$ is a compact connected Riemannian manifold and the cut locus $C_x$ of $x\in\Mcal$ satisfies $G(C_x) = 0$, then $f$ is differentiable at $x\in\Mcal$ and we have \begin{equation}\label{eqn:tweedie}
        \Ebb[\Log_{x}\Theta\,|\,X=x] = \sigma^{2}\nabla_{x}\log f(x).
    \end{equation}
    Consequently, if $G$ has a density with respect to the volume form,  then $G(C_x) = 0$ for every $x \in \Mcal$, hence $f$ is differentiable, and \eqref{eqn:tweedie} holds for all $x\in\Mcal$.
\end{lemma}

\begin{proof}[Sketch of proof of Lemma~\ref{lem:Tweedie}]
    For any constant-speed geodesic $\gamma:[0,1]\to\Mcal$ with $\gamma(0) = x$ and $\gamma'(0) = v$, differentiate under the integral and use the chain rule to obtain (see~\eqref{eqn:Gradient} and~\eqref{eqn:squared-dist-grad})
    \begin{align*}
        \lim_{t\to 0}\frac{f(\gamma(t))-f(x)}{t} &= \lim_{t\to 0}\int_{\Mcal}\frac{p_{\theta,\sigma^2}(\gamma(t))-p_{\theta,\sigma^2}(x)}{t}\,\diff G(\theta) \\
        &= \int_{\Mcal}
        \left\langle
            \sigma^{-2}\Log_x(\theta)\,p_{\theta,\sigma^2}(x),\,v
        \right\rangle_x
        \diff G(\theta) \\
        &= \left\langle
            \sigma^{-2}\int_{\Mcal}\Log_x(\theta)\,p_{\theta,\sigma^2}(x)\,\diff G(\theta),\,v
        \right\rangle_x.
    \end{align*}
    As this holds for all $v\in\Tan_x(\Mcal)$, we see that $f$ is differentiable at $x$, with gradient equal to
    \[
        \nabla f(x)=\sigma^{-2}\int_{\Mcal}\Log_x(\theta)\,p_{\theta,\sigma^2}(x)\,\diff G(\theta).
    \]
    Multiplying this expression by $\sigma^2/f(x)$ gives
    \[
        \sigma^2\nabla_x\log f(x)
        =
        \frac{\int_{\Mcal}\Log_x(\theta)\,p_{\theta,\sigma^2}(x)\,\diff G(\theta)}{f(x)},
    \]
    and we recognize on the right side that $q_x(\theta):=p_{\theta,\sigma^2}(x)/f(x)$ is the density, with respect to $G$, of the conditional distribution of $\Theta$ given $\{X=x\}$; this finishes the proof.
\end{proof}

Although $\Mcal=\Rbb^m$ is non-compact and hence does not satisfy our hypotheses, it is useful to note that the conclusion of Lemma~\ref{lem:Tweedie} in this Euclidean setting coincides with the usual Tweedie-Eddington formula
\begin{equation*}
    \Ebb[\Theta\,|\,X=x] = x + \sigma^2\nabla_x \log f(x),
\end{equation*}
since we have $\Log_x(\theta) = \theta-x$.

One key difference between Riemannian Gaussian mixture densities and Euclidean Gaussian mixture densities is that the former can be uniformly lower bounded; we show this in the following result, where ess inf and ess sup denote the essential infimum and essential supremum, with respect to the volume form of $\Mcal$, respectively.

\begin{lemma}\label{lem:f-lower-bd}
    If $\Mcal$ is a compact connected Riemannian manifold and $G$ has a density $g$ with respect to the volume form, then we have 
    \begin{equation*}
        \inf_{x\in\Mcal}f(x)\ge c_{\Mcal}\,\underset{\theta\in\Mcal}{\textnormal{ess}\inf}\,g(\theta)\qquad \textnormal{ and }\qquad\sup_{x\in\Mcal}f(x)\le C_{\Mcal} \,\underset{\theta\in\Mcal}{\textnormal{ess}\sup}\,g(\theta)
    \end{equation*}
    for all $\sigma^2\le(\diam(\Mcal))^2$, where $c_{\Mcal},C_{\Mcal}>0$ are constants depending only on $\Mcal$.
\end{lemma}

We will also need the following uniform bound on the gradient of a Riemannian mixture density.

\begin{lemma}\label{lem:f-diff-1}
    If $\Mcal$ is a compact connected Riemannian manifold and $G$ has a density $g$ with respect to the volume form, we have
    \begin{equation*}
        \sup_{x\in\Mcal}\|\nabla f(x)\|_x \le C_{\Mcal}\sigma^{-1}\|g\|_{L^{\infty}(\Mcal)}
    \end{equation*}
    for all $\sigma^2\le(\diam(\Mcal))^2$, where $C_\Mcal>0$ depends only on $\Mcal$.
\end{lemma}

\subsection{Oracle-Level Denoising}\label{subsec:oracle}

In this subsection we analyze the risks of the oracle-level denoisers $\delta_{\N},\delta_{\B}$, and $\delta_{\T}$, in the asymptotic regime where $\sigma^2\to 0$, for any compact connected Riemannian manifold.
Our main results show (Proposition~\ref{prop:naive-risk}) that the risk of $\delta_{\N}$ is asymptotically equivalent to $m\sigma^2$, and (Theorem~\ref{thm:uncond-main}) that the risks of $\delta_{\B}$ and $\delta_{\T}$ differ by at most $O(\sigma^5)$.
The proofs of all results (and some auxiliary results needed along the way) are provided in Appendix~\ref{subsec:proofs-oracle}.

To begin, we consider the na\"ive denoiser $\delta_{\N}(x) = x$ which is implicitly used in works that do not use techniques from empirical Bayes.
In the general setting, we can describe the risk of $\delta_{\N}$ as follows.

\begin{proposition}\label{prop:naive-risk}
    If $\Mcal$ is a compact connected Riemannian manifold, then $R(\delta_{\N},\sigma^2) \sim m\sigma^2$ as $\sigma^2\to 0$.
\end{proposition}

The preceding result shows that the risk of $\delta_{\N}$ is equivalent to its risk in the Euclidean setting $\Rbb^m$; intuitively speaking, this comes from the fact that, for small $\sigma^2>0$, a Riemannian Gaussian distribution is well-approximated by a Euclidean Gaussian distribution in the tangent space.

Next we introduce the oracle-level denoisers of interest whose analysis will occupy the remainder of this subsection.
That is, we define the \textit{oracle Bayes denoiser} as
\begin{equation}\label{eqn:def-deltaB}
    \delta_{\B}(x)\in \underset{\vartheta\in\Mcal}{\arg\min}\,\Ebb[d^2(\vartheta,\Theta)\,|\,X=x]
\end{equation}
and the \textit{oracle tangential Bayes denoiser} as
\begin{equation}\label{eqn:def-deltaT}
    \delta_{\T}(x):=\Exp_x\left(\Ebb[\Log_x\Theta\,|\,X=x]\right)
\end{equation}
for all $x\in\Mcal$.
As we explained in the Introduction, one can view $\delta_{\T}(x)$ as a one-step approximation of $\delta_{\B}(x)$, by initializing Riemannian gradient descent at the output of the na\"ive denoiser $\delta_{\N}(x)= x$; also, by Lemma~\ref{lem:Tweedie}, the oracle tangential Bayes denoiser can be written as $\delta_{\T}(x) =\Exp_x(\sigma^2\nabla_x\log f(x))$, which depends only on the marginal density of $X$ and hence can be approximated in a fully data-driven manner.

\begin{example}[homogeneous $\Mcal$, uniform $G$]\label{ex:oracle-uniform}
    If $\Mcal$ is a compact connected homogeneous Riemannian manifold and $G$ is proportional to the volume form (i.e., the uniform distribution), then the denoisers $\delta_{\N},\delta_{\B}$, and $\delta_{\T}$ are all equal.
    To see $\delta_{\B} = \delta_{\N}$, note that the conditional distribution of $\Theta$ given $\{X=x\}$ is a Riemannian Gaussian distribution with location parameter $x$ and scale parameter $\sigma$, hence $\delta_{\B}(x) = x$ by definition.
    To see $\delta_{\T} = \delta_{\N}$, note that the marginal distribution of $X$ is uniform, hence $\delta_{\T}(x) = \Exp_x(\sigma^2\cdot 0) = x$.
    This behavior is distinct from the Euclidean case, where the na\"ive denoiser cannot be written as the Bayes denoiser with respect to any $G$.
    In fact, this argument shows that $\delta_{\N}$, being the minimizer of a posterior expected loss, is admissible; this somewhat contrasts the recent work \cite{SteinFrechet} on the James-Stein estimator in spaces of non-positive curvature.
\end{example}

Next we turn to a comparison of the risks of $\delta_{\T}$ and $\delta_{\B}$, in the regime where $\sigma^2\to 0$.
Since $\delta_{\T}$ directly involves the random variable $\Ebb[\Log_X\Theta\,|\,X]$, we must first develop some auxiliary results for conditional distributions of this random variable and some related quantities.

\begin{proposition}\label{prop:mux-sigmax-estimates}
Suppose that $\Mcal$ is a compact connected Riemannian manifold and that $G$ has a strictly positive, continuously differentiable density with respect to the volume form. Then we have
\begin{equation*}
\sup_{x\in\Mcal} \big\|\Ebb\left[\Log_x\Theta\,|\,X=x\right]\big\|_x = O(\sigma^2)\qquad \textnormal{ and }\qquad \sup_{x\in\Mcal}\big\|\Cov\left(\Log_x\Theta\,|\,X=x\right)\big\|_{x,2} = O(\sigma^2)
\end{equation*}
as $\sigma^2\to 0$.
\end{proposition}

Although $\Mcal=\Rbb$ is non-compact and hence does not satisfy our hypotheses, it is useful to note that we can do explicit calculations in this Euclidean setting in order to confirm that the rates in Proposition~\ref{prop:mux-sigmax-estimates} are correct.
That is, suppose $\Mcal= \Rbb$ and $G=\Ncal(0,\tau^2)$, so that the conjugate prior structure allows us to directly calculate 
\begin{equation*}
    \Ebb[\Theta-x\,|\,X=x] = -\frac{\sigma^2}{\tau^2 + \sigma^2}x = O(\sigma^2)\qquad\textnormal{and}\qquad\Cov(\Theta-x\,|\,X=x) = \frac{\tau^2\sigma^2}{\tau^2 + \sigma^2} = O(\sigma^2).
\end{equation*}
    If we restrict $\Mcal$ to be a bounded subset of $\Rbb$, then these bounds also hold uniformly over $x$.
    
In the previous result we established that $\Ebb[\Log_X\Theta\,|\,X]$ is of order $O(\sigma^2)$; the next result shows that $\Log_X(\delta_{\B}(X))$ is also of order $O(\sigma^2)$, but that the difference $\Ebb[\Log_X\Theta\,|\,X]-\Log_X(\delta_{\B}(X))$ is of the smaller order $O(\sigma^4)$.

\begin{proposition}\label{prop:wB-estimates}
Suppose that $\Mcal$ is a compact connected Riemannian manifold and that $G$ has a strictly positive, continuously differentiable density with respect to the volume form. Then we have
\begin{equation*}
    \sup_{x\in\Mcal} \big\|\Log_x(\delta_{\B}(x)) \big\|_x = O(\sigma^2) \qquad \textnormal{ and }\qquad \sup_{x\in\Mcal}\big\|\Ebb[\Log_x\Theta\,|\,X=x]-\Log_x(\delta_{\B}(x)) \big\|_x = O(\sigma^4)
\end{equation*}
as $\sigma^2\to 0$.
\end{proposition}

    While Proposition~\ref{prop:mux-sigmax-estimates} is a relatively straightforward calculation using normal coordinates (see~\eqref{eq:normal-coor} in Appendix~\ref{subsec:proofs-oracle}), Proposition~\ref{prop:wB-estimates} is more difficult because $\delta_{\B}(x)$, for $x\in\Mcal$, is only implicitly defined via \eqref{eqn:def-deltaB}.
    Nonetheless, the proof proceeds by observing that the first-order optimality conditions for $\delta_{\B}(x)$ provide a recursive bound on the distance between $\Log_x(\delta_{\B}(x))$ and $\Ebb[\Log_x\Theta\,|\,X=x]$, which upon iterating and combining with Proposition~\ref{prop:mux-sigmax-estimates} implies the desired estimates.

Now we combine Proposition~\ref{prop:mux-sigmax-estimates} and Proposition~\ref{prop:wB-estimates} to prove the following, which is our main result comparing the risks of $\delta_{\B}$ and $\delta_{\T}$.
The proof (given in Appendix~\ref{subsec:proofs-oracle}) reveals that the Taylor expansions of the squared geodesic distances $d^2(\theta,\delta_{\B}(x))$ and $d^2(\theta,\delta_{\T}(x))$ agree up to second-order, i.e., the denoiser $\delta_{\T}$ captures precisely the impact of the sectional curvature of $\Mcal$ on $\delta_{\B}$.

\begin{theorem}
\label{thm:uncond-main}
Suppose that $\Mcal$ is a compact connected Riemannian manifold and that $G$ has a strictly positive, continuously differentiable density with respect to the volume form. Then we have $0 \le R(\delta_{\T},\sigma^2) - R(\delta_{\B},\sigma^2) = O(\sigma^5)$ as $\sigma^2\to 0$.
\end{theorem}

\subsection{Sobolev Norm Bounds}\label{subsec:sobolev}

In the Euclidean setting, the key to the fast rates of convergence of empirical Bayes denoising is that the density $f$ of a Gaussian mixture is always infinitely differentiable (in fact, analytic), regardless of $G$.
In the Riemannian setting, it turns out that smoothness of the density $f$ of a Riemannian Gaussian mixture is more delicate. In this subsection, we will develop some sufficient conditions on the geometry of $\Mcal$ and the density $g$ of $G$ which imply a bound on a suitable Sobolev norm of $f$.
The proofs of all results in this subsection are contained in Appendix~\ref{subsec:proofs-sobolev}.

To state the desired estimates precisely, let $\{\phi_{j}\}_{j=0}^{\infty}$ and $\{\lambda_{j}\}_{j=0}^{\infty}$ denote the eigenfunctions and eigenvalues of the Laplace-Beltrami operator $\Delta$ on $\Mcal$.
For any function $h:\Mcal\to\Cbb$, we define $h^{\ast}(j):=\int_{\Mcal}h(x)\overline{\phi_j(x)}\diff x$ for all $j=0,1,\ldots$, which are just its coefficients in the basis $\{\phi_j\}_{j=0}^{\infty}$, i.e., we have $h = \sum_{j=0}^{\infty}h^{\ast}(j)\phi_j$ with the sum convergent in $L^2(\Mcal)$.
For $r\ge 0$, we also define
\begin{equation}\label{eqn:def-Sobolev}
    \|h\|_{\Hcal^r(\Mcal)}^2:= \sum_{j=0}^{\infty}(1 + \lambda_{j})^r|h^{\ast}(j)|^2,
\end{equation}
called the \textit{Sobolev norm of order $r$} of $h$.

To understand these Sobolev norms, it is helpful to consider some special cases.
For example, for $r=0$ we see from Parseval's identity that $\|h\|_{\Hcal^0(\Mcal)} = \|h\|_{L^2(\Mcal)}$.
A more interesting example is $r=1$, where integration by parts allows us to compute
\begin{align*}
    \int_{\Mcal}\|\nabla h(x)\|_x^2\diff x &= \int_{\Mcal}\overline{h(x)}\Delta h(x)\diff x \\
    &= \int_{\Mcal} \left(\sum_{j=0}^{\infty}\overline{h^{\ast}(j)}\,\overline{\phi_j(x)}\right)\left(\sum_{j=0}^{\infty}h^{\ast}(j)\lambda_j\phi_j(x)\right)\diff x = \sum_{j=0}^{\infty}\lambda_j |h^{\ast}(j)|^2;
\end{align*}
hence 
\begin{equation}\label{eqn:Sobolev-1-equiv}
    \|h\|^2_{\Hcal^1(\Mcal)} = \|h\|_{L^2(\Mcal)}^2+\int_{\Mcal}\|\nabla h(x)\|^2_x\diff x.
\end{equation}
More generally, for integer $r$, the spectral norm in~\eqref{eqn:def-Sobolev}
is equivalent to the usual Sobolev norm involving weak derivatives up to order
$r$; in particular, $\|h\|_{\Hcal^r(\Mcal)}<\infty$
if and only if all weak derivatives of $h$ up to order $r$ belong to the
corresponding tensor-valued $L^2$ spaces.
The definition in~\eqref{eqn:def-Sobolev}, however, makes sense for all
$r\ge0$, including non-integer values; see \cite[Chapter~2]{Aubin}.

Our empirical Bayes results will require the assumption that we have finite Sobolev norm $\|f\|_{\Hcal^r(\Mcal)}<\infty$ for some $r> 1$, and the resulting rates of convergence will depend explicitly on $\|f\|_{\Hcal^r(\Mcal)}$ and $r$; we will see that large values of $r$ are desirable, since the resulting rates approach the parametric rate of convergence $n^{-1}$ as $r\to\infty$.
To understand the assumption $\|f\|_{\Hcal^r(\Mcal)}<\infty$ for some $r> 1$, we give the following result (which is an immediate consequence of Lemma~\ref{lem:f-diff-1}) that states that this condition is nearly satisfied in general.

\begin{lemma}\label{lem:general-Sobolev}
    If $\Mcal$ is a compact connected Riemannian manifold and $G$ has a density $g\in L^\infty(\Mcal)$ with respect to the volume form, then we have $\|f\|_{\Hcal^1(\Mcal)}\le C_{\Mcal}\sigma^{-1}\|g\|_{L^{\infty}(\Mcal)}$, for all $\sigma^2\le(\diam(\Mcal))^2$, where $C_\Mcal>0$ is a constant depending only on $\Mcal$.
\end{lemma}

In the Riemannian setting, we will require further assumptions on $\Mcal$ and $g$ in order to upgrade the finite Sobolev norm of order 1 into a finite Sobolev norm of order $r>1$.
This is a distinct behavior from the Euclidean setting, where Gaussian location mixture models always have finite Sobolev norms of all orders.
For instance, we give the following result in the Euclidean setting which implies that $f$ has finite Sobolev norms of all orders, although where the Sobolev space $\Hcal^r(\Rbb^m)$ needs to be defined slightly differently since $\Mcal=\Rbb^m$ is non-compact (and its Laplace-Beltrami operator does not have a discrete spectrum).

\begin{lemma}\label{lem:Sobolev-est-Euclidean}
    If $\Mcal=\Rbb^m$ and $G$ has a density $g\in L^{\infty}(\Rbb^m)$ with respect to Lebesgue measure, then for all $r\in\Nbb$ we have
    \begin{equation*}
        \sup_{x\in\Rbb^m}\|\nabla^r f(x)\|_{r} \le C_{m,r}\sigma^{-r}\|g\|_{L^{\infty}(\Rbb^m)}
    \end{equation*}
    for all $\sigma^2>0$ where $C_{m,r}>0$ is a constant that depends only on $m$ and $r$. 
\end{lemma}

The crucial aspect of Euclidean spaces $\Mcal=\Rbb^m$ which permits a straightforward argument for the preceding result is that the functional $x\mapsto \|x-\theta\|^2$ has uniformly bounded $r$th order derivatives for all $r\ge 2$.
In the Riemannian setting such bounds can fail dramatically, because the $r$th order derivatives of the functional $x\mapsto d^2(x,\theta)$ for $r\ge 2$ may explode as $x$ approaches the cut locus $C_{\theta}$ of $\theta$.
We illustrate this with the following basic calculation.

\begin{example}[Hessian of distance squared on $\mathbb{S}^2$]\label{ex:sphere}
    Fix $\theta\in\mathbb{S}^2$ and $\sigma^2=1$ and  use \cite{PennecHessian} to calculate $\nabla_x^2d^2(x,\theta)$, (where we emphasize the differentiation with respect to the $x$ variable); to write the expression explicitly, we represent $\mathbb{S}^2 = \{x\in\Rbb^3:\|x\|=1\}$ and we have
    \begin{equation}\label{eqn:sphere-Hessian-sq-dist}
        \frac{1}{2}\nabla_x^2d^2(x,\theta) = \frac{\Log_x\theta (\Log_x\theta)^{\top}}{d^2(x,\theta)} + h(d(x,\theta))\left(I-xx^{\top}-\frac{\Log_x\theta (\Log_x\theta)^{\top}}{d^2(x,\theta)}\right)
    \end{equation}
    where $h(t) = t\cos(t)/\sin(t)$ for $0< t\le \pi$ with $h(0):=1$ by continuity.
    Observe that, as $x\to C_{\theta} = \{-\theta\}$, we have $h(d(x,\theta))\to -\infty$ and hence $\|\nabla_x^2d^2(x,\theta)\|_{x,2}\to \infty$.
\end{example}

Consequently, in the Riemannian setting, the Sobolev smoothness of $f$ relies more delicately on both the geometry of $\Mcal$ and the properties of $g$; in the rest of this subsection we give some results of this form.

For instance, most of the examples of interest in this work are Lie groups (e.g., $\mathbb{S}^1,\mathbb{S}^1\times\mathbb{S}^1$, and $\SO(3)$), where the convolution structure allows us to differentiate $f$ without directly differentiating the Riemannian logarithm.
This allows us to show the following, namely that $f$ inherits the same smoothness as $g$; we write $C^r(\Mcal)$ for the set of real functions on $\Mcal$ which are $r$-times continuously differentiable.

\begin{lemma}\label{lem:Sobolev-Lie}
    If $\Mcal$ is a compact connected Lie group and $G$ possesses a density $g\in C^r(\Mcal)$ with respect to the volume form, then we have $\|f\|_{\Hcal^r(\Mcal)} \le C_{\Mcal,r}(\|g\|_{L^{\infty}(\Mcal)}+\sup_{\theta\in\Mcal}\|\nabla^rg(\theta)\|_{\theta,r})$ for all $\sigma^2>0$, where $C_{\Mcal,r}>0$ is a constant depending only on $\Mcal$ and $r$.
\end{lemma}

The sphere $\mathbb{S}^2$ is not a Lie group so the previous result does not apply, but we can nonetheless use the explicit calculations from Example~\ref{ex:sphere} to say something general in the following.

\begin{lemma}\label{lem:sphere-estimate}
    If $\Mcal=\mathbb{S}^2$ and $G$ possesses a density $g\in L^\infty(\mathbb S^2)$ with respect to the volume form, then $\|f\|_{\Hcal^2(\Mcal)}\le C_{\Mcal} \sigma^{-2}\|g\|_{L^{\infty}(\Mcal)}$ for all $\sigma^2\le (\diam(\Mcal))^2$, where $C_\Mcal>0$ is a constant depending only on $\Mcal$. 
\end{lemma}

Note the qualitative difference between Lemma~\ref{lem:Sobolev-Lie} and Lemma~\ref{lem:sphere-estimate}.
In Lemma~\ref{lem:Sobolev-Lie} we show that $f$ is at least as smooth as $g$, while in Lemma~\ref{lem:sphere-estimate} we show that $f$ gains some smoothness over that of $g$.
We do not expect either of these results to be sharp, as we show next in the special case of the circle $\Mcal=\mathbb{S}^1$.
In this special case, we show that $f$ gains two derivatives over that of $g$, which can be understood via standard Fourier analysis since the density $f$ of the Riemannian Gaussian mixture is exactly a convolution of $g$ with a kernel that is differentiable but not continuously differentiable (because of its behavior at the cut locus).

\begin{lemma}\label{lem:Sobolev-bd-S1}
    If $\Mcal=\mathbb{S}^1$ and $G$ possesses a density $g\in\Hcal^s(\mathbb{S}^1)$ with respect to the volume form, then we have $c_{\sigma^2}\|g\|_{\Hcal^{s}(\mathbb{S}^1)}\le \|f\|_{\Hcal^{s+2}(\mathbb{S}^1)}\le C \sigma^{-2}\|g\|_{\Hcal^{s}(\mathbb{S}^1)}$ for all $s\ge 0$, where $C>0$ is a universal constant and $c_{\sigma^2}>0$ is a constant depending only on $\sigma^2$.
\end{lemma}

Another special case where we can make our results precise is the following.

\begin{lemma}\label{lem:ex-uniform}
    If $\Mcal$ is a compact connected homogeneous Riemannian manifold and $G$ is proportional to the volume form (i.e., $G$ is the uniform distribution), then $\|f\|_{\Hcal^r(\Mcal)}=(\vol(\Mcal))^{-1/2}$ for all $r\ge 1$.
\end{lemma}

As the preceding examples show, obtaining sharp Sobolev regularity bounds for the marginal density $f$ in terms of the Sobolev regularity of the mixing density $g$ is a difficult problem which depends on the geometry of $\Mcal$.
One obstruction is the behavior of higher-order derivatives of the squared-distance function $x\mapsto d^2(x,\theta)$ near the cut locus $C_\theta$; already at second order, the Hessian $\nabla_x^2 d^2(x,\theta)$ can exhibit singular behavior near $C_\theta$ (cf.~Example~\ref{ex:sphere}).
Furthermore, the integrability of these singularities also depends on the (co)dimension and local structure of the cut loci $C_\theta$ themselves (cf.~the proof of Lemma~\ref{lem:sphere-estimate}).
In the case of compact Lie groups, the convolution structure allows us to avoid differentiating the Riemannian logarithm directly and gives partial results in this direction (cf.~Lemma~\ref{lem:Sobolev-Lie}), but a sharp analysis appears to require further tools from representation theory.
A detailed study of Sobolev bounds for densities of Riemannian Gaussian mixtures is therefore an interesting direction for future work, and lies beyond the scope of this paper.

\subsection{Empirical-Level Denoising}\label{subsec:empirical}

In this final subsection we establish (Theorem~\ref{thm:EB-risk}) a rate of convergence for the empirical tangential Bayes denoiser $\hat{\delta}_{\T}$ to the oracle tangential Bayes denoiser $\delta_{\T}$.
The proofs of all results in this section are given in Appendix~\ref{subsec:proofs-empirical}.
Combined with the results of Subsection~\ref{subsec:oracle}, this shows that the empirical tangential Bayes denoiser $\hat{\delta}_{\T}$ can achieve nearly the same risk as the oracle Bayes denoiser $\delta_{\B}$, when $\sigma^2$ is small and $n$ is large.

To begin, let $\{\lambda_j\}_{j=0}^{\infty}$ and $\{\phi_j\}_{j=0}^{\infty}$ denote the eigenvalues and eigenfunctions of the Laplace-Beltrami operator $\Delta$ on $L^2(\Mcal)$, and define
    \begin{equation*}
K(x,y,\ell):=\sum_{j=0}^{ \ell}\phi_j(x)\overline{\phi_j(y)},
    \end{equation*}
for $x,y \in \Mcal$, which is just a finite-rank projection of the identity operator, with respect to the basis $\{\phi_j\}_{j=0}^{\infty}$ of $L^2(\Mcal)$.
The function $K$ plays the role of the kernel in our methodology, in the sense that we define our estimator of the density $f$ of $X$ to be
\begin{equation}\label{eq:Density-Est}
        \hat f_n(x):=\frac{1}{n}\sum_{i=1}^{n}K(x,X_i,\ell_n),
    \end{equation}
    where $\ell_n$ is some sequence that will be determined later.

    Before we use this density estimator in our empirical Bayes approximation of $\delta_{\T}$, we give the following result which provides simultaneous control on the $L^2(\Mcal)$ estimation error of $\hat f_n$ and its gradient $\nabla_x \hat{f}_n$.

\begin{proposition}\label{prop:kde-error}
    Suppose that $\Mcal$ is a compact connected Riemannian manifold, that $G$ has a density $g\in L^{\infty}(\Mcal)$ with respect to the volume form, and that we have $\|f\|_{\Hcal^r(\Mcal)}<\infty$ for some $r>0$.
    Then setting
    \begin{equation*}
        \ell_n :=\left\lceil \left(\frac{n\|f\|^2_{\Hcal^r(\Mcal)}}{\|g\|_{L^{\infty}(\Mcal)}}\right)^{\frac{m}{2r+m}}\right\rceil
    \end{equation*}
    yields
    \begin{equation*}
        \Ebb\left[\int_{\Mcal}\left|\hat{f}_n(x)-f(x)\right|^2\diff x\right] \le \frac{C_{\Mcal,r}\|f\|_{\Hcal^{r}(\Mcal)}^{\frac{2m}{2r+m}}\|g\|_{L^{\infty}(\Mcal)}^{\frac{2r}{2r+m}}}{n^{\frac{2r}{2r+m}}}
    \end{equation*}
    and, for $r >1$,
    \begin{equation*}
        \Ebb\left[\int_{\Mcal}\left\|\nabla_x\hat{f}_n(x)-\nabla_x f(x)\right\|_x^2\diff x\right] \le \frac{C_{\Mcal,r}\|f\|_{\Hcal^{r}(\Mcal)}^{\frac{2m+4}{2r+m}}\|g\|_{L^{\infty}(\Mcal)}^{\frac{2r-2}{2r+m}}}{n^{\frac{2r-2}{2r+m}}},
    \end{equation*}
    for all $n$ and $\sigma^2\le (\diam(\Mcal))^2$, where $C_{\Mcal,r}>0$ is a constant depending only on $\Mcal$ and $r$.
\end{proposition}

    The work \cite{Hendriks1990} initiated the study of estimating a nonparametric density $f$ from a sample of i.i.d. random variables $X_1,\ldots, X_n$ on a compact Riemannian manifold $\Mcal$, using a kernel made from the eigenfunctions of the Laplace-Beltrami operator;
    our Proposition~\ref{prop:kde-error} is limited to the setting that $f$ is the density of a Riemannian Gaussian mixture, but our results are more precise in that the rate of convergence is more explicit and that we also establish a rate of convergence for the estimation of the gradient $\nabla f$.
    These minor extensions are needed for the analysis of our empirical Bayes denoiser in Subsection~\ref{subsec:empirical}, but the proofs are conceptually similar to the methods used in \cite{Hendriks1990}; see also the discussion in \cite[Section~4.6]{Hendriks1990}.
    Nonetheless, they may be of independent interest in the literature on density estimation on Riemannian manifolds.
    Also, note that this result provides a non-trivial rate of convergence for estimating $\nabla f$ only if $r>1$; this is precisely what requires us to develop the finer results of Subsection~\ref{subsec:sobolev}.

Now we define the denoiser of interest, namely
\begin{equation}\label{eq:Est-Denoiser}
    \hat{\delta}_{\T}(x):= \Exp_{x}\left(\sigma^2\frac{\nabla_x \hat f_n(x)}{\max\{\hat f_n(x),\rho\}}\right),
\end{equation}
where $\hat{f}_n$ is chosen with $\ell_n$ according to the previous result and $\rho>0$ will be chosen later.
By comparing with~\eqref{eq:Tan-Bayes}, we see that this is like a plug-in form of $\delta_{\T}$, unless $\hat{f}_n(x)$ provides a very low estimate for $f(x)$ in which case we regularize by replacing $\hat{f}_n(x)$ with $\rho$.

Our main result on the empirical Bayes denoiser $\hat{\delta}_{\T}$ is the following, which provides an explicit rate of convergence for the expected distance $\Ebb[D_{\T}(\hat{\delta}_{\T})]$ defined in \eqref{eqn:def-dist}.

\begin{theorem}\label{thm:EB-risk}
    Suppose that $\Mcal$ is a compact connected Riemannian manifold, that $G$ has a density $g\in L^\infty(\Mcal)$ with respect to the volume form, and that we have $\|f\|_{\Hcal^r(\Mcal)}<\infty$ for some $r>1$.
    Then setting $0<\rho\le\min_{x\in\Mcal}f(x)$ and
    \begin{equation*}
        \ell_n := \left\lceil\left(\frac{n\|f\|^2_{\Hcal^r(\Mcal)}}{\|g\|_{L^{\infty}(\Mcal)}}\right)^{\frac{m}{2r+m}}\right\rceil
    \end{equation*}
    yields
    \begin{equation*}
        \Ebb\left[D_{\T}(\hat{\delta}_{\T})\right]=\Ebb\left[\int_{\Mcal}d^2\left(\hat{\delta}_{\T}(x),\delta_{\T}(x)\right)f(x)\diff x\right] \le\frac{C_{\Mcal,r}\sigma^{2}\|f\|_{\Hcal^r(\Mcal)}^{\frac{2m+4}{2r+m}}\|g\|_{L^{\infty}(\Mcal)}^{\frac{8r+3m}{2r+m}}}{\rho^{4}n^{\frac{2r-2}{2r+m}}}
    \end{equation*}
    for all $n$ and $\sigma^2\le (\diam(\Mcal))^2$, where $C_{\Mcal,r}>0$ is a constant depending only on $\Mcal$ and $r$.
\end{theorem}

It is easy to upgrade this result to a bound on the expected risk of $\hat{\delta}_{\T}$, although we obtain a non-sharp oracle inequality rather than a sharp one.

\begin{lemma}\label{lem:risk-to-dist}
    For any $\varepsilon>0$ there exists $C_{\varepsilon}>0$ such that we have $R(\delta) \le (1+\varepsilon)R(\delta_{\T}) + C_{\varepsilon}D_{\T}(\delta)$ for all denoisers $\delta:\Mcal\to\Mcal$, where $C_{\varepsilon}\to\infty$ as $\varepsilon\to 0$.
\end{lemma}

The above result, combined with Theorem~\ref{thm:EB-risk}, immediately yields
\begin{equation*}
        \Ebb\left[R\left(\hat{\delta}_{\T},\sigma^2\right)\right]\le (1+\varepsilon)R(\delta_{\T},\sigma^2) +\frac{C_{\Mcal,r,\varepsilon}\sigma^{2}\|f\|_{\Hcal^r(\Mcal)}^{\frac{2m+4}{2r+m}}\|g\|_{L^{\infty}(\Mcal)}^{\frac{8r+3m}{2r+m}}}{\rho^{4}n^{\frac{2r-2}{2r+m}}}
    \end{equation*}
    for all $n$ and $\sigma^2\le (\diam(\Mcal))^2$, where $C_{\Mcal,r,\varepsilon}>0$ is a constant depending only on $\Mcal$, $r$, and $\varepsilon$.
    
In the remainder of this subsection, we interpret the rate of convergence appearing in the previous result, in terms of $n$.
We note that $g$ is unknown but appears only in the constant, and that by Lemma~\ref{lem:f-lower-bd} we can choose $\rho$ to depend only on $\Mcal$ and $g$, provided that $g$ has a strictly positive lower bound.
In some cases it is also desirable to understand the dependence of the rate on $\sigma^2$, since the oracle-level results from Subsection~\ref{subsec:oracle} are asymptotic results in the regime where $\sigma^2\to 0$; however, this is a hard problem even in the Euclidean setting, and most existing rates of convergence for nonparametric empirical Bayes (e.g., \cite{Zhang1997, SahaGuntuboyina, Soloff}) either do not depend explicitly on $\sigma^2$ or in fact deteriorate as $\sigma^2\to 0$.
In the following, we use results of Subsection~\ref{subsec:sobolev} to explicitly describe the resulting rates of convergence in various settings.

First we consider Lie groups, where the rate of convergence depends explicitly on the smoothness of $g$.

\begin{example}[Lie group $\Mcal$, smooth $G$]\label{ex:Lie}
    Suppose $\Mcal$ is a compact connected Lie group (e.g., the circle $\mathbb{S}^1$, the torus $\mathbb{S}^1\times\mathbb{S}^1$, or the space of orientation-preserving rotations $\SO(3)$) and that the density $g$ of $G$ satisfies $g\in C^r(\Mcal)$.
    Then, Lemma~\ref{lem:Sobolev-Lie} implies $\|f\|_{\Hcal^r(\Mcal)}<\infty$.
    If $r\ge 2$, then,  in terms of constants depending only on $\Mcal,g$, and $\sigma^2$, we can set $\ell_n \asymp n^{m/(2r+m)}$ and achieve the rate of convergence $n^{-(2r-2)/(2r+m)}$.
    Note that this can be made arbitrarily close to the parametric rate $n^{-1}$ as $r\to\infty$, although the pre-factor may explode.
    To be concrete, consider the case that $g$ is twice-continuously differentiable so that we take $r=2$ above; if $m=1$ (e.g., the circle $\mathbb{S}^1$) then setting $\ell_n\asymp n^{1/5}$ leads to the rate of convergence $n^{-2/5}$, if $m=2$ (e.g., the torus $\mathbb{S}^1\times \mathbb{S}^1$) then setting $\ell_n\asymp n^{1/3}$ leads to the rate of convergence $n^{-1/3}$, and if $m=3$ (e.g., rotations $\SO(3)$) then setting $\ell_n\asymp n^{3/7}$ leads to the rate of convergence $n^{-2/7}$.
\end{example}

Second, we consider the sphere, where our result holds for all $g$ but with a slower rate of convergence.

\begin{example}[$\Mcal=\mathbb{S}^2$, general $G$]\label{ex:EB-sphere-rate}
    If $\Mcal$ is the sphere $\mathbb{S}^2$ and $g\in L^{\infty}(\mathbb{S}^2)$ has a strictly positive lower bound, then Lemma~\ref{lem:sphere-estimate} guarantees $\|f\|_{\Hcal^2(\Mcal)}\lesssim \sigma^{-2}\|g\|_{L^{\infty}(\Mcal)}$ hence we can take $r=m=2$ in the main result.
    Consequently, and again ignoring constants that depend on $\Mcal,g$, and $\sigma^2$, we see that setting $\ell_n\asymp n^{1/3}$ leads to the rate of convergence $n^{-1/3}$.
\end{example}

Last is the case of the circle, in which we can precisely identify the smoothness of $f$ in terms of the smoothness of $g$; note that the following result is sharper than what we get by treating $\mathbb{S}^1$ as a Lie group above.

\begin{example}[$\Mcal=\mathbb{S}^1$, smooth $G$]\label{ex:empirical-S1-smooth}
    If $\Mcal = \mathbb{S}^1$ and $G$ has a density $g$ with some smoothness, then we can obtain an explicit rate of convergence in terms of $\sigma^2$ and $n$.
    Indeed, if we have $\|g\|_{\Hcal^s(\Mcal)}<\infty$ for some $s\ge 0$, then Lemma~\ref{lem:Sobolev-bd-S1} shows that we may apply Theorem~\ref{thm:EB-risk} for $r=s+2$; then, ignoring factors that depend only on $\Mcal$ and $g$, it follows that setting $\ell_n\asymp \sigma^{-4/(2s+5)}n^{1/(2s+5)}$ leads to a rate of convergence of $\sigma^{(4s-2)/(2s+5)}n^{-(2s+2)/(2s+5)}$.
    We can understand this rate of convergence in a few cases of interest.
    For $s=0$, it coincides with the rate $\sigma^{-2/5}n^{-2/5}$ in Example~\ref{ex:Lie}, but with an explicit dependence on $\sigma^2>0$ which deteriorates as $\sigma^2\to 0$.
    For $s=1$, the rate is $\sigma^{2/7}n^{-4/7}$ which is faster in $n$ and vanishes as $\sigma^2\to 0$.
    (More generally, with respect to $\sigma^2\to 0$, the rate is non-deteriorating if and only if $s\ge 1/2$ and vanishing if and only if $s> 1/2$.)
    Lastly, we note that the rate can be made arbitrarily close to parametric rate $\sigma^2n^{-1}$ if $\|g\|_{\Hcal^s(\Mcal)}<\infty$ for all $s\ge 0$; indeed, for any $0<\varepsilon\le 1$, set $s:=(3-5\varepsilon)/(2\varepsilon)$ and $\ell_n\asymp \sigma^{-4/(2s+5)}n^{1/(2s+5)}\approx \sigma^{-4\varepsilon/3}n^{\varepsilon/3}$ to achieve the rate of convergence $\sigma^{2-4\varepsilon}n^{-(1-\varepsilon)}$.
    (However, the constant prefactor may explode as $\varepsilon\to 0$.) 
\end{example}

In the special case of the circle $\Mcal=\mathbb{S}^1$, we complement the rate of convergence in Example~\ref{ex:empirical-S1-smooth} by providing the following matching lower bound, uniformly over all densities $g$ in a suitable subset of a Sobolev ball of order $s\ge 0$.
In particular, this establishes that $\hat{\delta}_{\T}$ is minimax-optimal and that the minimax rate is nonparametric, $n^{-(2s+2)/(2s+5)}$.
For convenience, we emphasize the dependence of $\Ebb,f$, and $\delta_{\T}$ on $g$ by writing $\Ebb_g,f_g$, and $\delta_{\T,g}$, respectively.

\begin{theorem}\label{thm:S1-lower-bd}
    Suppose $\Mcal=\mathbb{S}^1$, fix $0<\alpha<(2\pi)^{-1}<\beta$, $s\ge0$, and $L>(2\pi)^{-1/2}$, and define
    \begin{equation*}
        \mathcal{G}_s(\alpha,\beta,L):=\left\{g:\mathbb{S}^1\to\Rbb: \int_{\mathbb{S}^1}g(\theta)\diff\theta =1,\alpha\le g(\theta)\le \beta \textnormal{ for all }\theta\in\mathbb{S}^1\textnormal{, and } \|g\|_{\Hcal^s(\mathbb{S}^1)}\le L\right\}.
    \end{equation*}
    Then, for all $0<\sigma^2<\pi^2$, we have
    \begin{equation*}
        \inf_{\hat{\delta}}\sup_{g\in \mathcal{G}_s(\alpha,\beta,L)}\Ebb_g\left[\int_{\mathbb{S}^1}d^2\left(\hat{\delta}(x;X_1,\ldots,X_n),\delta_{\T,g}(x)\right)f_g(x)\diff x\right] \ge\frac{C_{\sigma^2, s,\alpha,\beta,L}}{n^{\frac{2s+2}{2s+5}}}
    \end{equation*}
    for all $n$, where $C_{\sigma^2,s,\alpha,\beta,L}>0$ is a constant depending only on $\sigma^2,s,\alpha,\beta$, and $L$, and the infimum is taken over all measurable functions $\hat{\delta}:\mathbb{S}^1\times(\mathbb{S}^1)^n \to \mathbb{S}^1$.
\end{theorem}

The proof of this result reveals that the hardness of denoising coincides with the hardness of estimating the derivative of $f$, and the rate of convergence reflects the fact that $f$ inherits 2 more degrees of smoothness than $g$ because of the convolution structure; indeed, the rate of convergence should be understood as $n^{-2p/(2p+1)}$ for $p=(s+2)-1$.
The obstruction to extending this analysis to general $\Mcal$ is that we do not have a precise understanding of the smoothing properties of Riemannian Gaussian mixture models in cases other than $\Mcal=\mathbb{S}^1$, which is the same difficulty we encountered in  Subsection~\ref{subsec:sobolev}.
In some sense, these non-parametric rates are a manifestation of the well-known ``smeariness'' phenomenon in non-Euclidean statistics, whereby the rates of convergence can suffer in problems where the likelihood is not differentiable at the cut locus of its Fr\'echet mean \cite{HotzHuckemann,EltznerHuckemann,Eltzner,HundrieserEltznerHuckemann}.

\section{Numerical Simulations}\label{sec:sim}

In this section we give some numerical simulations to illustrate the theory of Section~\ref{sec:theory}, and discuss some computational aspects of the denoising procedures that will be used in the applications of Section~\ref{sec:appl}.
The simulations are designed with two goals in mind.
First, we are interested in comparing the risks of $\delta_{\T}$ with those of $\delta_{\N}$ and $\delta_{\B}$ as $\sigma^2\to 0$ (as in Subsection~\ref{subsec:oracle}).
Second, we are interested in comparing the distance between $\hat{\delta}_{\T}$ and $\delta_{\T}$ (as in Subsection~\ref{subsec:empirical}).

As in the rest of the paper, the latent variables $\Theta_1,\ldots, \Theta_n$ are i.i.d.~samples from some distribution $G$, and the measurements $X_1,\ldots, X_n$ are conditionally independent Riemannian Gaussian random variables given $\{\Theta_1=\theta_1,\ldots, \Theta_n=\theta_n\}$ where $X_i$ has location parameter $\theta_i$ and some scale parameter $\sigma>0$, for all $1\le i\le n$.
Also recall that for any denoiser $\delta$ we define the risk $R(\delta,\sigma^2)$ via~\eqref{eqn:def-risk} and the distance $D_{\T}(\delta)$ from the oracle $\delta_{\T}$ via~\eqref{eqn:def-dist}.

In Appendix~\ref{app:implementation} we give a detailed description of the implementation of our denoisers and our simulations, but we briefly summarize a few important points here.
First, we note that, even with oracle knowledge of the latent variable distribution $G$, evaluating the oracle tangential Bayes denoiser $\delta_{\T}$ requires numerical integration because evaluating the marginal density $f(x)=\int_{\Mcal}p_{\theta,\sigma^2}(x)\diff G(\theta)$ already requires numerical integration; indeed, we use a Monte Carlo approximation of $f$ by drawing additional oracle samples of $\Theta$ from $G$, and then we use the identity $\delta_{\T}(x) = \Exp_x(\sigma^2\nabla_xf(x)/f(x))$.
Second, we compute the oracle Bayes denoiser $\delta_{\B}$ by first approximating the {conditional Fr\'echet mean of $\Theta$ given $\{X=x\}$ via Nadaraya-Watson smoothing (see~\eqref{eqn:pop-FM} and Appendix~\ref{app:implementation})} using additional oracle samples of $(\Theta,X)$ from \eqref{eqn:pop-model}, and then by numerically computing the Fr\'echet mean of the resulting distribution using the built-in methods from \texttt{Geomstats} \cite{GeomStats}.
Lastly, we note that choosing the optimal parameters $\ell_n$ and $\rho$ (recall~\eqref{eq:Density-Est} and~\eqref{eq:Est-Denoiser}) requires some oracle knowledge of $G$; for $\ell_n$, in this section we follow the oracle choices suggested in Subsection~\ref{subsec:empirical}, since our goal is simply to verify our theoretical developments, while $\rho$ is set to the 1st percentile of the positive values taken on by the density $\hat{f}_n$.
Later, {in the applications section, we will choose both parameters in a fully data-driven way with cross-validation of the Riemannian version of Hyv\"arinen's score-matching objective \cite{hyvarinen2005,vincent2011connection,de2022riemannian}.}

All of our simulations have the same form.
That is, for various structured distributions $G$ on a Riemannian manifold $\Mcal$, we compute the following.
First, we compute the oracle-level na\"ive, Bayes, and tangential Bayes denoisers on the basis of 10,000 oracle samples and we let $\sigma^2$ vary over the interval $(0,0.25]$ in 6 linearly-spaced increments, then we plot the risks $R(\delta_{\N},\sigma^2)$, $R(\delta_{\B},\sigma^2)$ and $R(\delta_{\T},\sigma^2)$ as a function of $\sigma^2$.
Second, we fix {$\sigma^2=0.15$ and we let $n$ vary over the interval $\{100, \ldots, 10,000\}$ in 5 logarithmically-spaced increments}, then we plot $\Ebb[D_{\T}(\hat{\delta}_{\T})]$ as a function of $n$, on a log-log scale; in this way, the slope of the resulting line reflects the rate of convergence of the empirical tangential Bayes denoiser $\hat{\delta}_{\T}$, as predicted by Theorem~\ref{thm:EB-risk}.
In the rest of this section, we discuss several simulations in the circle $\mathbb{S}^1$ and the sphere $\mathbb{S}^2$.

\subsection{Circle, $\mathbb{S}^1$}\label{subsec:sim-circle}
The first setting of interest is the circle, $\Mcal=\mathbb{S}^1$, in which our methodology and results already reveal some interesting new phenomena.
Throughout this subsection, we represent $\mathbb{S}^1 =\Rbb/2\pi \Zbb$ for convenience in some expressions.

In order to compute the empirical Bayes denoiser $\hat{\delta}_{\T}$, we need to be able to compute sums of eigenfunctions of the Laplace-Beltrami operator $\Delta$, and derivatives thereof.
To do this, we recall from standard Fourier analysis that the eigenvalue $\lambda_0=0$ occurs with multiplicity 1 and that $\lambda_{m}=m^2$ occurs with multiplicity 2 for $m\ge 1$.
As such, it is natural to index our eigenfunctions by integers $k\in\Zbb$, so that the eigenvalue $\lambda_{k}=k^2$ has corresponding eigenfunction $\phi_{k}(x) = (2\pi)^{-1/2}e^{ik x}$, where we recall that we have chosen to represent $\mathbb{S}^1=\Rbb/2\pi\Zbb$.
Thus if we construct the kernel $K$ by fixing $M>0$ and summing over $|k|\le M$, we have 
\begin{equation*}
    K(x,y,M) = \sum_{|k|\le M}\frac{1}{2\pi}e^{ik(x-y)} := \frac{1}{2\pi}D_M(x-y)
\end{equation*}
where $\{D_M\}_{M\in\Nbb}$ is the \textit{Dirichlet kernel} (see \cite[p. 206]{TaylorPDE}) defined by
\begin{equation}\label{eqn:Dirichlet-kernel}
    D_M(z) := \frac{\sin\left(\left(M+\frac{1}{2}\right)z\right)}{\sin\left(\frac{1}{2}z\right)},
\end{equation}
for $z\neq 0$, and $D_M(0)=2M+1$ by continuity.
Recall that the Dirichlet kernel is responsible for the ``ringing'' phenomenon whereby high-density regions close to low-density regions can create the appearance of waves (e.g., in the density estimate) which are not truly present in the data; this phenomenon will be seen in the application to protein structure in Subsection~\ref{subsec:protein}.
Note that $M$ is related to $\ell_n$ via $\ell_n = 2M+1$, i.e., $M\asymp \ell_n$.

We consider the following examples for the latent variable distribution $G$ and corresponding hyperparameter choices.
The first is the uniform distribution (i.e., $\diff G(\theta) \propto\diff \theta$), in which we set $M \propto \log n$. The second is the uniform distribution on a cap (i.e., $\diff G(\theta) \propto\ind\{a\le \theta\le b\}\diff \theta$), in which (Example~\ref{ex:empirical-S1-smooth}) we set $M\asymp n^{1/6}$ since direct calculation reveals $\|g\|_{\Hcal^s(\mathbb{S}^1)}<\infty$ if and only if $s<1/2$.
The remaining distributions are themselves Riemannian Gaussian mixture models with varying numbers of equally-spaced components and varying component scale parameters $\tau>0$, i.e.,
in the 2-mode case (i.e., antipodal components) as $\tau^2=0.1$, in the 3-mode case as $\tau^2=0.05$, and in the 5-mode case as $\tau^2=0.0005$; in these cases, we set (Example~\ref{ex:empirical-S1-smooth}) $M\asymp \log n$ since the theory suggests letting $M$ grow slower than $n^{\varepsilon}$ for all $0<\varepsilon<1$.
We always set $\rho>0$ to the first percentile of the empirical density $\hat{f}_n$.

\begin{figure}[t]
    \centering
    \includegraphics[width=1\linewidth]{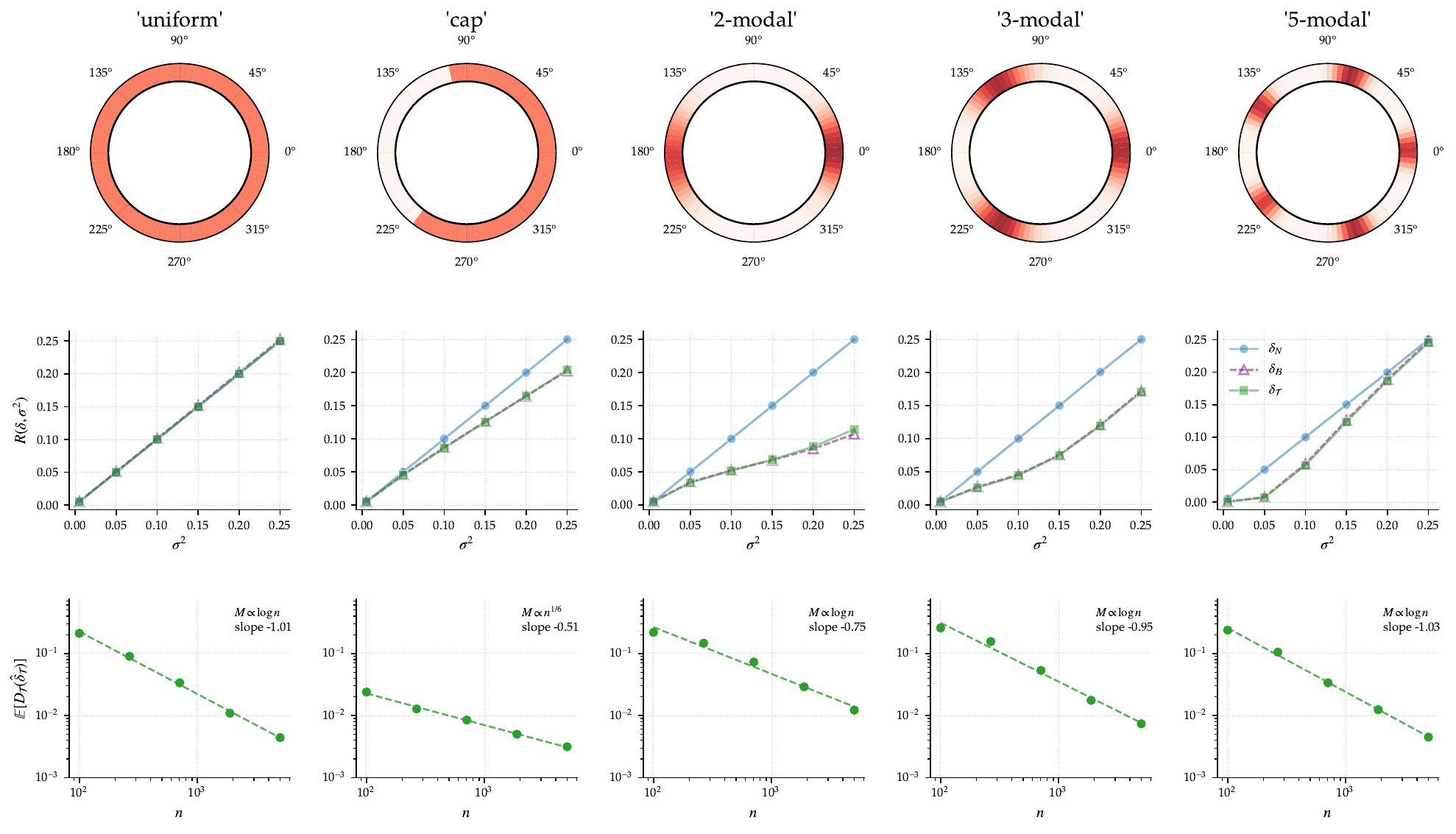}
    \caption{A simulation of the risks of the oracle and empirical denoisers in the circle $\mathbb{S}^1$.
    We show the latent variable distribution $G$ (top row), the risks of the oracle-level denoisers $\delta_{\N}$, $\delta_{\T}$, and $\delta_{\B}$ when $n$ is fixed and $\sigma^2$ decreases (middle row), and the distance between the empirical-level and oracle-level denoisers $\hat{\delta}_{\T}$ and $\delta_{\T}$ when $\sigma^2$ is fixed and $n$ increases (bottom row).}
    \label{fig:S1_sims}
\end{figure}

The results of these simulations are shown in Figure~\ref{fig:S1_sims}, in which the top row shows the 5 different choices of $G$, the middle row shows the risks of the oracle denoisers $\delta_{\N},\delta_{\T}$, and $\delta_{\B}$, and the bottom row shows the distance $D_{\T}(\hat{\delta}_{\T})$ of the empirical tangential Bayes denoiser from the oracle tangential Bayes denoiser.
Next we make some remarks about these plots and their relationship with the theory of Section~\ref{sec:theory}.
At the oracle level, we observe that the risks of $\delta_{\B}$ and $\delta_{\T}$ are nearly indistinguishable (Theorem~\ref{thm:uncond-main}) while the risk of $\delta_{\N}$ is approximately equal to $\sigma^2$ (Proposition~\ref{prop:naive-risk}); the only exception is the uniform distribution in which case all of the risks coincide (Example~\ref{ex:oracle-uniform}).
At the empirical level, we note that the slope of the line in the case of the cap example is $-0.5$ which corresponds to a rate of convergence of $n^{-1/2}$  (Example~\ref{ex:empirical-S1-smooth}), and in all other cases the slope is approximately $-1$ which corresponds to the nearly-parametric rate of convergence $n^{-1}$ (Example~\ref{ex:empirical-S1-smooth}).


\subsection{Sphere, $\mathbb{S}^2$}\label{subsec:sim-sphere}

In this next example, we consider the sphere, $\Mcal=\mathbb{S}^2$.
This is one of the fundamental examples in directional statistics, and will be the focus of our application to astronomy in Subsection~\ref{subsec:GRB}.
We emphasize that visualizing data on the sphere requires choosing a particular two-dimensional projection of $\mathbb{S}^2$ onto the plane; in this work we always use the \textit{Mollweide projection}, which is an equal-area projection, hence it is well-suited to visualizing scatter plots and other statistical objects, and is commonly used in astronomy \cite{Tirion}.
However, we emphasize that our denoising procedure takes place in the intrinsic geometry of $\mathbb{S}^2$, not in the projected geometry; in this regard, other projections can be chosen for the visualization step without changing the underlying methodology.

In order to apply our method we need to recall some aspects of the eigendecomposition of the Laplace-Beltrami operator on $\mathbb{S}^2$.
It turns out that the eigenvalues occur with multiplicities, so the eigendecomposition can be conveniently parameterized by integer pairs $(m,k)$ with $|k|\le m$;
more precisely, the eigenvalue $\lambda_{m} = m(m+1)$ occurs with multiplicity $2m+1$, and its corresponding eigenfunctions are the so-called \textit{spherical harmonics} $\phi_{m,-m},\ldots, \phi_{m,m}:\mathbb{S}^2\to\Cbb$.
We can also simplify the calculation of our kernel by choosing coordinates for $\mathbb{S}^2$ by representing it as $\{x\in\Rbb^3: \|x\|=1\}$; then, by fixing $M$ and summing over all $(m,k)$ with $m\le M$, the addition formula for spherical harmonics yields
\begin{equation*}
    K(x,y,M):=\sum_{m=0}^{M}\sum_{|k|\le m}\phi_{m,k}(x)\overline{\phi_{m,k}(y)} = \sum_{m=0}^{M}\frac{2m+1}{4\pi}P_{m}\left(x^{\top} y\right)
\end{equation*}
where $\{P_{m}\}_{m\in\Nbb}$ are the so-called \textit{Legendre polynomials}.
Note that $M$ is related to $\ell_n$ via $\ell_n =\sum_{m=0}^{M}(2m+1)=(M+1)^2$, i.e., $M\asymp \ell_n^{1/2}$.

The examples of our latent variable distribution $G$ and hyperparameter choices are as follows.
As before, we consider 5 choices for $G$; the first is the uniform distribution (i.e., $\diff G(\theta) \propto\diff \theta$), and the second is the uniform distribution on a cap (i.e., $\diff G(\theta) \propto\ind\{\theta\in C\}\diff \theta$ for suitable $C\subseteq \mathbb{S}^2$).
The third is approximately supported on the equator; more precisely it has a density given by $g(\theta)=\int_{\Mcal}p_{\eta,\tau^2}(\theta)\diff \nu(\eta)$ where $\nu$ is the uniform distribution on the equator and $\tau^2= 0.0001$.
The remaining are Riemannian Gaussian mixture models with varying numbers of components and varying component scale parameters $\tau^2>0$.
While the locations of the components cannot be exactly equally-spaced in this setting, we make them approximately so by placing them at the locations of a Fibonacci spiral.
The component scale parameters are chosen in the 5-mode case as $\tau^2=0.05$, and in the 10-mode case as $\tau^2=0.001$.
In all cases we set $M$ according to the conservative estimate $M\asymp n^{1/6}$, i.e., $\ell_n=n^{1/3}$ (Example~\ref{ex:EB-sphere-rate}).

\begin{figure}[t]
    \centering
    \includegraphics[width=1\linewidth]{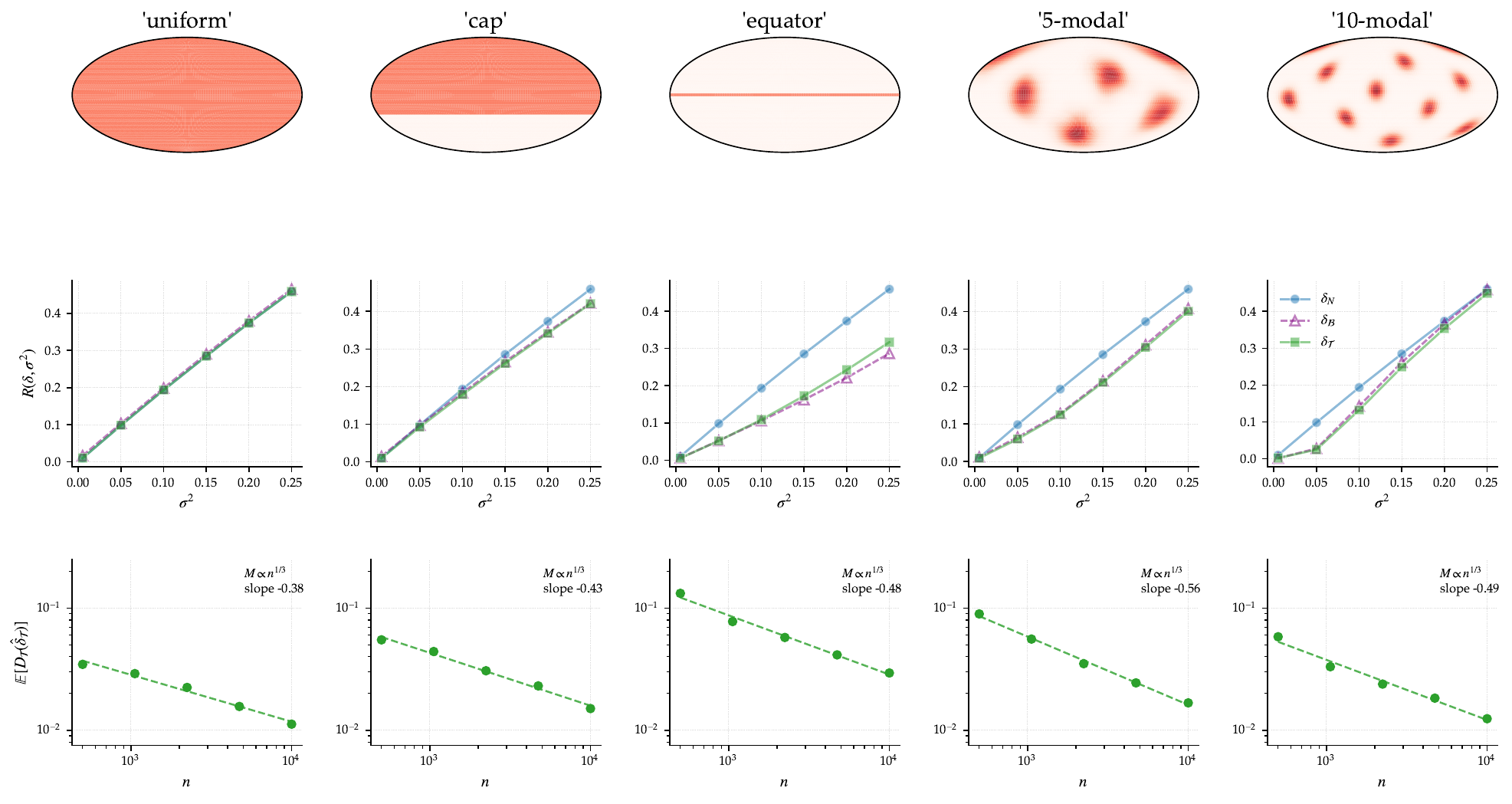}
    \caption{A simulation of the risks of the oracle and empirical denoisers in the sphere $\mathbb{S}^2$.
    We show the latent variable distribution $G$ (top row), the risks of the oracle-level denoisers $\delta_{\N}$, $\delta_{\T}$, and $\delta_{\B}$ when $n$ is fixed and $\sigma^2$ decreases (middle row), and the distance between the empirical-level and oracle-level denoisers $\hat{\delta}_{\T}$ and $\delta_{\T}$ when $\sigma^2$ is fixed and $n$ increases (bottom row).}
    \label{fig:S2_sims}
\end{figure}

We show the results of these simulations in Figure~\ref{fig:S2_sims}, where the top row shows the 5 different choices of $G$, the middle row shows the risks of the oracle denoisers, and the bottom row shows the distance between the oracle and empirical tangential Bayes denoisers, and we make some comments.
At the oracle level, we note as before that $\delta_{\B}$ and $\delta_{\T}$ have nearly equal risks (Theorem~\ref{thm:uncond-main}), and that the risk of $\delta_{\N}$ is approximately equal to $2\sigma^2$ (Proposition~\ref{prop:naive-risk}); we emphasize in the equator example that the risks of $\delta_{\B}$ and $\delta_{\T}$ are approximately $\sigma^2$, which reflects the fact that the latent variable distribution is approximately supported on a one-dimensional submanifold. At the empirical level, we observe slopes approximately equal to, or steeper
than, $-1/3$, corresponding to convergence at least as fast as $n^{-1/3}$ in
these simulations (Example~\ref{ex:EB-sphere-rate}).
Also note in the 10-modal example that the our approximation of the oracle Bayes denoiser $\delta_{\B}$ has larger risk than the oracle tangential Bayes denoiser $\delta_{\T}$, and the difference is due both to Monte Carlo error and to numerical precision of the approximation; this highlights that both denoisers have essentially the same risk in this regime.

\section{Scientific Applications}\label{sec:appl}

In this section we apply our methodology to two denoising problems from concrete scientific applications, where the goal is to estimate some latent manifold-valued random variables $\Theta_1,\ldots, \Theta_n$ from noisy manifold-valued measurements $X_1,\ldots, X_n$ thereof; that is, we compute the empirical tangential Bayes denoiser $\hat{\delta}_{\T}$ and then the denoised data set $\hat{\delta}_{\T}(X_1),\ldots,\hat{\delta}_{\T}(X_n)$.
In each application we also display our estimate $\nabla \hat f_n/ \max\{\rho, \hat {f}_n\}$ of the score field, which (by slight abuse of notation) we denote $\nabla \log \hat{f}_n$; this helps to visualize the fine properties of the empirical tangential Bayes denoiser.

Since the distribution $G$ is not known in practice, we cannot use the oracle choice of $M$ and $\rho$ described in Theorem~\ref{thm:EB-risk} for hyperparameter selection; instead, we use cross-validation of the Riemannian version of Hyv\"arinen's score-matching objective \cite{de2022riemannian}, i.e. for a candidate density $\tilde{f}:\Mcal\to\Rbb$ computed on some training data we compute
\begin{equation}\label{eqn:Hyvarinen-Riemannian}
    \hat{\mathcal{J}}(\tilde{f}) = \frac{1}{n}\sum_{i=1}^{n}\left(\|\nabla\log\tilde{f}(X_i)\|_{X_i}^2 + 2\,\Delta\log\tilde{f}(X_i)\right)
\end{equation}
on a test set of hold-out data.
We also use an Akaike Information Criterion (AIC) penalty to prevent overfitting.
See Appendix~\ref{app:implementation} for further detail on computation and implementation.
(Recall the sign convention for $\Delta$ in the Riemannian setting, which may cause confusion about the $+$ sign in \eqref{eqn:Hyvarinen-Riemannian}.)

\subsection{Localization of Gamma Ray Bursts}\label{subsec:GRB}

Our first application concerns denoising the locations of a catalog of gamma ray bursts (GRBs), which are short-lived celestial objects consisting of an extremely powerful emission of light; see \cite{Piran} for a detailed introduction.
Because of their transient nature, estimating the positions of GRBs, a task referred to as \textit{localization}, is known to be difficult and has attracted a great deal of research attention \cite{UlyssesGRB, Connaughton_2015, GRB_errors}.
Moreover, GRBs are situated at an extremely large distance away from Earth, so their positions are often recorded as points on a sphere, in the so-called \textit{celestial coordinates} of declination and right-ascension.
Thus, we naturally regard GRB localization as a sphere-valued denoising problem.
We use the data from the BATSE-4Br catalog\footnote{Available at \url{https://gammaray.nsstc.nasa.gov/batse/grb/catalog/4b/4br_basic.html}}, which includes the positions of GRBs observed during 1991-1996.

To cast this problem in our framework, we set $\Mcal := \mathbb{S}^2$ to be the two-dimensional sphere, we write $\Theta_1,\ldots, \Theta_n$ for the latent positions of $n=2702$ GRBs, which we assume are i.i.d.~samples from some $G\in\Pcal(\mathbb{S}^2)$ which possesses a density with respect to the volume form, and we write $X_1,\ldots, X_n$ for measurements of $\Theta_1,\ldots, \Theta_n$, respectively.
We follow the conservative estimates from \cite{GRB_errors} that the error of the measurements is roughly 10 degrees, yielding $\sigma^2 = (10\pi/180)^2\approx 3.046\times 10^{-2}$.

Implementing the empirical tangential Bayes denoiser $\hat{\delta}_{\T}$ requires some computations with the eigenfunctions of the Laplace-Beltrami operator, which in this case are the spherical harmonics; to this end, we follow the set-up in Subsection~\ref{subsec:sim-sphere}. 

We also note that the positions of GRBs are not typically modeled with Riemannian Gaussian distributions but rather with von Mises-Fisher distributions \cite[Section~5]{Connaughton_2015}, but that these distributions share many qualitative features (unimodality, squared-exponential decay away from the location parameter, etc.) and are extremely similar when $\sigma^2$ is small.

The results are shown in Figure~\ref{fig:astro}.
As in Section~\ref{sec:sim}, we use the Mollweide projection to visualize the data, but the denoising takes place in the intrinsic geometry of the sphere and hence does not rely on this choice of visualization.
The selected hyperparameters are $\ell_n=576$ (i.e., $M=23$) and $\rho \approx 0.017$.
While the marginal distribution of the measurements appears to be close to uniform, we observe that the denoised data set has many small clusters which themselves seem to be approximately uniformly distributed over the sphere.

\begin{figure}[t]
    \centering
    \includegraphics[width=1.0\linewidth]{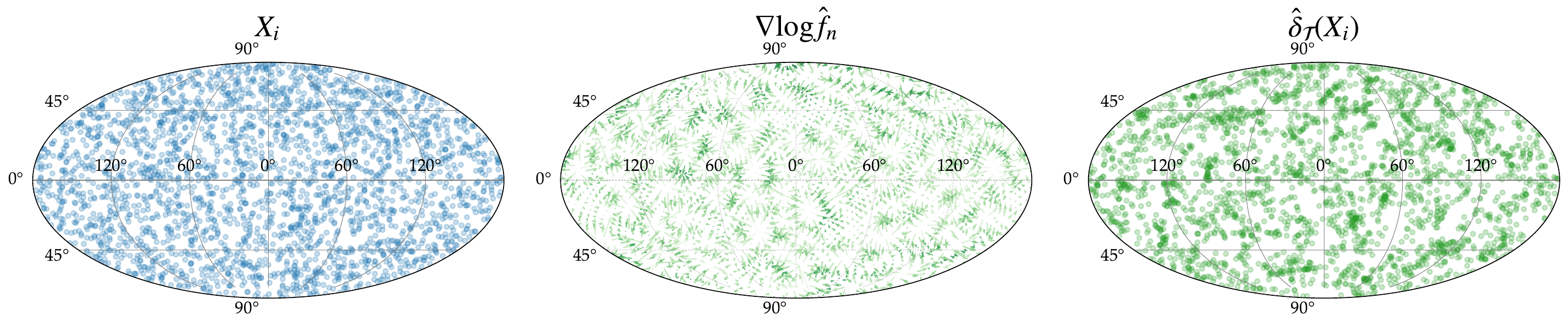}
    \caption{Denoising the locations of a catalog of gamma ray bursts, visualized with the Mollweide projection of the sphere $\mathbb{S}^2$.
	We plot the measurements $X_1,\ldots, X_n$ (left), the estimated score field $\nabla \log\hat{f}_n$ (center), and the denoised data set $\hat{\delta}_{\T}(X_1),\ldots, \hat{\delta}_{\T}(X_n)$ (right).}
    \label{fig:astro}
\end{figure}

\subsection{Protein Structure from Ramachandran Plots}\label{subsec:protein}

Our second application concerns denoising the pairs of torsion angles formed by adjacent amino acids along the backbone of a given protein; the plot of all such pairs, referred to as a \textit{Ramachandran plot}, is a scatter plot on the torus which contains a great deal of information regarding the structure of a protein \cite{Ramachandran,Hollingsworth,Pertsemlidis,Park}.
We focus on the ten enzymes of the glycolytic pathway (hexokinase, phosphoglucose isomerase, phosphofructokinase, aldolase, triosephosphate isomerase, GAPDH, phosphoglycerate kinase, phosphoglycerate mutase, enolase, and pyruvate kinase), obtained from the Protein Data Bank (PDB)\footnote{Available at: \url{https://www.rcsb.org/} via 1HKG, 1IRI, 1PFK, 1ADO, 2YPI, 1GD1, 3PGK, 1E58, 1ENO, and 1A3W}.
Analyzing a functionally related set of enzymes from the same metabolic pathway allows us to study shared structural features in the Ramachandran plot across proteins with diverse sequences.

To put this in the framework above, we set $\Mcal := \mathbb{S}^1\times \mathbb{S}^1$ to be the two-dimensional flat torus, and for each enzyme $1\le j\le 10$ we write $\Theta_1,\ldots, \Theta_{n_j}$ for the pairs of latent adjacent torsion angles, which we assume are i.i.d.~samples from some $G\in\Pcal(\mathbb{S}^1\times \mathbb{S}^1)$ possessing a density with respect to the volume form, and $X_1,\ldots, X_{n_j}$ for measurements thereof.
We pool all $n = \sum_j n_j$ torsion-angle pairs and fit a single density $\hat f_n$ to the pooled data, then apply the resulting denoiser to each enzyme separately.
In order to determine $\sigma^2$, we note that dihedral angles are not measured directly but rather are computed from measurements of atomic positions; as such, we take, as a crude estimate, $\sigma:=\arctan(\tau/b)$ where $\tau^2 = \bar{B}/(8\pi^2)$ is the mean-square atomic displacement derived from the pooled mean B-factor $\bar{B} \approx 26.42$\AA$^2$ (as in \cite{Rupp2009,Bfactors}) and $b \approx 1.5$\AA\ is a typical bond length, leading to the pooled estimate $\sigma^2 \approx 0.136$.
Although our assumptions of independence need not hold in this setting, we believe that correlations among $\Theta_1,\ldots, \Theta_n$ (and conditional correlations among $X_1,\ldots, X_n$ given $\Theta_1,\ldots, \Theta_n$) are weak and hence the methodology still provides a reasonable application of the method in practice.

Implementing our results requires an explicit form of the eigendecomposition of the Laplace-Beltrami operator on $\mathbb{S}^1\times \mathbb{S}^1$; this is easy since we studied the eigendecomposition of $\Delta$ on $\mathbb{S}^1$ in Subsection~\ref{subsec:sim-circle}, and the Laplace-Beltrami operator tensorizes on product manifolds.
More precisely, if we coordinatize $\mathbb{S}^1\times \mathbb{S}^1$ by representing $x\in\mathbb{S}^1\times \mathbb{S}^1$ as $x=(x_1,x_2)\in(\Rbb/2\pi\Zbb)\times (\Rbb/2\pi\Zbb)$, we may fix $M>0$ and take our kernel to be
\begin{equation*}
    K(x,y,M) = \sum_{|k_1|,|k_2|\le M}\frac{1}{(2\pi)^2}e^{ik_1(x_1-y_1) + ik_2(x_2-y_2)} = \frac{1}{(2\pi)^2}D_M(x_1-y_1)D_M(x_2-y_2)
\end{equation*}
where $\{D_M\}_{M\in\Nbb}$ is the Dirichlet kernel defined in~\eqref{eqn:Dirichlet-kernel}.

We show the results in Figures~\ref{fig:chemistry} (and Figure~\ref{fig:chemistry-perenzyme} in Appendix~\ref{app:additional-figs}), where the selected hyperparameters are {$\ell=38$ (i.e., $M=2$)} and $\rho \approx 0.006$.
We observe that the denoiser exhibits some shrinkage near the upper-left, center-left, and center-right regions of the torus, which are known to correspond to the so-called \textit{$\beta$-sheets}, \textit{right-handed $\alpha$-helices}, and \textit{left-handed $\alpha$-helices}, respectively.
Overall, the empirical Bayes denoiser exhibits the same qualitative behavior as in the Euclidean setting: shrinkage towards high-density regions and stretching away from low-density regions of $\hat{f}_n$.
However, we notice some ``ringing'' phenomenon near the coordinates $(\frac{\pi}{2},-\frac{\pi}{4})$, whereby the shrinkage is not due to concentration of left-handed $\alpha$-helix torsion angles but rather due to nearby concentration of $\beta$-sheet torsion angles.

We make one last remark on the statistical interpretation of the denoised Ramachandran plot.
In many settings the Ramachandran plot is meant to be compared with a known decomposition of the torus $\mathbb{S}^1\times \mathbb{S}^1$ into ``favorable'', ``allowed'', and ``disallowed'' regions, in which case the denoised data set may not be a good surrogate for the latent variables; indeed, it is a manifestation of the shrinkage phenomenon that $\hat{\delta}_{\T}(X_1),\ldots, \hat{\delta}_{\T}(X_n)$ and $\Theta_1,\ldots,\Theta_n$ can be close in the sense that their mean squared error is small while being far in the sense that their distributions differ.
Performing empirical Bayes denoising subject to some constraints that the distributions match, would require adapting methodologies like those in \cite{ConstrainedEB} to data on the torus or more general Riemannian manifolds.

\begin{figure}[t]
\centering
\includegraphics[width=0.85\linewidth]{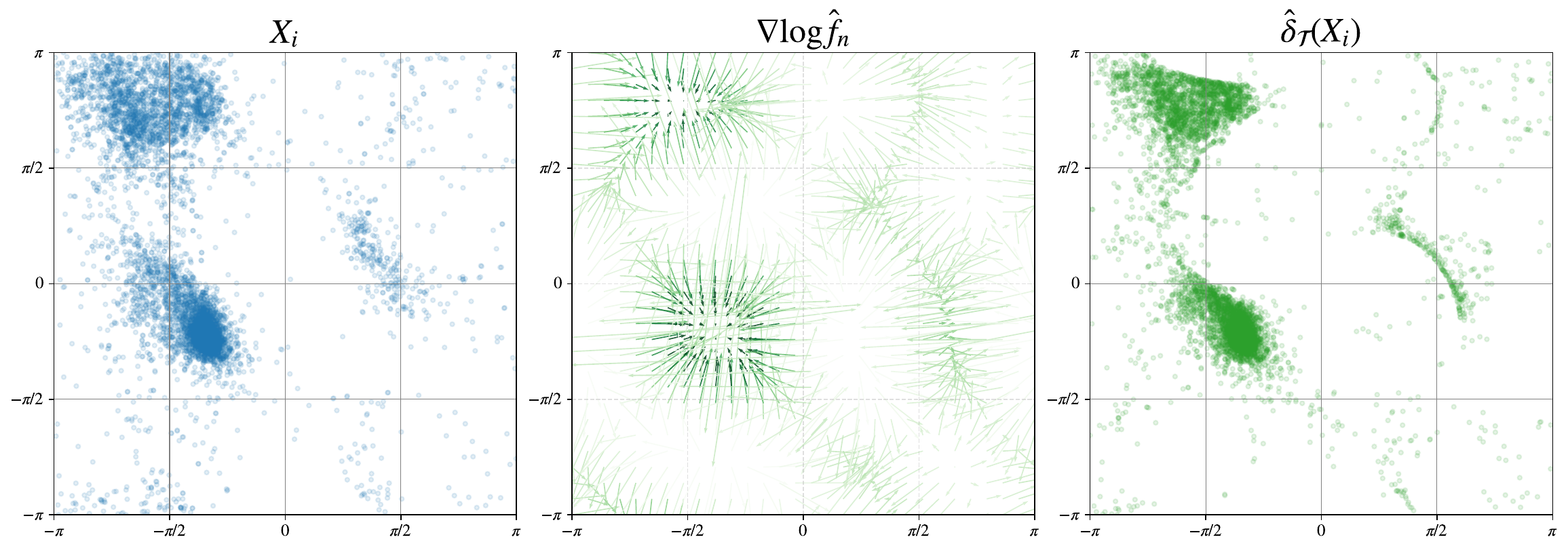}
\caption{Denoising the Ramachandran plot of pairs of torsion angles $(\phi,\psi)\in \mathbb{S}^1\times \mathbb{S}^1$ of adjacent amino acids across ten glycolytic enzymes.
	We plot the measurements $X_1,\ldots, X_n$ (left), the estimated score field $\nabla \log\hat{f}_n$ (center), and the denoised data set $\hat{\delta}_{\T}(X_1),\ldots, \hat{\delta}_{\T}(X_n)$ (right).}
\label{fig:chemistry}
\end{figure}

\section{Future Work}

In this paper we have developed an empirical Bayes methodology for denoising, which (in the terminology of Efron \cite{Efron2014}) is based on $f$-modeling, but it is natural to also consider strategies based on $g$-modeling.
That is, we may first estimate the density $g$ of the distribution of $\Theta$ (or, we may estimate the distribution $G$ itself), and then we plug this into the definition of $\delta_{\T}$ or $\delta_{\B}$ to form our empirical Bayes denoiser.
Although $g$-modeling does not involve the Tweedie-Eddington formula in the construction of the denoiser, it is still used crucially in the development of statistical theory; as such, our Tweedie-Eddington formula (Lemma~\ref{lem:Tweedie}) can provide the first step in future work on $g$-modeling approaches to nonparametric Riemannian empirical Bayes. 
It is generally known that $g$-modeling is preferable to $f$-modeling (see \cite{ShenWu} for a detailed discussion), and allows one to accommodate heteroskedasticity in the likelihood, as in our applications in Section~\ref{sec:appl}.

Another direction for future work is to extend our theory from the empirical Bayes setting to the compound decisions setting, namely the setting where $\Theta_1,\ldots,\Theta_n\in\Mcal$ are fixed (i.e., non-random) but unknown \cite{Zhang1997, Soloff, PolyanskiyWuPoisson}.
This is closely related to our empirical Bayes setting but where the distribution $G$ is replaced by the empirical measure of $\Theta_1,\ldots,\Theta_n$, denoted $\bar{G}_n$, which does not even have a density.
While our Tweedie-Eddington formula (Lemma~\ref{lem:Tweedie}) fails at some points since $\bar{G}_n$ has atoms, it is still well-defined almost everywhere (with respect to the volume form).
The compound decisions setting thus introduces further technical difficulties beyond what we have addressed in this work.

Lastly, we discuss the potential application of the methodology in this work to the mathematical theory of cryogenic electron microscopy (cryo-EM).
In cryo-EM, many identical copies of a particle of interest are suspended in an aqueous solution, rapidly frozen, and then imaged with an electron microscope; in this process, each particle leads to exactly one image since the imaging procedure destroys the particle.
While the random orientations are sometimes viewed as a nuisance parameter, several recent works \cite{Harpaz_Shkolnisky_2023, Mendez, BayesianOrientationCryoEM} have argued that the subroutine of estimating the latent orientations is of independent interest, e.g., since it is known empirically that particles do not experience re-orientations uniformly at random.
The work \cite{BayesianOrientationCryoEM} proposes a Bayesian methodology for the orientation estimation problem, where priors are carefully designed to reflect expected symmetries in the unknown molecule.
However, as the molecule and its symmetries are unknown in practice, it is of interest to develop methods that can borrow strength across particles while making minimal assumptions about the structure of the latent rotations; we believe that our empirical Bayes methodology balances these practical considerations.
We view this as an interesting future direction.

\begin{acks}[Acknowledgments]
We thank Nikolaos Ignatiadis for many useful conversations about empirical Bayes, and we thank Philip H.~Woods for recommending the application to Ramachandran plots.
LS gratefully acknowledges support from SNSF grant 10001283 to Victor M.~Panaretos.
BS gratefully acknowledges support from NSF (DMS-2311062 and DMS-2515520).
\end{acks}

\bibliographystyle{imsart-number} 
\bibliography{refs}       

\begin{appendix}

\section{Proofs of Results}\label{App:Proofs}

In this appendix we prove the main results from Section~\ref{sec:theory} of the paper, as well as some useful auxiliary results that are needed along the way.

\subsection{Proofs from Subsection~\ref{subsec:prelim}}\label{subsec:proofs-prelim}

First we give proofs of our results on Riemannian Gaussian mixture models.

\begin{proof}[Proof of Lemma~\ref{lem:Tweedie}]
    The second claim follows from the first, since the cut locus of each point has measure zero under the volume form; see \cite[p.~64]{Chavel}. 
    To prove the first claim, fix $x\in\Mcal$ such that $G(C_x)=0$.
    The compactness of $\Mcal$ implies that there exists $L>0$ such that for every $\theta\in\Mcal$ the map $y\mapsto d^2(y,\theta)$ is $L$-Lipschitz on $\Mcal$, and that we have 
    \begin{align*}
        \exp(-A(\sigma^2,\theta)) \le \frac{1}{\int_{\Mcal}\exp\left(-\frac{(\diam(\Mcal))^2}{2\sigma^2}\right)\diff x} = \frac{\exp\left(\frac{(\diam(\Mcal))^2}{2\sigma^2}\right)}{\int_{\Mcal}1\diff x}=:R
    \end{align*}
    for all $\theta\in\Mcal$.

    Now we let $\gamma:[0,1]\to\Mcal$ be a constant-speed geodesic with $\gamma(0)=x$ and $v:=\gamma'(0)\in\Tan_x(\Mcal)$, and we justify applying dominated convergence in order to differentiate $f$ under the integral along $\gamma$.
    Note that, if $\theta\notin C_x$, then $y\mapsto d^2(y,\theta)$ is differentiable at $x$ and satisfies
    \[
        \nabla_x d^2(x,\theta) = -2\Log_x(\theta)
    \]
    hence
    \begin{align*}
        \lim_{t\to 0}\frac{p_{\theta,\sigma^2}(\gamma(t))-p_{\theta,\sigma^2}(x)}{t}
        &= \left\langle \nabla_x p_{\theta,\sigma^2}(x),v\right\rangle_x \\
        &= \left\langle
            -\frac{1}{2\sigma^2}\nabla_x d^2(x,\theta)\,
            p_{\theta,\sigma^2}(x),\,v
        \right\rangle_x \\
        &= \left\langle
            \sigma^{-2}\Log_x(\theta)\,p_{\theta,\sigma^2}(x),\,v
        \right\rangle_x
    \end{align*}
    by the chain rule.
    Also, use that $a\mapsto e^{-a}$ is $1$-Lipschitz on $[0,\infty)$ to get
    \begin{equation}\label{eqn:ptheat-diff-quotient}
        \begin{split}
            \frac{\left|p_{\theta,\sigma^2}(\gamma(t))-p_{\theta,\sigma^2}(x)\right|}{t}
        &= \frac{\exp(-A(\sigma^2,\theta))}{t}
        \left|
            \exp\left(-\frac{d^2(\gamma(t),\theta)}{2\sigma^2}\right)
            -
            \exp\left(-\frac{d^2(x,\theta)}{2\sigma^2}\right)
        \right| \\
        &\le \frac{\exp(-A(\sigma^2,\theta))}{2\sigma^2\,t}
        \left|d^2(\gamma(t),\theta)-d^2(x,\theta)\right| \\
        &\le \frac{LR}{2\sigma^2\,t}d(\gamma(t),x) \\
        &\le \frac{LR\,\|v\|_x}{2\sigma^2},
        \end{split}
    \end{equation}
    for all $0< t\le 1$.

    The preceding paragraph shows that the difference quotients are dominated by an integrable constant independent of $\theta$, and that they converge for $G$-almost every $\theta$ because $G(C_x)=0$.
    Therefore, we can differentiate under the integral to obtain
    \begin{align*}
        \left\langle \nabla f(x),v\right\rangle_x
        &= \lim_{t\to 0}\frac{f(\gamma(t))-f(x)}{t} \\
        &= \lim_{t\to 0}\int_{\Mcal}\frac{p_{\theta,\sigma^2}(\gamma(t))-p_{\theta,\sigma^2}(x)}{t}\,\diff G(\theta) \\
        &= \int_{\Mcal}
        \left\langle
            \sigma^{-2}\Log_x(\theta)\,p_{\theta,\sigma^2}(x),\,v
        \right\rangle_x
        \diff G(\theta) \\
        &= \left\langle
            \sigma^{-2}\int_{\Mcal}\Log_x(\theta)\,p_{\theta,\sigma^2}(x)\,\diff G(\theta),\,v
        \right\rangle_x.
    \end{align*}
    Since this holds for every $v\in\Tan_x(\Mcal)$, we conclude that $f$ is differentiable at $x$, and
    \[
        \nabla f(x)=\sigma^{-2}\int_{\Mcal}\Log_x(\theta)\,p_{\theta,\sigma^2}(x)\,\diff G(\theta).
    \]
    Dividing by $f(x)$ gives
    \[
        \nabla_x\log f(x)
        =
        \frac{\sigma^{-2}\int_{\Mcal}\Log_x(\theta)\,p_{\theta,\sigma^2}(x)\,\diff G(\theta)}{f(x)}.
    \]
    Lastly, note that $q_x(\theta):=p_{\theta,\sigma^2}(x)/f(x)$ is the density, with respect to $G$, of the conditional distribution of $\Theta$ given $\{X=x\}$; so, rearranging this yields~\eqref{eqn:tweedie}.
\end{proof}

The rest of the results in this subsection rely on some technical preliminaries that we now develop.
First, it will be convenient to parameterize $\Mcal$ by coordinates which are centered at $x\in\Mcal$, usually called \textit{normal coordinates}.
That is, we fix an open domain $D_x\subseteq \Tan_x(\Mcal)$ such that $\Exp_x:D_x\to\Mcal$ is a smooth bijection, possibly up to sets of zero volume form and we parameterize $\Mcal$ by $D_x$; we refer to $D_x$ as an \textit{exponential domain}.
Since $\Mcal$ is compact, there exist $0<r<R$ (depending only on $\Mcal$) such that we have $B_r(0) \subseteq D_x \subseteq B_R(0)$ for all $x\in\Mcal$, where $B_r(0),B_R(0)$ denote the open balls of radii $r,R$, respectively, in $\Tan_x(\Mcal)$.
When we wish to integrate a function $h:\Mcal\to\Rbb$ using normal coordinates at $x\in\Mcal$, we write
\begin{equation}\label{eq:normal-coor}
    \int_{\Mcal}h(y)\diff y = \int_{D_x }h(\Exp_x(v))j_x(v)\diff v
\end{equation}
where $j_x: \Tan_x(\Mcal)\to\Rbb$ is the Jacobian determinant of the map $\Exp_x:D_x\to\Mcal$; it is known that $(x,v)\mapsto j_x(v)$ is smooth (since $(x,v)\mapsto \Exp_x(v)$ is smooth, e.g., \cite[Chapter~5.5.1]{Petersen}) and that $x\mapsto j_x(v)$ satisfies $\lim_{v\to 0}\sup_{x\in \Mcal}|j_x(v) - 1| = 0$ for $\Mcal$ compact \cite[Appendix~B]{PennecRiemNorm}.
(Also note that $\diff y$ represents the volume form of $\Mcal$ on the left side of~\eqref{eq:normal-coor}, whereas $\diff v$ represents the Lebesgue measure of $\Tan_x(\Mcal)$ on the right side.)

First, we show in the following that the Jacobian determinant is uniformly bounded.

\begin{lemma}\label{lem:jac-det-bdd}
If $\Mcal$ is a compact connected Riemannian manifold, then
$\sup_{y\in\Mcal}\sup_{v\in D_y} j_y(v)<\infty$.
\end{lemma}

\begin{proof}
    Recall that the exponential domains can be assumed to satisfy the property that there exists $R>0$ such that $D_x\subseteq B_R(0)$ for all $x\in \Mcal$.
    Then, just use the fact that $\{(x,v): x\in \Mcal, v\in D_x\}$ is pre-compact and $(x,v)\mapsto j_x(v)$ is globally smooth hence uniformly bounded.
\end{proof}

Second, we use normal coordinates to show that the normalizing constant of the Riemannian Gaussian distributions is comparable to the analogous normalizing constant for Euclidean Gaussian distributions.

\begin{lemma}\label{lem:normalizing-const}
    If $\Mcal$ is a compact connected Riemannian manifold, then there exist constants $c_{\Mcal},C_{\Mcal},\varepsilon_{\Mcal}>0$ such that we have
    \begin{equation*}
        c_{\Mcal}\Pbb(\|\Ncal(0,\sigma^2I_m)\|\le \varepsilon_{\Mcal})\le\frac{\exp(A(\sigma^2,\theta))}{(\sqrt{2\pi}\sigma)^m}\le C_{\Mcal}
    \end{equation*}
    for all $\theta\in\Mcal$ and $\sigma^2>0$.
    Consequently, we have
    \begin{equation*}
        \frac{\exp(A(\sigma^2,\theta))}{(\sqrt{2\pi}\sigma)^m}=\Theta(1)\qquad \textnormal{ as }\qquad \sigma^2\to 0,
    \end{equation*}
    uniformly in $\theta\in\Mcal$ and
    \begin{equation*}
        c_{\Mcal}'\le\frac{\exp(A(\sigma^2,\theta))}{(\sqrt{2\pi}\sigma)^m}\le C_{\Mcal} \qquad \textnormal{ for all }\qquad \sigma^2\le (\diam(\Mcal))^2
    \end{equation*}
    for all $\theta\in\Mcal$, where $c_{\Mcal}' := c_{\Mcal}\Pbb(\|\Ncal(0,(\diam(\Mcal))^2I_m)\|\le \varepsilon_{\Mcal})>0$.
\end{lemma}

\begin{proof}
    First, use normal coordinates~\eqref{eq:normal-coor} and $\|\Log_x\theta\|_x =d(x,\theta)$ to write
    \begin{equation*}
        \exp(A(\sigma^2,\theta))= \int_{\Mcal}\exp\left(-\frac{d^2(x,\theta)}{2\sigma^2}\right)\diff x = \int_{D_\theta}\exp\left(-\frac{\|v\|^2_{\theta}}{2\sigma^2}\right)j_{\theta}(v)\diff v.
    \end{equation*}
    Then, for the upper bound, use Lemma~\ref{lem:jac-det-bdd} to get $C:=\sup_{\theta\in\Mcal}\sup_{v\in D_\theta} j_\theta(v)<\infty$, hence
    \begin{align*}
        \exp(A(\sigma^2,\theta)) &= \int_{D_\theta}\exp\left(-\frac{\|v\|^2_{\theta}}{2\sigma^2}\right)j_{\theta}(v)\diff v \\
        &\le C\int_{D_\theta}\exp\left(-\frac{\|v\|^2_{\theta}}{2\sigma^2}\right)\diff v \\
        &\le C\int_{\Tan_{\theta}(\Mcal)}\exp\left(-\frac{\|v\|^2_{\theta}}{2\sigma^2}\right)\diff v \\
        &= C(\sqrt{2\pi}\sigma)^m.
    \end{align*}
    For the lower bound, get some sufficiently small $\varepsilon>0$ such that we have $B_{\varepsilon}(0)\subseteq D_\theta$ for all $\theta\in\Mcal$ and, using continuity of $(\theta,v)\mapsto j_\theta(v)$,  also $c:=\inf_{\theta\in\Mcal}\inf_{v\in B_{\varepsilon}(0)}j_\theta(v)>0$; thus we can bound
    \begin{align*}
        \exp(A(\sigma^2,\theta)) &= \int_{D_\theta}\exp\left(-\frac{\|v\|^2_{\theta}}{2\sigma^2}\right)j_{\theta}(v)\diff v \\
        &\ge \int_{B_{\varepsilon}(0)}\exp\left(-\frac{\|v\|^2_{\theta}}{2\sigma^2}\right)j_{\theta}(v)\diff v \\
        &\ge c\int_{B_{\varepsilon}(0)}\exp\left(-\frac{\|v\|^2_{\theta}}{2\sigma^2}\right)\diff v \\
        &= c(\sqrt{2\pi}\sigma)^m\Pbb(\|\Ncal(0,\sigma^2I_m)\|\le \varepsilon),
    \end{align*}
    and this establishes the first claim.
    To obtain the second claim, note that 
    \begin{equation*}
        \Pbb(\|\Ncal(0,\sigma^2I_m)\|\le \varepsilon)\to 1\qquad \textnormal{ as }\qquad \sigma^2\to 0,
    \end{equation*}
    and, to obtain the third claim, note that
    \begin{equation*}
        \Pbb(\|\Ncal(0,\sigma^2I_m)\|\le \varepsilon)\ge \Pbb(\|\Ncal(0,(\diam(\Mcal))^2I_m)\|\le \varepsilon)\qquad \textnormal{ for }\qquad \sigma^2\le (\diam(\Mcal))^2,
    \end{equation*}
    This finishes the proof.    
\end{proof}

Now we prove the results stated in the main body.

\begin{proof}[Proof of Lemma~\ref{lem:f-lower-bd}]
    The proof is similar to that of Lemma~\ref{lem:normalizing-const}.
    Use normal coordinates~\eqref{eq:normal-coor} and $\|\Log_x\theta\|_x =d(x,\theta)$ to rewrite the integral expression for $f(x)$, then Lemma~\ref{lem:normalizing-const} to bound
    \begin{align*}
        f(x) &= \int_{\Mcal}p_{\theta,\sigma^2}(x)g(\theta)\diff \theta \\
        &=\int_{\Mcal}\exp\left(-\frac{d^2(x,\theta)}{2\sigma^2}-A(\sigma^2,\theta)\right)g(\theta)\diff \theta \\
        & = \int_{D_x}\exp\left(-\frac{\|v\|_x^2}{2\sigma^2}-A(\sigma^2,\Exp_x(v))\right)g(\Exp_x(v))j_x(v)\diff v \\
        & \asymp_{\Mcal} \frac{1}{(\sqrt{2\pi}\sigma)^m}\int_{D_x}\exp\left(-\frac{\|v\|_x^2}{2\sigma^2}\right)g(\Exp_x(v))j_x(v)\diff v,
    \end{align*}
    where $\asymp_{\Mcal}$ denotes equality up to constants that depend only on $\Mcal$, and $j_x$ is the usual Jacobian determinant.
    For the upper bound, use Lemma~\ref{lem:jac-det-bdd} to get $C:=\sup_{x\in\Mcal}\sup_{v\in D_x} j_x(v)<\infty$, hence
    \begin{align*}
        f(x) &\asymp_{\Mcal} \frac{1}{(\sqrt{2\pi}\sigma)^m}\int_{D_x}\exp\left(-\frac{\|v\|_x^2}{2\sigma^2}\right)g(\Exp_x(v))j_x(v)\diff v \\
        &\lesssim_{\Mcal} \left(\underset{\theta\in\Mcal}{\textnormal{ess}\sup}\,g(\theta)\right)\frac{1}{(\sqrt{2\pi}\sigma)^m}\int_{D_x}\exp\left(-\frac{\|v\|_x^2}{2\sigma^2}\right)j_x(v)\diff v \\
        &\le C\left(\underset{\theta\in\Mcal}{\textnormal{ess}\sup}\,g(\theta)\right)\frac{1}{(\sqrt{2\pi}\sigma)^m}\int_{D_x}\exp\left(-\frac{\|v\|_x^2}{2\sigma^2}\right)\diff v \\
        &\le C\left(\underset{\theta\in\Mcal}{\textnormal{ess}\sup}\,g(\theta)\right)\frac{1}{(\sqrt{2\pi}\sigma)^m}\int_{\Tan_x(\Mcal)}\exp\left(-\frac{\|v\|_x^2}{2\sigma^2}\right)\diff v \\
        &\le C\,\underset{\theta\in\Mcal}{\textnormal{ess}\sup}\,g(\theta).
    \end{align*}
    For the lower bound, get some sufficiently small $\varepsilon>0$ such that we have $B_{\varepsilon}(0)\subseteq D_x$ for all $x\in\Mcal$ and, using continuity of $(x,v)\mapsto j_x(v)$,  also $c:=\inf_{x\in\Mcal}\inf_{v\in B_{\varepsilon}(0)}j_x(v)>0$; then,
    \begin{align*}
        f(x) &\asymp_{\Mcal} \frac{1}{(\sqrt{2\pi}\sigma)^m}\int_{D_x}\exp\left(-\frac{\|v\|_x^2}{2\sigma^2}\right)g(\Exp_x(v))j_x(v)\diff v \\
        &\gtrsim_{\Mcal} \left(\underset{\theta\in\Mcal}{\textnormal{ess}\inf}\,g(\theta)\right)\frac{1}{(\sqrt{2\pi}\sigma)^m}\int_{D_x}\exp\left(-\frac{\|v\|_x^2}{2\sigma^2}\right)j_x(v)\diff v \\
        &\ge \left(\underset{\theta\in\Mcal}{\textnormal{ess}\inf}\,g(\theta)\right)\frac{1}{(\sqrt{2\pi}\sigma)^m}\int_{B_{\varepsilon}(0)}\exp\left(-\frac{\|v\|_x^2}{2\sigma^2}\right)j_x(v)\diff v \\
        &\ge c\left(\underset{\theta\in\Mcal}{\textnormal{ess}\inf}\,g(\theta)\right)\frac{1}{(\sqrt{2\pi}\sigma)^m}\int_{B_{\varepsilon}(0)}\exp\left(-\frac{\|v\|_x^2}{2\sigma^2}\right)\diff v \\
        &\ge c\left(\underset{\theta\in\Mcal}{\textnormal{ess}\inf}\,g(\theta)\right)\Pbb(\|\Ncal(0,\sigma^2I_m)\|\le \varepsilon) \\
        &\ge c\left(\underset{\theta\in\Mcal}{\textnormal{ess}\inf}\,g(\theta)\right)\Pbb(\|\Ncal(0,(\diam(\Mcal))^2I_m)\|\le \varepsilon) \\
        &\gtrsim_{\Mcal} {\textnormal{ess}\inf}_{\theta\in\Mcal}g(\theta).
    \end{align*}
    This completes the proof.
\end{proof}

\begin{proof}[Proof of Lemma~\ref{lem:f-diff-1}]
By Lemma~\ref{lem:Tweedie}, $f$ is differentiable and we may differentiate under the integral once to compute 
    \begin{equation*}
        \nabla f(x) = \sigma^{-2}\int_{\Mcal}\Log_x\theta \,p_{\theta,\sigma^2}(x)g(\theta)\diff\theta.
    \end{equation*}
    Now use normal coordinates~\eqref{eq:normal-coor}, $\|\Log_x\theta\|_x =d(x,\theta)$, Lemma~\ref{lem:normalizing-const}, and $C:=\sup_{x\in\Mcal}\sup_{v\in D_x}j_x(v)<\infty$ from Lemma~\ref{lem:jac-det-bdd} to get:
    \begin{align*}
        \|\nabla f(x)\|_x &= \left\|\sigma^{-2}\int_{\Mcal}\Log_x\theta \,p_{\theta,\sigma^2}(x)g(\theta)\diff\theta\right\|_x \\
        &\le \sigma^{-2}\int_{\Mcal}\|\Log_x\theta\|_x \,p_{\theta,\sigma^2}(x)g(\theta)\diff\theta \\
        &= \sigma^{-2}\int_{D_x}\|v\|_x\exp\left(-\frac{\|v\|_x^2}{2\sigma^2}-A(\sigma^2,\Exp_x(v))\right)g(\Exp_x(v))j_x(v)\diff v \\
        &\lesssim_{\Mcal} \frac{\sigma^{-2}}{(\sqrt{2\pi}\sigma)^m}\int_{D_x}\|v\|_x\exp\left(-\frac{\|v\|_x^2}{2\sigma^2}\right)g(\Exp_x(v))j_x(v)\diff v \\
        &\le \frac{C\|g\|_{L^{\infty}(\Mcal)}\sigma^{-2}}{(\sqrt{2\pi}\sigma)^m}\int_{D_x}\|v\|_x\exp\left(-\frac{\|v\|_x^2}{2\sigma^2}\right)\diff v \\
        &\le \frac{C\|g\|_{L^{\infty}(\Mcal)}\sigma^{-2}}{(\sqrt{2\pi}\sigma)^m}\int_{\Tan_x(\Mcal)}\|v\|_x\exp\left(-\frac{\|v\|_x^2}{2\sigma^2}\right)\diff v.
    \end{align*}
    Lastly, note that, if $Z$ denotes a standard Gaussian random variable in $\Rbb^m$ (i.e., with mean 0 and covariance matrix $I_m$), then we have
    \begin{equation*}
        \frac{1}{(\sqrt{2\pi}\sigma)^m}\int_{\Tan_x(\Mcal)}\|v\|_x\exp\left(-\frac{\|v\|_x^2}{2\sigma^2}\right)\diff v = \sigma\,\Ebb\left[\|Z\|\right]
    \end{equation*}
    where the expectation depends only on $m$.
    This finishes the proof.
    \end{proof}

\subsection{Proofs from Subsection~\ref{subsec:oracle}}\label{subsec:proofs-oracle}

Next we give proofs for our results concerning the risks of the oracle denoisers $\delta_{\N},\delta_{\B}$, and $\delta_{\T}$.

Our first goal is to establish the risk of the na\"ive denoiser $\delta_{\N}$.
This requires the following general result that, in the limit $\sigma^2\to 0$, many functions of the geodesic distance can be calculated by expectations of a corresponding Euclidean Gaussian distribution.
We remind the reader that the symbol $\sim$ refers to the statement that the ratio of expressions converges to one in the limit.

\begin{lemma}\label{lem:integral-limit}
If $\Mcal$ is a compact connected Riemannian manifold, then for any $k\in\Nbb$ we have
\begin{equation*}
\int_{\Mcal}d^k(x,\theta)\,p_{\theta,\sigma^2}(x)\diff x
\sim
\Ebb[\|Z\|^k]
\end{equation*}
as $\sigma^2\to 0$, uniformly in $\theta\in\Mcal$, where
$Z\sim \Ncal(0,\sigma^2I_m)$.
\end{lemma}

\begin{proof}
    First, recall from \cite[Theorem~5]{PennecRiemNorm} that we have
    \begin{equation}\label{eqn:eA-limit}
        \lim_{\sigma^2\to 0}\sup_{\theta\in\Mcal}\frac{\exp(A(\sigma^2,\theta))}{(\sqrt{2\pi}\sigma)^m} = 1,
    \end{equation}
    which is similar to Lemma~\ref{lem:normalizing-const} with the sharp constant 1, but only asymptotically as $\sigma^2\to 0$.
    Now fix $0<\varepsilon<1$, and use the uniform convergence $\lim_{v\to 0}\sup_{\theta\in \Mcal}|j_{\theta}(v)-1|= 0$ to get some $r>0$ such that we have $|j_\theta(v)-1| < \varepsilon$ for all $\theta\in\Mcal$ and $v\in B_r(0)\subseteq\Tan_\theta(\Mcal)$; by taking $r>0$ smaller if necessary, using the fact that $\Mcal$ being compact implies a positive injectivity radius, we may also assume $B_r(0)\subseteq D_{\theta}$ for all $\theta\in\Mcal$.
    Now, convert to normal coordinates~\eqref{eq:normal-coor} at $\theta$ and use $d(x,\theta)=\|\Log_x\theta\|_x$ to get
    \begin{equation}\label{eqn:integral-asymp}
        \begin{split}
        \int_{\Mcal}d^k(x,\theta)p_{\theta,\sigma^2}(x)\diff x &=\int_{\Mcal}d^k(x,\theta)\exp\left(-\frac{d^2(x,\theta)}{2\sigma^2}-A(\sigma^2,\theta)\right)\diff x\\
        &\sim\frac{1}{(\sqrt{2\pi}\sigma)^m}\int_{\Mcal}d^k(x,\theta)\exp\left(-\frac{d^2(x,\theta)}{2\sigma^2}\right)\diff x \\
        &=\frac{1}{(\sqrt{2\pi}\sigma)^m}\int_{D_{\theta}}\|v\|^k_{\theta}\exp\left(-\frac{\|v\|^2_{\theta}}{2\sigma^2}\right)j_{\theta}(v)\diff v \\
        &=\frac{1}{(\sqrt{2\pi}\sigma)^m}\int_{B_r(0)}\|v\|^k_{\theta}\exp\left(-\frac{\|v\|^2_{\theta}}{2\sigma^2}\right)j_{\theta}(v)\diff v \\
        &\qquad+\frac{1}{(\sqrt{2\pi}\sigma)^m}\int_{D_{\theta}\setminus B_r(0)}\|v\|^k_{\theta}\exp\left(-\frac{\|v\|^2_{\theta}}{2\sigma^2}\right)j_{\theta}(v)\diff v 
        \end{split}
    \end{equation}
    uniformly in $\theta$.
    Next, note
    \begin{equation*}
        \frac{1}{(\sqrt{2\pi}\sigma)^m}\int_{\Tan_{\theta}(\Mcal)\setminus B_r(0)}\|v\|^k_{\theta}\exp\left(-\frac{\|v\|^2_{\theta}}{2\sigma^2}\right)\diff v  \lesssim \exp\left(-\frac{r^2}{2\sigma^2}\right)
    \end{equation*}
    which vanishes as $\sigma^2\to 0$.
    In particular, combining with Lemma~\ref{lem:jac-det-bdd} implies that the second term of \eqref{eqn:integral-asymp} vanishes, i.e.
    \begin{align*}
        &\frac{1}{(\sqrt{2\pi}\sigma)^m}\int_{D_{\theta}\setminus B_r(0)}\|v\|^k_{\theta}\exp\left(-\frac{\|v\|^2_{\theta}}{2\sigma^2}\right)j_{\theta}(v)\diff v \\
        &\lesssim_{\Mcal} \frac{1}{(\sqrt{2\pi}\sigma)^m}\int_{\Tan_{\theta}(\Mcal)\setminus B_r(0)}\|v\|^k_{\theta}\exp\left(-\frac{\|v\|^2_{\theta}}{2\sigma^2}\right)\diff v
    \end{align*}
    For the first term of \eqref{eqn:integral-asymp}, note that we have the upper bound    
    \begin{align*}
        \frac{1}{(\sqrt{2\pi}\sigma)^m}&\int_{B_r(0)}\|v\|^k_{\theta}\exp\left(-\frac{\|v\|^2_{\theta}}{2\sigma^2}\right)j_{\theta}(v)\diff v \\
        &\le \frac{1+\varepsilon}{(\sqrt{2\pi}\sigma)^m}\int_{B_r(0)}\|v\|^k_{\theta}\exp\left(-\frac{\|v\|^2_{\theta}}{2\sigma^2}\right)\diff v \\
        &= \frac{1+\varepsilon}{(\sqrt{2\pi}\sigma)^m}\int_{\Tan_{\theta}(\Mcal)}\|v\|_{\theta}^k\exp\left(-\frac{\|v\|^2_{\theta}}{2\sigma^2}\right)\diff v \\
        &\qquad- \frac{1+\varepsilon}{(\sqrt{2\pi}\sigma)^m}\int_{\Tan_{\theta}(\Mcal)\setminus B_{r}(0)}\|v\|^k_{\theta}\exp\left(-\frac{\|v\|^2_{\theta}}{2\sigma^2}\right)\diff v \\
        &= (1+\varepsilon)\Ebb\left[\|Z\|^k\right] - (1+\varepsilon)\Ebb\left[\|Z\|^k\ind\{\|Z\|\ge r\}\right].
    \end{align*}
    Since $\Ebb[\|Z\|^k\ind\{\|Z\|\ge r\}]=o(\Ebb[\|Z\|^k])$ as $\sigma^2\to 0$, we have shown
    \begin{equation*}
        \limsup_{\sigma^2\to 0}\frac{\frac{1}{(\sqrt{2\pi}\sigma)^m}\int_{B_r(0)}\|v\|^k_{\theta}\exp\left(-\frac{\|v\|^2_{\theta}}{2\sigma^2}\right)j_{\theta}(v)\diff v}{\Ebb[\|Z\|^k]} \le 1+\varepsilon.
    \end{equation*}
    Similarly, the first term of \eqref{eqn:integral-asymp} satisfies the lower bound
    \begin{align*}
        \frac{1}{(\sqrt{2\pi}\sigma)^m}&\int_{B_r(0)}\|v\|^k_{\theta}\exp\left(-\frac{\|v\|^2_{\theta}}{2\sigma^2}\right)j_{\theta}(v)\diff v \\
        &\ge \frac{1-\varepsilon}{(\sqrt{2\pi}\sigma)^m}\int_{B_r(0)}\|v\|^k_{\theta}\exp\left(-\frac{\|v\|^2_{\theta}}{2\sigma^2}\right)\diff v \\
        &= \frac{1-\varepsilon}{(\sqrt{2\pi}\sigma)^m}\int_{\Tan_{\theta}(\Mcal)}\|v\|^k_{\theta}\exp\left(-\frac{\|v\|^2_{\theta}}{2\sigma^2}\right)\diff v \\
        &\qquad- \frac{1-\varepsilon}{(\sqrt{2\pi}\sigma)^m}\int_{\Tan_{\theta}(\Mcal)\setminus B_{r}(0)}\|v\|^k_{\theta}\exp\left(-\frac{\|v\|^2_{\theta}}{2\sigma^2}\right)\diff v \\
        &= (1-\varepsilon)\Ebb\left[\|Z\|^k\right] - (1-\varepsilon)\Ebb\left[\|Z\|^k\ind\{\|Z\|\ge r\}\right].
    \end{align*}
    hence    
     \begin{equation*}
        \liminf_{\sigma^2\to 0}\frac{\frac{1}{(\sqrt{2\pi}\sigma)^m}\int_{B_r(0)}\|v\|^k_{\theta}\exp\left(-\frac{\|v\|^2_{\theta}}{2\sigma^2}\right)j_{\theta}(v)\diff v}{\Ebb[\|Z\|^k]} \ge 1-\varepsilon.
    \end{equation*}
    As $0<\varepsilon<1$ was arbitrary, we have shown
    \begin{equation*}
        \lim_{\sigma^2\to 0}\frac{\frac{1}{(\sqrt{2\pi}\sigma)^m}\int_{D_{\theta}}\|v\|^k_{\theta}\exp\left(-\frac{\|v\|^2_{\theta}}{2\sigma^2}\right)j_{\theta}(v)\diff v}{\Ebb[\|Z\|^k]} =1,
    \end{equation*}
    as claimed.
\end{proof}

Next, we use this to calculate the risk of $\delta_{\N}$, as follows.

\begin{proof}[Proof of Proposition~\ref{prop:naive-risk}]
    Immediate from Lemma~\ref{lem:integral-limit} with $k=2$.
\end{proof}

Next we turn to the comparisons of the oracle risks of $\delta_{\B}$ and $\delta_{\T}$.
As we discussed in the main body, this requires some auxiliary results that we now prove.

\begin{proof}[Proof of Proposition~\ref{prop:mux-sigmax-estimates}]
For convenience, we introduce the notation
    \begin{equation}\label{eq:V_x}
        V_x := \Log_x \Theta, \qquad \mu_x := \Ebb[V_x \mid X=x], \qquad \Sigma_x := \Cov(V_x \mid X=x),
    \end{equation}
    for all $x\in\Mcal$.
    Also, we introduce normal coordinates in order to make some explicit calculations.
    Let us write $D_x$ for an exponential domain at $x\in\Mcal$, and recall that the compactness of $\Mcal$ implies that there exists $r>0$ such that we have $B_r(0)\subseteq D_x$ for all $x\in\Mcal$.

    First, let us rewrite the integral expression for $\mu_x$ in normal coordinates using~\eqref{eq:normal-coor} and $\|\Log_x\theta\|_x = d(x,\theta)$, then use Lemma~\ref{lem:normalizing-const} to bound:
    \begin{equation}\label{eqn:mux-1}
        \begin{split}
        \mu_x &= \frac{\int_{\Mcal}\Log_x\theta \;p_{\theta,\sigma^2}(x)\diff G(\theta)}{\int_{\Mcal} p_{\theta,\sigma^2}(x)\diff G(\theta)} \\
        &= \frac{\int_{D_x}v \exp\left(-\frac{\|v\|_x^2}{2\sigma^2}-A(\sigma^2,\Exp_x(v))\right)g(\Exp_x(v))j_x(v)\diff v}{\int_{D_x} \exp\left(-\frac{\|v\|_x^2}{2\sigma^2}-A(\sigma^2,\Exp_x(v))\right)g(\Exp_x(v))j_x(v)\diff v} \\
        &\asymp_{\Mcal} \frac{(\sqrt{2\pi}\sigma)^{-m}\int_{D_x}v \exp\left(-\frac{\|v\|_x^2}{2\sigma^2}\right)g(\Exp_x(v))j_x(v)\diff v}{(\sqrt{2\pi}\sigma)^{-m}\int_{D_x} \exp\left(-\frac{\|v\|_x^2}{2\sigma^2}\right)g(\Exp_x(v))j_x(v)\diff v},
        \end{split}
    \end{equation}
    where $\asymp_{\Mcal}$ denotes equality up to constants that depend only on $\Mcal$ and where $j_x$ is the usual Jacobian determinant. Now let us control the numerator and denominator of the right side separately.

    For the numerator, the idea is that the fast decay of the integrand allows us to replace the domain $D_x$ with the smaller domain $B_r(0)$ while only incurring an exponentially small error.
    More precisely, we write
    \begin{equation}\label{eqn:mux-2}
        \begin{split}
        &(\sqrt{2\pi}\sigma)^{-m}\int_{D_x}v \exp\left(-\frac{\|v\|_x^2}{2\sigma^2}\right)g(\Exp_x(v))j_x(v)\diff v \\
        &= (\sqrt{2\pi}\sigma)^{-m}\int_{B_r(0)}v \exp\left(-\frac{\|v\|_x^2}{2\sigma^2}\right)g(\Exp_x(v))j_x(v)\diff v \\
        &\qquad \quad + (\sqrt{2\pi}\sigma)^{-m}\int_{D_x\setminus B_r(0)}v \exp\left(-\frac{\|v\|_x^2}{2\sigma^2}\right)g(\Exp_x(v))j_x(v)\diff v,
        \end{split}
    \end{equation}
    and we will show that the first term is $O(\sigma^2)$ and the second term is exponentially small hence also $O(\sigma^2)$.
    For the first term, develop a Taylor expansion with a first-order remainder
\[
g(\Exp_x(v))\,j_x(v)=g(x)+r(x,v),
\]
where $r$ is continuous and satisfies the uniform bound
\[
\sup_{x\in\Mcal}\sup_{\|v\|_x\le r}\frac{|r(x,v)|}{\|v\|_x}<\infty,
\]
which holds since $g$ is continuously differentiable, $j_x$ is smooth in $(x,v)$ on $\{\|v\|_x\le r\}$, and $\Mcal$ is compact.
    Then compute:
    \begin{align*}
        &(\sqrt{2\pi}\sigma)^{-m}\int_{B_r(0)}v \exp\left(-\frac{\|v\|_x^2}{2\sigma^2}\right)g(\Exp_x(v))j_x(v)\diff v \\
        &= g(x)(\sqrt{2\pi}\sigma)^{-m}\int_{B_r(0)}v \exp\left(-\frac{\|v\|_x^2}{2\sigma^2}\right)\diff v + (\sqrt{2\pi}\sigma)^{-m}\int_{B_r(0)} v r(x,v)\exp\left(-\frac{\|v\|_x^2}{2\sigma^2}\right)\diff v.
    \end{align*}
Note that the first term vanishes by symmetry while the second term is $O(\sigma^{2})$ uniformly in $x\in\Mcal$.
    For the second term of \eqref{eqn:mux-2}, use compactness of $\Mcal$, uniform boundedness of the exponential domains, and continuity of $(x,v)\mapsto g(\Exp_x(v))j_x(v)$ to get $C:=\sup_{x\in\Mcal}\sup_{v\in D_x}g(\Exp_x(v))j_x(v)<\infty$, then use this to bound
    \begin{equation}
        \begin{split}
            &(\sqrt{2\pi}\sigma)^{-m}\left\|\int_{D_x\setminus B_r(0)}v \exp\left(-\frac{\|v\|_x^2}{2\sigma^2}\right)g(\Exp_x(v))j_x(v)\diff v\right\| \\
            &\le (\sqrt{2\pi}\sigma)^{-m}\int_{D_x\setminus B_r(0)}\|v\|_x \exp\left(-\frac{\|v\|_x^2}{2\sigma^2}\right)g(\Exp_x(v))j_x(v)\diff v \\
        &\le C(\sqrt{2\pi}\sigma)^{-m}\int_{\Tan_x(\Mcal)\setminus B_r(0)}\|v\|_x \exp\left(-\frac{\|v\|_x^2}{2\sigma^2}\right)\diff v \\
        &\lesssim C(\sqrt{2\pi}\sigma)^{-m}\sigma\exp\left(-\frac{r^2}{\sigma^2}\right),
        \end{split}
    \end{equation}
    which is certainly $O(\sigma^{2})$, uniformly in $x\in\Mcal$.
    
    Lastly, we turn to the denominator of \eqref{eqn:mux-1}.
    To control this, use continuity of $(x,v)\mapsto g(\Exp_x(v))j_x(v)$ to get some sufficiently small $0<\varepsilon<r$ such that we have $$c:=\inf_{x\in\Mcal}\inf_{v\in B_{\varepsilon}(0)}g(\Exp_x(v))j_x(v)>0.$$
    Then, compute:
    \begin{equation}\label{eqn:denom}
        \begin{split}
        (\sqrt{2\pi}\sigma)^{-m}\int_{D_x} \exp\left(-\frac{\|v\|_x^2}{2\sigma^2}\right)g(\Exp_x(v))j_x(v)\diff v &\ge c(\sqrt{2\pi}\sigma)^{-m}\int_{B_{\varepsilon}(0)} \exp\left(-\frac{\|v\|_x^2}{2\sigma^2}\right)\diff v \\
        &= c  \cdot\Pbb(\|\Ncal(0,\sigma^2I_m)\|\le \varepsilon).
        \end{split}
    \end{equation}
    Note that the probability converges to 1 as $\sigma^2\to 0$, hence the right hand side is $\Omega(1)$, uniformly in $x\in\Mcal$.
    By combining this calculation with that of the numerator, we have shown
    \begin{equation}\label{eqn:mu-x}
        \sup_{x\in\Mcal}\|\mu_x\| = \frac{O(\sigma^{2})}{\Omega(1)} = O(\sigma^2),
    \end{equation}
    as claimed.

    Now we turn to the second claim.
    To prove it, again use normal coordinates~\eqref{eq:normal-coor} and $\|\Log_x\theta\|_x= d(x,\theta)$ to write
   \begin{equation}\label{eqn:sigmax-1}
        \begin{split}
        \Sigma_x &= \frac{\int_{\Mcal}\Log_x\theta(\Log_x\theta)^{\top} \;p_{\theta,\sigma^2}(x)\diff G(\theta)}{\int_{\Mcal} p_{\theta,\sigma^2}(x)\diff G(\theta)} - \mu_x\mu_x^{\top} \\
        &= \frac{(\sqrt{2\pi}\sigma)^{-m}\int_{D_x}vv^{\top} \exp\left(-\frac{\|v\|_x^2}{2\sigma^2}\right)g(\Exp_x(v))j_x(v)\diff v}{(\sqrt{2\pi}\sigma)^{-m}\int_{D_x} \exp\left(-\frac{\|v\|_x^2}{2\sigma^2}\right)g(\Exp_x(v))j_x(v)\diff v} - \mu_x\mu_x^{\top}.
        \end{split}
    \end{equation}
    Note that \eqref{eqn:mu-x} establishes that the subtracted term is $O(\sigma^4)$, so it suffices to show
    \begin{equation*}
        \frac{(\sqrt{2\pi}\sigma)^{-m}\int_{D_x}vv^{\top} \exp\left(-\frac{\|v\|_x^2}{2\sigma^2}\right)g(\Exp_x(v))j_x(v)\diff v}{(\sqrt{2\pi}\sigma)^{-m}\int_{D_x} \exp\left(-\frac{\|v\|_x^2}{2\sigma^2}\right)g(\Exp_x(v))j_x(v)\diff v} = O(\sigma^2).
    \end{equation*}
    We already have~\eqref{eqn:denom} showing that the denominator is $\Omega(1)$, so it suffices to focus on the numerator.
    As before, we use the constant $C$ to bound:
    \begin{equation*}
        \begin{split}
            &\left\|(\sqrt{2\pi}\sigma)^{-m}\int_{D_x}vv^{\top} \exp\left(-\frac{\|v\|_x^2}{2\sigma^2}\right)g(\Exp_x(v))j_x(v)\diff v\right\| \\
            &\le (\sqrt{2\pi}\sigma)^{-m}\int_{D_x}\|v\|_x^2 \exp\left(-\frac{\|v\|_x^2}{2\sigma^2}\right)g(\Exp_x(v))j_x(v)\diff v \\
        &\le C(\sqrt{2\pi}\sigma)^{-m}\int_{D_x}\|v\|_x^2 \exp\left(-\frac{\|v\|_x^2}{2\sigma^2}\right)\diff v \\
        &\le C(\sqrt{2\pi}\sigma)^{-m}\int_{\Tan_x(\Mcal)}\|v\|_x^2 \exp\left(-\frac{\|v\|_x^2}{2\sigma^2}\right)\diff v \\
        &\le Cm\sigma^2.
        \end{split}
    \end{equation*}
    Thus, we have shown
    \begin{equation}\label{eqn:sigma-x}
        \sup_{x\in\Mcal}\|\Sigma_x\| \le \frac{O(\sigma^{2})}{\Omega(1)} + O(\sigma^4)  = O(\sigma^2)
    \end{equation}
    as claimed.
\end{proof}

Our next result involves a Taylor expansion of the squared geodesic distance, which requires us to introduce some concepts related to the curvature of $\Mcal$.
Recall  (see \cite[equation~(14.1)]{VillaniOldAndNew}) that if $K_x(v,w)\in\Rbb$ denotes the \textit{sectional curvature} of the plane spanned by $v,w\in \Tan_x(\Mcal)$ at the point $x\in\Mcal$ and, then we have
\begin{equation}\label{eqn:sq-dist-Taylor}
    d^2(\Exp_x(v),\Exp_x(w)) = \|v-w\|_x^2 - \frac{1}{3} K_x(v,w)(\|v\|^2_x\|w\|_x^2-\langle v,w\rangle_x^2) + O(\|v\|_x^5 + \|w\|_x^5)
\end{equation}
as $\|v\|_x,\|w\|_x\to 0$, where the remainder is a continuous function of $x\in\Mcal$.

\medskip

\begin{proof}[Proof of Proposition~\ref{prop:wB-estimates}]
    Let us write $\mathcal{W}$ for the vector space of all vector fields $\{w_x\}_{x\in\Mcal}$ on $\Mcal$ (i.e., $w_x\in\Tan_x(\Mcal)$ for all $x\in\Mcal$) satisfying $\int_{\Mcal}\|w_x\|_x^2\diff x<\infty$, which is a Hilbert space when endowed with the inner product $\langle w,w'\rangle :=\int_{\Mcal}\langle w_x,w_x'\rangle_x\diff x$.
    We observe that any denoiser $\delta:\Mcal\to\Mcal$ satisfying $\delta(x)\notin C_x$ for all $x\in\Mcal$ can be written as $\delta = \delta_w$ for some $\{w_x\}_{x\in\Mcal}\in\mathcal{W}$, where $\delta_w(x):=\Exp_x(w_x)$; indeed, just set $w_x:=\Log_x(\delta(x))$.
    Consequently, we may parameterize all such denoisers by elements of $\mathcal{W}$.
    In fact, if we write $R>0$ for a constant depending only on $\Mcal$ such that we have $D_x\subseteq B_R(0)$ for all $x\in\Mcal$, then all such denoisers $\delta$ can be represented by some $w\in\mathcal{W}$ satisfying the stronger condition $\sup_{x\in\Mcal}\|w_x\|
    _x\le R$.
    
    Now let $w_{\B}\in\mathcal{W}$ denote the vector field corresponding to the Bayes denoiser $\delta_{\B}$, and let us explore the first-order optimality conditions for $w_{\B}$.
    For convenience, we adopt the notation from the proof of Proposition~\ref{prop:mux-sigmax-estimates}.
    Recall that $w_{\B}$ is, by definition a minimizer of the conditional risk which we can easily calculate; fix $x\in\Mcal$, and use the Taylor series \eqref{eqn:sq-dist-Taylor} to get the following, where we write $\|\cdot\|$ in place of $\|\cdot\|_x$ in order to keep the notation simpler:
    \begin{equation}\label{eqn:conditional-risk}
    \begin{split}
        &\Ebb[d^2(\Theta,\delta_w(X))\,|\,X=x] \\
        &= \Ebb\left[ { d^2(\Exp_x (V_x),\Exp_X(w_x))} \mid X=x\right]\\
        & = \Ebb\left[\lVert V_x-w_x\rVert^2 - \frac{K_x(V_x,w_x)}{3}\left(\lVert V_x\rVert^2 \lVert w_x\rVert^2 - \langle V_x, w_x\rangle_{x}^2\right)  + O\!\left(\lVert V_x\rVert^5+\lVert  w_x\rVert^5\right)\mid X=x\right] \\
        &= \trace(\Sigma_x) + \|\mu_x- w_x\|^2 \\
        &\qquad- \frac{\Ebb[K_x(V_x,w_x)\,|\,X=x]}{3}\left(\| w_x\|^2\trace(\Sigma_x)  - \langle w_x,\Sigma_x w_x\rangle + \|w_x\|^2\|\mu_x\|^2 - \langle w_x,\mu_x\rangle^2\right) \\
        &\qquad+ \Ebb[O\!\left(\lVert V_x\rVert^5+\lVert  w_x\rVert^5\right)\,|\,X=x].
    \end{split}
\end{equation}
    In the last equality, we simply expanded all squared norms and plugged in the definitions in~\eqref{eq:V_x}.
    Because on a Riemannian manifold all squared distances and Taylor expansions thereof are smooth away from their cut loci, we can take the gradient of the expression above with respect to $w$, and the result must vanish at $w=w_{\B}$.
    Thus we obtain
    \begin{align*}
        0 &= 2(w_{\B x}-\mu_x) - \frac{2\Ebb[K_x(V_x,w_{\B x})\,|\,X=x]}{3}\left( w_{\B x}\trace(\Sigma_x)  - \Sigma_x w_{\B x} + w_{\B x}\|\mu_x\|^2 - \mu_x\langle w_{\B x},\mu_x\rangle\right) \\
        &+O\!\left(\lVert  w_{\B x}\rVert^4\right)-\frac{2\Ebb[\nabla K_x(V_x,w_{\B x})\,|\,X=x]}{3}\left(\| w_x\|^2\trace(\Sigma_x)  - \langle w_x,\Sigma_x w_x\rangle + \|w_x\|^2\|\mu_x\|^2 - \langle w_x,\mu_x\rangle^2\right),
    \end{align*}
    where $\nabla K_x(v,w)$ denotes the gradient of the sectional curvature as a function of its second argument.
    Now apply the elementary bounds
    \begin{align*}
        \bigg\|w_x\trace(\Sigma_x) - \Sigma_x w_x\bigg\| &\lesssim \|w_x\|\|\Sigma_x\|, \\
        \bigg\|w_x\|\mu_x\|^2 - \mu_x\langle w_x,\mu_x\rangle\bigg\| &\lesssim \|w_x\|\|\mu_x\|\|w_x-\mu_x\|, \\
        \bigg\|\|w_x\|^2\trace(\Sigma_x) - \langle w_x,\Sigma_x w_x\rangle\bigg\| &\lesssim \|w_x\|^2\|\Sigma_x\| \le R\|w_x\|\|\Sigma_x\|, \\
        \bigg\|\|w_x\|^2\|\mu\|^2-\langle w_x,\mu_x\rangle^2\bigg\| &\lesssim \|w_x\|^2\|\mu_x\|\|w_x-\mu_x\| \le R\|w_x\|\|\mu_x\|\|w_x-\mu_x\| 
    \end{align*}
    along with the fact that $\Mcal$ being compact implies that
    \begin{equation*}
        \sup_{x\in\Mcal}\sup_{v,w\in\Tan_x(\Mcal)}\left(|K_x(v,w)| + \|\nabla K_x(v,w)\|_x\right) <\infty
    \end{equation*}
    to see that there exists a constant $C$, depending only on $\Mcal$, such that we have
    \begin{equation}\label{eqn:wB-implicit}
        \|w_{\B x}-\mu_x\| \le C\left(\|w_{\B x}\|\|\Sigma_x\| + \|w_{\B x}\| \|\mu_x\|\|w_{\B x} - \mu_x\| + \|w_{\B x}\|^4\right).
    \end{equation}
    Now plug in the left side of \eqref{eqn:wB-implicit} into the right side of itself, and iterate to get
    \begin{equation}\label{eqn:wB-implicit-iter}
        \begin{split}
            &\|w_{\B x}-\mu_x\| \\
            &\le C\left(\|w_{\B x}\|\|\Sigma_x\|+\|w_{\B x}\|^4\right)\sum_{j=0}^{k}C^j\|w_{\B x}\|^j\|\mu_x\|^j+ C^{k+1}\|w_{\B x}\|^{k+1}\|\mu_x\|^{k+1} \|w_{\B x}-\mu_x\|^{k+1}
        \end{split}
    \end{equation}
    for all $k\in\Nbb$.
    Since Proposition~\ref{prop:mux-sigmax-estimates} shows that the value $\|\mu_x\|$ is strictly less than $(C\sup_{x\in\Mcal}\|w_{\B x}\|)^{-1}$ for sufficiently small $\sigma^2$, we see that \eqref{eqn:wB-implicit-iter} converges as $k\to\infty$, yielding
    \begin{equation*}
        \|w_{\B x}-\mu_x\| \le \frac{C\left(\|w_{\B x}\|\|\Sigma_x\| + \|w_{\B x}\|^4\right)}{1-C\|w_{\B x}\|\|\mu_x\|}.
    \end{equation*}
    Thus, we get
    \begin{equation}\label{eqn:wB-explicit}
        \|w_{\B x}-\mu_x\| \lesssim \|w_{\B x}\|\sigma^2 + \|w_{\B x}\|^4,
    \end{equation}
    by applying the estimates of Proposition~\ref{prop:mux-sigmax-estimates}.

    Next, we use this to control the size of $\|w_{\B x}\|$.
    First,  we show the qualitative statement that $\|w_{\B x}\|\to 0$ as $\sigma^2\to 0$.
    To see this, note that 
    $\delta_{\B}(x)$ is the Fr\'echet mean of the conditional distribution of $\Theta$ given $\{X=x\}$, and this converges weakly to $\delta_x$ as $\sigma^2\to 0$ due to the assumption that $\Theta$ has full support.
    Since $\Mcal$ is compact, weak convergence implies $W_2$ convergence, hence standard continuity results for Fr\'echet means (e.g., \cite{EvansJaffeSLLN, JaffeInfDim}) imply that $\delta(x)$ converges to $x$ as $\sigma^2>0$.
    Thus, the continuity of the logarithm gives the desired result.
    Second, we use \eqref{eqn:wB-explicit} to upgrade this to the quantitative statement that $\|w_{\B x}\| = O(\sigma^2)$ as $\sigma^2\to 0$, as follows.
    By combining \eqref{eqn:wB-explicit}, the triangle inequality, and $\|\mu_x\| = O(\sigma^2)$, we obtain
    $$
    \|w_{\mathcal B x}\|
    \lesssim
    \sigma^2 + \|w_{\mathcal B x}\|\sigma^2 + \|w_{\mathcal B x}\|^4.
    $$
    Now write $C$ for the constant on the right side, which depends only on $\Mcal$, and use $\|w_{\B x}\| = o(1)$ to get $\|w_{\B x}\|^4 \le \frac{1}{2C}\|w_{\B x}\|$ for sufficiently small $\sigma^2$.
    Then the rearranging the previous display shows
    \begin{equation*}
        \left(\frac{1}{2C}-\sigma^2\right)\|w_{\mathcal B x}\|
    \le
    \sigma^2
    \end{equation*}
    for sufficiently small $\sigma^2$, hence $\|w_{\B x}\| = O(\sigma^2)$.
    Finally, plugging the first claim into~\eqref{eqn:wB-explicit} gives the second claim. In fact, all of the estimates above hold uniformly over $x\in\Mcal$ since $\Mcal$ is compact.    
\end{proof}

\begin{proof}[Proof of Theorem~\ref{thm:uncond-main}]
We adopt the notation $V_x,\mu_x,$ and $\Sigma_x$ from the proof of Proposition~\ref{prop:mux-sigmax-estimates}, and we write $\|\cdot\|$ for all norms (i.e., dropping subscripts) for the sake of simplicity.
We recall
 $\mu_x = \Log_x(\delta_{\T}(x))$ from Lemma~\ref{lem:Tweedie}, and we use~\eqref{eqn:sq-dist-Taylor} to see that we have the following Taylor expansion of the conditional risk of $\delta_{\T}$:
\begin{equation*}
    \begin{split}
        &\Ebb[d^2(\Theta,\delta_{\T}(X))\,|\,X=x] \\
        &= \trace(\Sigma_x) - \frac{\Ebb[K_x(V_x,\mu_x)\,|\,X=x]}{3}\left(\| \mu_x\|^2\trace(\Sigma_x)  - \langle \mu_x,\Sigma_x \mu_x\rangle\right) \\
        &\qquad+ \Ebb\left[O\!\left(\lVert V_x\rVert^5+\lVert  \mu_x\rVert^5\right)\,|\,X=x\right].
    \end{split}
    \end{equation*}
    Now recall that $\Mcal$ being compact implies $\sup_{x\in\Mcal}\sup_{v,w\in\Tan_x(\Mcal)}|K_x(v,w)|<\infty$, and also that we have the elementary inequality
    \begin{equation*}
        \left|\|\mu\|^2\trace(\Sigma)-\langle\mu,\Sigma\mu\rangle\right| \lesssim \|\mu\|^2\|\Sigma\|.
    \end{equation*}
    Therefore, combining the previous displays with the estimates of Proposition~\ref{prop:mux-sigmax-estimates} and Proposition~\ref{prop:wB-estimates} shows
    \begin{equation}\label{eqn:conditional-risk-T}
        \Ebb[d^2(\Theta,\delta_{\T}(X))\,|\,X=x] = \trace(\Sigma_x) + O(\sigma^6) + O(\sigma^{6}) + \Ebb\left[O\!\left(\lVert V_x\rVert^5\right)\,|\,X=x\right] + O(\sigma^{10})        
    \end{equation}
    We follow a similar approach for $\delta_{\B}$, where we define $w_{\B x}:=\Log_x(\delta_{\B}(x))$ and use~\eqref{eqn:sq-dist-Taylor} to obtain:
    \begin{equation*}
        \begin{split}
            &\Ebb[d^2(\Theta,\delta_{\B}(X))\,|\,X=x] \\
            &= \trace(\Sigma_x) + \|\mu_x- w_{\B x}\|^2 \\
            &\qquad-\frac{\Ebb[K_x(V_x,w_{\B x})\,|\,X=x]}{3}\left(\| w_{\B x}\|^2\trace(\Sigma_x)  - \langle w_{\B x},\Sigma_x w_{\B x}\rangle + \|w_{\B x}\|^2\|\mu_x\|^2 - \langle w_{\B x},\mu_x\rangle^2\right) \\
            &\qquad+ \Ebb\left[O\!\left(\lVert V_x\rVert^5+\lVert  w_{\B x}\rVert^5\right)\,|\,X=x\right].
        \end{split}
    \end{equation*}
    Then use the estimates of Proposition~\ref{prop:mux-sigmax-estimates} and Proposition~\ref{prop:wB-estimates} along with the elementary inequalities
    \begin{align*}
        \left|\|w\|^2\trace(\Sigma)-\langle w,\Sigma w\rangle\right| &\lesssim \|w\|^2\|\Sigma\| \\
        \left|\|w\|^2\|\mu\|^2-\langle w,\mu\rangle^2\right| &\lesssim \|w\|^2\|\mu\|\,\|\mu-w\|
    \end{align*}
    to obtain
    \begin{equation}\label{eqn:conditional-risk-B}
        \begin{split}
            \Ebb[d^2(\Theta,\delta_{\B}(X))\,|\,X=x] &= \trace(\Sigma_x) + O(\sigma^8) + O(\sigma^6) \\
            &+ O(\sigma^{10}) + O(\sigma^{6}) + \Ebb\left[O\!\left(\lVert V_x\rVert^5\right)\,|\,X=x\right] + O(\sigma^{10})   
        \end{split}
    \end{equation}    
    By subtracting \eqref{eqn:conditional-risk-B} from \eqref{eqn:conditional-risk-T}, we get
    \begin{equation*}
        \Ebb[d^2(\Theta,\delta_{\T}(X))\,|\,X=x]  - \Ebb[d^2(\Theta,\delta_{\B}(X))\,|\,X=x] = O(\sigma^{6}) + \Ebb\left[O\!\left(\lVert V_x\rVert^5\right)\,|\,X=x\right]
    \end{equation*}
    so taking the expectation yields
    \begin{equation*}
        R(\delta_{\T},\sigma^2) - R(\delta_{\B},\sigma^2) = O(\sigma^6) + \Ebb\left[O\!\left(\lVert V_X\rVert^5\right)\right].
    \end{equation*}
    Finally, use Lemma~\ref{lem:integral-limit} with $k=5$ to get
    \begin{equation*}
        \Ebb[\|V_X\|^5] = \Ebb[d^5(X,\Theta)] \sim c_m\sigma^5
    \end{equation*}
    where $c_m>0$ depends only on the dimension $m$ of $\Mcal$.
\end{proof}

\subsection{Proofs from Subsection~\ref{subsec:sobolev}}\label{subsec:proofs-sobolev}

In this subsection, we prove some results establishing Sobolev bounds on the density of Riemannian Gaussian mixtures.

In some cases we bound the Sobolev norm of a function by first uniformly bounding its derivatives, and then using the following result to upgrade this to a bound on the Sobolev norm. 

\begin{lemma}\label{lem:equiv-Sobolev-defs}
    For integer $r\ge 1$ and sufficiently smooth $h:\Mcal\to\Rbb$, we have
    \begin{equation*}
        \|h\|^2_{\Hcal^r(\Mcal)} \asymp_{\Mcal,r}\|h\|^2_{L^2(\Mcal)}+\int_{\Mcal}\|\nabla^rh(x)\|_{x,r}^2\diff x
    \end{equation*}
\end{lemma}

\begin{proof}
    It is standard (see \cite[Chapter~2]{Aubin}) that we have 
    \begin{equation*}
        \|h\|_{\Hcal^r(\Mcal)} \asymp_{\Mcal,r}\|h\|_{L^2(\Mcal)}+\|\Delta^{r/2}h\|_{L^2(\Mcal)}
    \end{equation*}
    so it suffices to show
    \begin{equation*}
        \int_{\Mcal}\|\nabla^rh(x)\|_{x,r}^2\diff x\asymp_{\Mcal,r}\|\Delta^{r/2}h\|_{L^2(\Mcal)}^2.
    \end{equation*}
    To do this, note that if $r$ is even then we can iteratively apply the following elliptic regularity bound (which follows, e.g., from G{\aa}rding's inequality \cite[Chapter 7.6, Vol II]{TaylorPDE}):
    \begin{equation*}
        \|\nabla^2h\|_{L^2(\Mcal)} \lesssim_{\Mcal}\|h\|_{L^2(\Mcal)}+\|\Delta h\|_{L^2(\Mcal)}.
    \end{equation*}
    If $r$ is odd, then we combine this with the identity in \eqref{eqn:Sobolev-1-equiv}.
\end{proof}

Now we turn to the proofs of the main results from the body.


\begin{proof}[Proof of Lemma~\ref{lem:general-Sobolev}]
By~\eqref{eqn:Sobolev-1-equiv},
\[
    \|f\|_{\Hcal^1(\Mcal)}^2
    =
    \|f\|_{L^2(\Mcal)}^2
    +
    \int_{\Mcal}\|\nabla f(x)\|_x^2\diff x.
\]
By Lemma~\ref{lem:f-lower-bd}, we have
\[
    \|f\|_{L^2(\Mcal)}^2
    \le
    \vol(\Mcal)\|f\|_{L^\infty(\Mcal)}^2
    \lesssim_\Mcal
    \|g\|_{L^\infty(\Mcal)}^2.
\]
By Lemma~\ref{lem:f-diff-1},
\[
    \int_{\Mcal}\|\nabla f(x)\|_x^2\diff x
    \le
    \vol(\Mcal)\sup_{x\in\Mcal}\|\nabla f(x)\|_x^2
    \lesssim_\Mcal
    \sigma^{-2}\|g\|_{L^\infty(\Mcal)}^2.
\]
Since $\sigma^2\le(\diam(\Mcal))^2$, the first term is absorbed into the second term, up to a constant depending only on $\Mcal$.
Taking square roots gives the result.
\end{proof}

\begin{proof}[Proof of Lemma~\ref{lem:Sobolev-est-Euclidean}]
    Write $H_r$ for the $r$th multivariate Hermite polynomial in $\Rbb^m$, i.e., the $r$-linear form defined by the relation
    \begin{equation*}
        \nabla^r\exp\left(-\frac{\|x\|^2}{2}\right) = (-1)^r H_r(x)\exp\left(-\frac{\|x\|^2}{2}\right).
    \end{equation*}
    Then we can compute $\nabla^r f(x)$ by differentiating under the integral $r$ times to obtain the following, where $\|\cdot\|_r$ denotes any norm on the finite-dimensional vector space of $r$-linear forms:
    \begin{align*}
        \|\nabla^rf(x)\|_{r} &= \left\|\frac{1}{(\sqrt{2\pi}\sigma)^m}\int_{\Rbb^m}\nabla^r\exp\left(-\frac{\|x-\theta\|^2}{2\sigma^2}\right)g(\theta)\diff \theta\right\|_{r} \\
        &= \left\|\frac{\sigma^{-r}}{(\sqrt{2\pi}\sigma)^m}\int_{\Rbb^m}H_r\left(\frac{x-\theta}{\sigma}\right)\exp\left(-\frac{\|x-\theta\|^2}{2\sigma^2}\right)g(\theta)\diff \theta\right\|_{r} \\
        &= \left\|\frac{\sigma^{-r}}{(\sqrt{2\pi}\sigma)^m}\int_{\Rbb^m}H_r\left(-\frac{v}{\sigma}\right)\exp\left(-\frac{\|v\|^2}{2\sigma^2}\right)g(x+v)\diff v\right\|_{r} \\
        &\le \frac{\sigma^{-r}}{(\sqrt{2\pi}\sigma)^m}\int_{\Rbb^m}\left\|H_r\left(-\frac{v}{\sigma}\right)\right\|_{r}\exp\left(-\frac{\|v\|^2}{2\sigma^2}\right)g(x+v)\diff v \\
        &\le \frac{\sigma^{-r}\|g\|_{L^{\infty}(\Rbb^m)}}{(\sqrt{2\pi}\sigma)^m}\int_{\Rbb^m}\left\|H_r\left(\frac{v}{\sigma}\right)\right\|_{r}\exp\left(-\frac{\|v\|^2}{2\sigma^2}\right)\diff v.
    \end{align*}
    Now write $Z$ for a standard Gaussian random variable in $\Rbb^m$ (i.e., with mean 0 and covariance matrix $I_m$), so that the expression above is just
    \begin{equation*}
        \|\nabla^rf(x)\|_{r} \le \sigma^{-r}\|g\|_{L^{\infty}(\Rbb^m)}\Ebb\left[\left\|H_r\left(Z\right)\right\|_{r}\right].
    \end{equation*}
    Since the expectation depends only on $m$ and $r$, this finishes the proof.
\end{proof}

\begin{proof}[Proof of Lemma~\ref{lem:Sobolev-Lie}]
    Define $d_0:\Mcal\to[0,\infty)$ so that $d_0(y)\ge 0$ equals the distance from $y\in\Mcal$ to the identity element.
    Then we have
    \begin{equation*}
        p_{\theta,\sigma^2}(x)=\exp\left(-\frac{d_0^2(x^{-1}\theta)}{2\sigma^2}-A(\sigma^2)\right)
    \end{equation*}
    hence, by changing variables and using the invariance of the volume form:
    \begin{align*}
        f(x) &= \int_{\Mcal}\exp\left(-\frac{d_0^2(x^{-1}\theta)}{2\sigma^2}-A(\sigma^2)\right)g(\theta)\diff\theta \\
        &= \int_{\Mcal}\exp\left(-\frac{d_0^2(\theta)}{2\sigma^2}-A(\sigma^2)\right)g(x\theta)\diff\theta.
    \end{align*}
    Now write $R_{\theta}:\Mcal\to\Mcal$ for the operator of right multiplication by $\theta$, which is an isometry and is infinitely-differentiable since $\Mcal$ is a Lie group endowed with a bi-invariant metric tensor.
    Consequently, $x\mapsto g(x\theta) = (g\circ R_{\theta})(x)$ is infinitely differentiable, and we can differentiate it by the chain rule, yielding $\nabla^r(g\circ R_{\theta})(x) = R_{\theta}^{\ast}(\nabla^rg(x\theta))$ where $R_{\theta}^{\ast}$ is the isometry of $r$-linear forms on $\Tan_x(\Mcal)$ which is induced by $R_{\theta}$.
    Thus, differentiating $r$ times under the integral (which is valid because $\Mcal$ is compact and all derivatives $\nabla^kg$ for $k\le r$ are uniformly bounded) yields   
    \begin{equation*}
        \nabla^rf(x) =\int_{\Mcal}\exp\left(-\frac{d_0^2(\theta)}{2\sigma^2}-A(\sigma^2)\right)R_{\theta}^{\ast}(\nabla^rg(x))\diff\theta
    \end{equation*}
    hence
    \begin{align*}
        \|\nabla^rf(x)\|_{x,r} &=\left\|\int_{\Mcal}\exp\left(-\frac{d_0^2(\theta)}{2\sigma^2}-A(\sigma^2)\right)R_{\theta}^{\ast}(\nabla^r g(x))\diff\theta\right\|_{x,r} \\
        &\le\int_{\Mcal}\exp\left(-\frac{d_0^2(\theta)}{2\sigma^2}-A(\sigma^2)\right)\|R_{\theta}^{\ast}(\nabla^rg(x))\|_{x,r}\diff\theta \\
        &\le\int_{\Mcal}\exp\left(-\frac{d_0^2(\theta)}{2\sigma^2}-A(\sigma^2)\right)\diff\theta\cdot \sup_{\theta\in\Mcal}\|\nabla^rg(\theta)\|_{\theta,r} \\
        &=\sup_{\theta\in\Mcal}\|\nabla^rg(\theta)\|_{\theta,r}.
    \end{align*}
    Now, use Lemma~\ref{lem:equiv-Sobolev-defs} and Lemma~\ref{lem:f-lower-bd} to get
    \begin{align*}
        \|f\|_{\Hcal^r(\Mcal)}^2 &\lesssim_{\Mcal,r}\|f\|^2_{L^2(\Mcal)}+\int_{\Mcal}\|\nabla^rf(x)\|^2_{x,r}\diff x \\
        &\le\|f\|_{L^2(\Mcal)}^2 + \vol(\Mcal)\sup_{x\in\Mcal}\|\nabla^rf(x)\|_{x,r}^2 \\
        &\lesssim_{\Mcal}\vol(\Mcal)\|g\|_{L^{\infty}(\Mcal)}^2 + \vol(\Mcal)\sup_{\theta\in\Mcal}\|\nabla^rg(\theta)\|_{\theta,r}^2
    \end{align*}
    and conclude.
\end{proof}

\begin{proof}[Proof of Lemma~\ref{lem:sphere-estimate}]
    Recall from Example~\ref{ex:sphere} that \cite{PennecHessian} allows us to compute the Hessian of the squared distance functional, by representing $\mathbb{S}^2 = \{x\in\Rbb^3:\|x\|=1\}$; combining this with the chain rule, we can directly calculate 
    \begin{align*}
        \nabla_x^2p_{\theta,\sigma^2}(x) &=-\frac{1}{\sigma^2}\left(\frac{\Log_x\theta (\Log_x\theta)^{\top}}{d^2(x,\theta)} + h(d(x,\theta))\left(I-xx^{\top}-\frac{\Log_x\theta (\Log_x\theta)^{\top}}{d^2(x,\theta)}\right)\right)p_{\theta,\sigma^2}(x) \\
        &\qquad+ \frac{\Log_x\theta(\Log_x\theta)^{\top}}{\sigma^4}p_{\theta,\sigma^2}(x)
    \end{align*}
    hence 
    \begin{align*}
        \|\nabla_x^2p_{\theta,\sigma^2}(x)\|_{x,2} &\le \left(\frac{1}{\sigma^2}\left(1+ |h(d(x,\theta))|\left(\|I-xx^{\top}\|_{x,2} + 1\right)\right)+\frac{d^2(x,\theta)}{\sigma^4}\right)p_{\theta,\sigma^2}(x),
    \end{align*}
    where $h(t) = t\cos(t)/\sin(t)$ for $0\le t\le \pi$.
    In particular, we have
    \begin{align*}
        &\|\nabla^2f(x)\|_{x,2} \\
        &\le\int_{\Mcal}\|\nabla_x^2p_{\theta,\sigma^2}(x)\|_{x,2} \,\diff G(\theta) \\
        &\le \int_{\Mcal}\left(\frac{1}{\sigma^2} +\frac{d^2(x,\theta)}{\sigma^4}\right)p_{\theta,\sigma^2}(x)\diff G(\theta) + \left(\|I-xx^{\top}\|_{x,2} + 1\right)\frac{1}{\sigma^2}\int_{\Mcal}|h(d(x,\theta))|p_{\theta,\sigma^2}(x)\diff G(\theta),
    \end{align*}
    and we can control these two terms separately.
    
    For the first term, use Lemma~\ref{lem:normalizing-const}, the normal coordinates~\eqref{eq:normal-coor}, the identity $d(x,y) = \|\Log_x\theta\|_x$, and Lemma~\ref{lem:jac-det-bdd} to get:
    \begin{align*}
        \int_{\Mcal}&\left(\frac{1}{\sigma^2} +\frac{d^2(x,\theta)}{\sigma^4}\right)p_{\theta,\sigma^2}(x)\diff G(\theta) \\
        &\lesssim_{\Mcal} \frac{1}{(\sqrt{2\pi}\sigma)^m}\int_{\Mcal}\left(\frac{1}{\sigma^2} +\frac{d^2(x,\theta)}{\sigma^4}\right)\exp\left(-\frac{d^2(x,\theta)}{2\sigma^2}\right)\diff G(\theta) \\
        &= \frac{1}{(\sqrt{2\pi}\sigma)^m}\int_{D_x}\left(\frac{1}{\sigma^2} +\frac{\|v\|_x^2}{\sigma^4}\right)\exp\left(-\frac{\|v\|_x^2}{2\sigma^2}\right)g(\Exp_x(v))j_x(v)\diff v \\
        &\lesssim_{\Mcal} \frac{\|g\|_{L^{\infty}(\Mcal)}}{(\sqrt{2\pi}\sigma)^m}\int_{D_x}\left(\frac{1}{\sigma^2} +\frac{\|v\|_x^2}{\sigma^4}\right)\exp\left(-\frac{\|v\|_x^2}{2\sigma^2}\right)\diff v \\
        &\le \frac{\|g\|_{L^{\infty}(\Mcal)}}{(\sqrt{2\pi}\sigma)^m}\int_{\Tan_x(\Mcal)}\left(\frac{1}{\sigma^2} +\frac{\|v\|_x^2}{\sigma^4}\right)\exp\left(-\frac{\|v\|_x^2}{2\sigma^2}\right)\diff v \\
        &= \|g\|_{L^{\infty}(\Mcal)}\left(\frac{1}{\sigma^2} +\frac{m\sigma^2}{\sigma^4}\right) \\
        &\lesssim_{\Mcal} \sigma^{-2}\|g\|_{L^{\infty}(\Mcal)}.
    \end{align*}
    For the second term, use Lemma~\ref{lem:normalizing-const}, write the integral in spherical polar coordinates centered at $x$, and define $t=d(x,\theta)$, to compute
    \begin{equation*}
        \begin{split}
        &\frac{1}{\sigma^2}
        \int_{\mathbb S^2}
        |h(d(x,\theta))|p_{\theta,\sigma^2}(x)g(\theta)\diff\theta \\
        &\qquad\lesssim_\Mcal
        \frac{\|g\|_{L^\infty}}{\sigma^2(\sqrt{2\pi}\sigma)^2}
        \int_0^\pi
        \left|\frac{t\cos t}{\sin t}\right|
        \exp\left(-\frac{t^2}{2\sigma^2}\right)
        \sin t\,\diff t \\
        &\qquad=
        \frac{\|g\|_{L^\infty}}{\sigma^2(\sqrt{2\pi}\sigma)^2}
        \int_0^\pi
        t|\cos t|
        \exp\left(-\frac{t^2}{2\sigma^2}\right)\diff t \\
        &\qquad\lesssim
        \frac{\|g\|_{L^\infty}}{\sigma^2\sigma^2}
        \int_0^\infty
        t\exp\left(-\frac{t^2}{2\sigma^2}\right)\diff t \\
        &\qquad\lesssim
        \sigma^{-2}\|g\|_{L^\infty}
        \end{split}
    \end{equation*}
    as needed.

    Finally, combine these two bounds with Lemma~\ref{lem:f-lower-bd} and Lemma~\ref{lem:equiv-Sobolev-defs} to get
    \begin{equation*}
        \|f\|_{\Hcal^2(\Mcal)}^2 \lesssim \|f\|^2_{L^{\infty}(\Mcal)} + \vol(\Mcal)\sup_{x\in\Mcal}\|\nabla^2f(x)\|^2_{x,2} \lesssim_{\Mcal} \|g\|^2_{L^{\infty}(\Mcal)} +\sigma^{-4}\|g\|^2_{L^{\infty}(\Mcal)},
    \end{equation*}
    and use $\sigma^2\le (\diam(\Mcal))^2$ to conclude.
\end{proof}

\begin{proof}[Proof of Lemma~\ref{lem:ex-uniform}]
    Directly compute $f^{\ast}(0) = (\vol(\Mcal))^{-1/2}$ and $f^{\ast}(j) = 0$ for all $j\ge 1$.
\end{proof}

\begin{proof}[Proof of Lemma~\ref{lem:Sobolev-bd-S1}]
First, we define the kernel $K_\sigma:\mathbb{S}^1\to\Rbb$ by
\begin{equation*}
    K_\sigma(x):=
\exp\left(
    -\frac{d^2(x,0)}{2\sigma^2}
    -A(\sigma^2)
\right),
\end{equation*}
which, by representing $\mathbb{S}^1=\Rbb/2\pi\Zbb$ can be written as
\[
K_\sigma(t):=\exp\left(-\frac{t^2}{2\sigma^2}-A(\sigma^2)\right)
\]
which extending periodically from $t\in[-\pi,\pi]$ to $t\in\Rbb$.
Note that $A(\sigma^2)$ does not depend on $\theta$ since $\Mcal$ is homogeneous.
Then note that
\[
f(x)=\int_{\mathbb{S}^1}K_\sigma(x-\theta)\,g(\theta)\,\diff\theta,
\]
which is exactly the convolution of $K_\sigma$ and $g$ on $\mathbb{S}^1$.

Next we bound the Fourier series coefficients of $K_{\sigma}$.
For convenience, we index the eigenvalues as $\lambda_k=k^2$ for $k\in\Zbb$, and we claim that
\begin{equation}\label{eqn:kernel-Fourier-bd}
    |K^{\ast}_\sigma(k)|\asymp \sigma^{-2}(1+k^2)^{-1}
\end{equation}
holds uniformly over $k\in\Zbb$ and $0<\sigma^2\le \pi^2$.
The claim is obvious for $k=0$, since $K_{\sigma}$ being a probability density implies $K_{\sigma}^{\ast}(0) \lesssim 1 \lesssim \sigma^{-2}$ for $\sigma^2\le \pi^2$.
So, fix $k\neq 0$, and use integration by parts to get
\[
K_\sigma^{\ast}(k)
=
\frac{1}{2\pi ik}\int_{-\pi}^{\pi}K_\sigma'(t)e^{-ikt}\,\diff t.
\]
For $-\pi<t<\pi$ we have
\[
K_\sigma'(t)=-\frac{t}{\sigma^2}K_\sigma(t),
\]
so $K_\sigma'$ is piecewise continuous and has a jump at $\pm\pi$.
Integrating by parts again gives
\begin{equation}\label{eqn:F-Fourier-series}
K_\sigma^{\ast}(k)
\le
\frac{1}{2\pi k^2}
\left(
\bigl|K_\sigma'(\pi^-)-K_\sigma'(-\pi^+)\bigr|
+
\int_{-\pi}^{\pi}|K_\sigma''(t)|\,\diff t
\right),
\end{equation}
and we can analyze the two terms in parentheses separately.
The first term of~\eqref{eqn:F-Fourier-series} is exactly
\[
|K_\sigma'(\pi^-)-K_\sigma'(-\pi^+)|
=
2\pi\sigma^{-2}K_\sigma(\pi)
=
2\pi\sigma^{-2}\exp\left(-\frac{\pi^2}{2\sigma^2}-A(\sigma^2)\right),
\]
and Lemma~\ref{lem:normalizing-const} implies
$\exp(-A(\sigma^2))\lesssim \sigma^{-1}$ for $0<\sigma^2\le \pi^2$, hence $|K_\sigma'(\pi^-)-K_\sigma'(-\pi^+)|\lesssim
\sigma^{-3}\exp(-\frac{\pi^2}{2\sigma^2})
\lesssim
\sigma^{-2}$.
For the second term of~\eqref{eqn:F-Fourier-series}, we compute $K_\sigma''(t)=(\sigma^{-4}t^2-\sigma^{-2})K_\sigma(t)$ for $-\pi<t<\pi$,
hence
\[
\int_{-\pi}^{\pi}|K_\sigma''(t)|\,\diff t
\le
\sigma^{-4}\int_{-\pi}^{\pi}t^2K_\sigma(t)\,\diff t
+
\sigma^{-2}\int_{-\pi}^{\pi}K_\sigma(t)\,\diff t.
\]
The second integral is exactly $1$, and the first integral is bounded by $\sigma^{-2}$, which follows from Lemma~\ref{lem:normalizing-const} and the
Gaussian moment bound $\int_{-\pi}^{\pi} t^2 e^{-t^2/(2\sigma^2)}\,\diff t \lesssim \sigma^3$, hence we have shown $\int_{-\pi}^{\pi}|K_\sigma''(t)|\,\diff t \asymp \sigma^{-2}.$
Substituting these bounds into~\eqref{eqn:F-Fourier-series} establishes
\[
|K_\sigma^{\ast}(k)|\asymp \sigma^{-2}k^{-2}
\]
for all $k\neq 0$, whence~\eqref{eqn:kernel-Fourier-bd}.

Lastly, we use the convolution identity for Fourier coefficients, the bound
\eqref{eqn:kernel-Fourier-bd}, and the definition of Sobolev norms to obtain
\begin{align*}
\|f\|_{\Hcal^{s+2}(\mathbb{S}^1)}^2
&=
\sum_{k\in\Zbb}(1+k^2)^{s+2}|f^{\ast}(k)|^2 \\
&=
\sum_{k\in\Zbb}(1+k^2)^{s+2}|K_{\sigma}^{\ast}(k)|^2|g^{\ast}(k)|^2 \\
&\asymp
\sigma^{-4}
\sum_{k\in\Zbb}(1+k^2)^{s}|g^{\ast}(k)|^2 \\
&=
\sigma^{-4}\|g\|_{\Hcal^s(\mathbb{S}^1)}^2.
\end{align*}
Taking square roots completes the proof.
\end{proof}

\subsection{Proofs from Subsection~\ref{subsec:empirical}}\label{subsec:proofs-empirical}

Lastly, we give the proofs for our results concerning our empirical Bayes approximation $\hat{\delta}_{\T}$ of $\delta_{\T}$.

\begin{proof}[Proof of Proposition~\ref{prop:kde-error}]    First, note that we have $f\in L^2(\Mcal)$ because of $\|f\|_{\Hcal^r(\Mcal)}<\infty$ for some $r>0$, and also that  $\hat{f}_n\in L^2(\Mcal)$ because it is written as a finite linear combination of $\{\phi_j\}_{j=0}^{\infty}\subseteq L^2(\Mcal)$.
Now observe that $f^{\ast},\hat{f}_n^{\ast}:\{0,1,\ldots,\}\to\Cbb$ are given by 
    \begin{equation*}
        f^{\ast}(j):= \int_{\Mcal}f(x)\overline{\phi_j(x)}\diff x
    \end{equation*}
    and
    \begin{equation*}
        \hat{f}_n^{\ast}(j):=\begin{cases}
            \frac{1}{n}\sum_{i=1}^{n}\overline{\phi_j(X_i)} &\textnormal{ if } j\le \ell \\
            0 &\textnormal{ if } j> \ell,
        \end{cases}
    \end{equation*}
    so that we have
    \begin{equation*}
        \hat f_n(x)=\sum_{j=0}^{\ell}\hat f_n^{\ast}(j)\phi_j(x)
    \end{equation*}
    and also that $\hat{f}_n^{\ast}(j)$ is an unbiased estimator of $f^{\ast}(j)$ for all $0\le j\le \ell$.

    The first claim is based on the following calculations.
    First, expand
    \begin{equation}\label{eqn:kde-1-k1}
        \begin{split}
            \int_{\Mcal}\left|\hat{f}_n(x)-f(x)\right|^2\diff x &= \sum_{j=0}^{\ell}\left|\hat{f}_n^{\ast}(j)-f^{\ast}(j)\right|^2+\sum_{j=\ell+1}^{\infty}|f^{\ast}(j)|^2 \\
            &= \sum_{j=0}^{\ell}\left|\frac{1}{n}\sum_{i=1}^{n}\phi_j(X_i)-f^{\ast}(j)\right|^2+\sum_{j= \ell+1}^{\infty}|f^{\ast}(j)|^2.
        \end{split}
    \end{equation}
    and let us now handle the two terms on the right side separately.
    For the first term (which represents the variance), we take the expectation over $X_1,\ldots, X_n$ then use Lemma~\ref{lem:f-lower-bd} to get:
    \begin{equation}\label{eqn:kde-4-k1}
        \begin{split}
        \Ebb\left[\sum_{j=0}^{ \ell}\left|\frac{1}{n}\sum_{i=1}^{n}\phi_j(X_i)-f^{\ast}(j)\right|^2\right] &= \frac{1}{n}\sum_{j=0}^{\ell}\,\Var(\phi_j(X_1))\\
        &\le \frac{1}{n}\sum_{j=0}^{\ell}\,\int_{\Mcal}|\phi_j(x)|^2f(x)\diff x \\
        &\le \|f\|_{L^\infty(\Mcal)}\frac{\ell}{n} \\
        &\lesssim_{\Mcal}\|g\|_{L^\infty(\Mcal)}\frac{\ell}{n} \\
        \end{split}
    \end{equation}
    for all $\ell\in\Nbb$.
    For the second term (which is non-random, and can be interpreted as the squared bias), use the assumption of~$\|f\|_{\Hcal^r(\Mcal)}<\infty$, and Weyl's law \cite[Section~VI.4]{Chavel}, to get
    \begin{equation}\label{eqn:kde-2-k1}
        \begin{split}
            \sum_{j= \ell+1}^{\infty}|f^{\ast}(j)|^2 &\le \lambda_{\ell}^{-r}\sum_{j= \ell+1}^{\infty}\lambda_{j}^{r}|f^{\ast}(j)|^2 \\
            &\le \lambda_{\ell}^{-r}\sum_{j=0}^{\infty}\lambda_{j}^{r}|f^{\ast}(j)|^2 \\
            &\le \lambda_{\ell}^{-r}\sum_{j=0}^{\infty}(1+\lambda_{j})^{r}|f^{\ast}(j)|^2 \\
            &= \lambda_{\ell}^{-r}\|f\|_{\Hcal^r(\Mcal)}^2 \\     &\lesssim_{\Mcal,r} \ell^{-\frac{2r}{m}}\|f\|_{\Hcal^r(\Mcal)}^2 \\
        \end{split}
    \end{equation}
    for all $\ell\in\Nbb$.
    Consequently, plugging in \eqref{eqn:kde-2-k1} and \eqref{eqn:kde-4-k1} into \eqref{eqn:kde-1-k1} yields
    \begin{equation}\label{eqn:kde-k1}
        \Ebb\left[\int_{\Mcal}\left|\hat{f}_n(x)-f(x)\right|^2\diff x\right] \lesssim_{\Mcal} \|g\|_{L^\infty(\Mcal)}\frac{\ell}{n} + \ell^{-\frac{2r}{m}}\|f\|_{\Hcal^r(\Mcal)}^2 
    \end{equation}
    for all $\ell\in\Nbb$.

    The second claim is based on similar calculations, as follows.
    First, use~\eqref{eqn:Sobolev-1-equiv} to expand
    \begin{equation}\label{eqn:kde-1-k2}
        \begin{split}
            \int_{\Mcal}\left\|\nabla\hat{f}_n(x)-\nabla f(x)\right\|_x^2\diff x &= \sum_{j=0}^{\ell}\lambda_j\left|\hat{f}_n^{\ast}(j)-f^{\ast}(j)\right|^2+\sum_{j=\ell+1}^{\infty}\lambda_j|f^{\ast}(j)|^2 \\
            &= \sum_{j=0}^{\ell}\lambda_j\left|\frac{1}{n}\sum_{i=1}^{n}\phi_j(X_i)-f^{\ast}(j)\right|^2+\sum_{j= \ell+1}^{\infty}\lambda_j|f^{\ast}(j)|^2,
        \end{split}
    \end{equation}
    and we control these terms as before.
    For the first term, take expectation and use Lemma~\ref{lem:f-lower-bd} and Weyl's law:
    \begin{equation}\label{eqn:kde-4-k2}
        \begin{split}
        \Ebb\left[\sum_{j=0}^{ \ell}\lambda_j\left|\frac{1}{n}\sum_{i=1}^{n}\phi_j(X_i)-f^{\ast}(j)\right|^2\right] &= \sum_{j=0}^{\ell}\lambda_j\Ebb\left[\left|\frac{1}{n}\sum_{i=1}^{n}\phi_j(X_i)-f^{\ast}(j)\right|^2\right] \\
        &= \frac{1}{n}\sum_{j=0}^{\ell}\,\lambda_j\Var(\phi_j(X_1))\\
        &\le \frac{1}{n}\sum_{j=0}^{\ell}\,\lambda_j\int_{\Mcal}|\phi_j(x)|^2f(x)\diff x \\
        &\le \frac{\|f\|_{L^\infty(\Mcal)}}{n}\sum_{j=0}^{\ell}\,\lambda_j \\
        &\lesssim_{\Mcal} \frac{\|g\|_{L^\infty(\Mcal)}}{n}\sum_{j=0}^{\ell}\,\lambda_j \\
        &\le\frac{\|g\|_{L^\infty(\Mcal)}}{n}\ell\lambda_{\ell} \\
        &\lesssim_{\Mcal}\frac{\|g\|_{L^\infty(\Mcal)}}{n}\ell^{1+\frac{2}{m}} \\
        \end{split}
    \end{equation}
    for all $\ell\in\Nbb$.
    For the second term, use~$\|f\|_{\Hcal^r(\Mcal)}<\infty$, and Weyl's law:   \begin{equation}\label{eqn:kde-2-k2}
        \begin{split}
            \sum_{j= \ell+1}^{\infty}\lambda_j|f^{\ast}(j)|^2 &= \sum_{j= \ell+1}^{\infty}\lambda_j^{1-r}\lambda_{j}^{r}|f^{\ast}(j)|^2 \\
            &\le \lambda_{\ell}^{1-r}\sum_{j= \ell+1}^{\infty}\lambda_{j}^{r}|f^{\ast}(j)|^2 \\
            &\le \lambda_{\ell}^{1-r}\sum_{j=0}^{\infty}\lambda_{j}^{r}|f^{\ast}(j)|^2 \\
            &\le \lambda_{\ell}^{1-r}\sum_{j=0}^{\infty}(1+\lambda_{j})^{r}|f^{\ast}(j)|^2 \\
            &= \lambda_{\ell}^{1-r}\|f\|_{\Hcal^r(\Mcal)}^2 \\            &\lesssim_{\Mcal,r} \ell^{\frac{2(1-r)}{m}}\|f\|_{\Hcal^r(\Mcal)}^2 \\
        \end{split}
    \end{equation}
    for all $\ell\in\Nbb$.
    Now plug \eqref{eqn:kde-2-k2} and \eqref{eqn:kde-4-k2} into \eqref{eqn:kde-1-k2} to obtain
    \begin{equation}\label{eqn:kde-k2}
        \Ebb\left[\int_{\Mcal}\left\|\nabla\hat{f}_n(x)-\nabla f(x)\right\|_x^2\diff x\right] \lesssim_{\Mcal} \|g\|_{L^\infty(\Mcal)}\frac{\ell^{1+\frac{2}{m}}}{n} + \ell^{\frac{2(1-r)}{m}}\|f\|_{\Hcal^r(\Mcal)}^2 
    \end{equation}
    for all $\ell\in\Nbb$.

    Note that we can capture both~\eqref{eqn:kde-k1} and~\eqref{eqn:kde-k2} in the following expression for $k\in\{0,1\}$, where we use the norm $\|\cdot\|$ to denote $|\cdot|$ in the $k=0$ case and to denote $\|\cdot\|_x$ in the $k=1$ case:
    \begin{equation*}
        \Ebb\left[\int_{\Mcal}\left\|\nabla^k\hat{f}_n(x)-\nabla^k f(x)\right\|^2\diff x\right] \lesssim_{\Mcal} \|g\|_{L^\infty(\Mcal)}\frac{\ell^{1+\frac{2k}{m}}}{n} + \ell^{\frac{2(k-r)}{m}}\|f\|_{\Hcal^r(\Mcal)}^2. 
    \end{equation*}
    Thus, we complete the proof by plugging in the given form of $\ell_n$ to obtain
    \begin{equation}\label{eqn:density-derivs-est}
        \begin{split}
        \int_{\Mcal}\Ebb\left[\left\|\nabla^k\hat{f}_n(x)-\nabla^k f(x)\right\|^2\right]\diff x &\lesssim_{\Mcal,r} \frac{\|g\|_{L^{\infty}(\Mcal)}\ell^{1+\frac{2k}{m}}}{n} +\ell^{\frac{2(k-r)}{m}}\|f\|_{\Hcal^{r}(\Mcal)}^2 \\
        &=\left(\frac{\|g\|_{L^{\infty}(\Mcal)}\ell}{n} +\frac{\|f\|_{\Hcal^{r}(\Mcal)}^2}{\ell^{\frac{2r}{m}}}\right)\ell^{\frac{2k}{m}} \\
        &\asymp \frac{\|f\|_{\Hcal^{r}(\Mcal)}^{\frac{2m}{2r+m}}\|g\|_{L^{\infty}(\Mcal)}^{\frac{2r}{2r+m}}}{n^{\frac{2r}{2r+m}}} \ell^{\frac{2k}{m}} \\
        &= \frac{\|f\|_{\Hcal^{r}(\Mcal)}^{\frac{2m+4k}{2r+m}}\|g\|_{L^{\infty}(\Mcal)}^{\frac{2r-2k}{2r+m}}}{n^{\frac{2r-2k}{2r+m}}}.
        \end{split}
    \end{equation}
    This implies both claims by taking $k=0$ and $k=1$, respectively.
\end{proof}

\begin{proof}[Proof of Theorem~\ref{thm:EB-risk}]

    For convenience, write $w_{\T x}:=\sigma^2\nabla \log f(x)$ and $\hat{w}_{\T x}:= \sigma^2\nabla_x\hat{f}_n(x)/\max\{\hat{f}_n(x),\rho\}$, and note that Lemma~\ref{lem:Tweedie} implies
    \begin{equation*}
        \|w_{\T x}\|_x = \|\Ebb[\Log_x\Theta\,|\,X=x]\|_x \le \Ebb[\|\Log_x\Theta\|_x\,|\,X=x]  \le \diam(\Mcal).
    \end{equation*}
    We claim that there exists a constant $C_\Mcal>0$ such that, for every
    $x\in\Mcal$ and every $v\in\Tan_x(\Mcal)$,
    \[
        d^2\left(\Exp_x(w_{\T x}),\Exp_x(v)\right)
        \le
        C_\Mcal \|w_{\T x}-v\|_x^2.
    \]
    Indeed, if $\|v\|_x\le 2\diam(\Mcal)$, this follows from smoothness of
    $(x,v)\mapsto \Exp_x(v)$ on the compact set $\{(x,v):x\in\Mcal,\ \|v\|_x\le 2\diam(\Mcal)\}$; otherwise $\|v\|_x>2\diam(\Mcal)$ hence $\|w_{\T x}-v\|_x\ge \diam(\Mcal)$ and $d(\Exp_x(w_{\T x}),\Exp_x(v))\le \diam(\Mcal)$, so the bound holds after increasing $C_\Mcal$ if necessary.

    Applying this with $v:=\hat{w}_{\T x}$ gives
    \begin{equation}\label{eqn:EB-1}
        \Ebb\left[\int_{\Mcal}d^2\left(\hat{\delta}_{\T}(x),\delta_{\T}(x)\right)f(x)\diff x\right]
        \le
        C_\Mcal \sigma^4
        \Ebb\left[\int_{\Mcal}
        \left\|
        \frac{\nabla_x\hat f_n(x)}{\max\{\hat f_n(x),\rho\}}
        -
        \frac{\nabla_x f(x)}{f(x)}
        \right\|_x^2
        f(x)\diff x\right]
    \end{equation}
    for all $n\in\Nbb$.
    Now we write
    \begin{equation*}
        \frac{\nabla_x\hat{f}_n(x)}{\max\{\hat{f}_n(x),\rho\}}-\frac{\nabla_x f(x)}{f(x)} = \frac{\nabla_x\hat{f}_n(x)-\nabla_xf(x)}{\max\{\hat{f}_n(x),\rho\}} + \frac{\nabla_x f(x)}{f(x)}\cdot\frac{f(x) - \max\{\hat{f}_n(x),\rho\}}{\max\{\hat{f}_n(x),\rho\}},
    \end{equation*}
    so that we can expand the norm appearing on the right side above as
    \begin{align*}
        &\left\|\frac{\nabla_x\hat{f}_n(x)}{\max\{\hat{f}_n(x),\rho\}}-\frac{\nabla_x f(x)}{f(x)}\right\|^2_{x} \\
        &\lesssim \left\|\frac{\nabla_x\hat{f}_n(x)-\nabla_xf(x)}{\max\{\hat{f}_n(x),\rho\}}\right\|_{x}^2 + \left\|\frac{\nabla_x f(x)}{f(x)}\right\|^2\left|\frac{f(x) - \max\{\hat{f}_n(x),\rho\}}{\max\{\hat{f}_n(x),\rho\}}\right|^2.
    \end{align*}
    Next we will take the expectation and integral and bound these terms separately.

    The first term can be bounded directly by Lemma~\ref{lem:f-lower-bd} and Proposition~\ref{prop:kde-error}, as
    \begin{equation}\label{eqn:EB-2-new}
    \begin{split}
        \Ebb\left[\int_{\Mcal}\left\|\frac{\nabla_x\hat{f}_n(x)-\nabla_xf(x)}{\max\{\hat{f}_n(x),\rho\}}\right\|_{x}^2f(x)\diff x\right] &\le \frac{\|f\|_{L^{\infty}(\Mcal)}}{\rho^2}\,\Ebb\left[\int_{\Mcal}\left\|\nabla_x\hat{f}_n(x)-\nabla_xf(x)\right\|^2_x\diff x\right] \\
        &\lesssim_{\Mcal,r} \frac{\|g\|_{L^{\infty}(\Mcal)}}{\rho^2}\cdot\frac{\|f\|_{\Hcal^{r}(\Mcal)}^{\frac{2m+4}{2r+m}}\|g\|_{L^{\infty}(\Mcal)}^{\frac{2r-2}{2r+m}}}{n^{\frac{2r-2}{2r+m}}} \\
        &= \frac{\|f\|_{\Hcal^{r}(\Mcal)}^{\frac{2m+4}{2r+m}}\|g\|_{L^{\infty}(\Mcal)}^{\frac{4r+m-2}{2r+m}}}{\rho^2n^{\frac{2r-2}{2r+m}}}. \\
        \end{split}
    \end{equation}
    For the second term, use $|f(x)-\max\{\hat{f}_n(x),\rho\}|\le|f(x)-\hat{f}_n(x)|$, then Lemma~\ref{lem:f-diff-1} and Lemma~\ref{lem:f-lower-bd} to get
    \begin{equation*}
        \left\|\frac{\nabla_x f(x)}{f(x)}\right\|^2\left|\frac{f(x) - \max\{\hat{f}_n(x),\rho\}}{\max\{\hat{f}_n(x),\rho\}}\right|^2f(x) \lesssim_{\Mcal} \frac{\|g\|^3_{L^{\infty}(\Mcal)}}{\rho^4\sigma^2}|f(x)-\hat{f}_n(x)|^2.
    \end{equation*}
    Therefore, we can use Proposition~\ref{prop:kde-error} to further bound:
    \begin{equation}\label{eqn:EB-3-new}
        \begin{split}
        &\Ebb\left[\int_{\Mcal}\left\|\frac{\nabla_x f(x)}{f(x)}\right\|^2\left|\frac{f(x) - \max\{\hat{f}_n(x),\rho\}}{\max\{\hat{f}_n(x),\rho\}}\right\|^2_{x}f(x)\diff x\right] \\
        &\lesssim_{\Mcal} \frac{\|g\|^3_{L^{\infty}(\Mcal)}}{\rho^4\sigma^2}\Ebb\left[\int_{\Mcal}|f(x)-\hat{f}_n(x)|^2\diff x\right] \\
        &\lesssim_{\Mcal,r} \frac{\|g\|^3_{L^{\infty}(\Mcal)}}{\rho^4\sigma^2}\cdot \frac{\|f\|_{\Hcal^{r}(\Mcal)}^{\frac{2m}{2r+m}}\|g\|_{L^{\infty}(\Mcal)}^{\frac{2r}{2r+m}}}{n^{\frac{2r}{2r+m}}} \\
        &= \frac{\|f\|^{\frac{2m}{2r+m}}_{\Hcal^{r}(\Mcal)}\|g\|^{\frac{8r+3m}{2r+m}}_{L^{\infty}(\Mcal)}}{\rho^4\sigma^2n^{\frac{2r}{2r+m}}}.
        \end{split}
    \end{equation}
    Combining \eqref{eqn:EB-2-new} and \eqref{eqn:EB-3-new} with \eqref{eqn:EB-1}, we obtain
    \begin{equation*}
        \begin{split}
        \Ebb&\left[\int_{\Mcal}d^2\left(\hat{\delta}_{\T}(x),\delta_{\T}(x)\right)f(x)\diff x\right] \\
        &\qquad\lesssim_{\Mcal,r}
        \sigma^4
        \frac{\|f\|_{\Hcal^{r}(\Mcal)}^{\frac{2m+4}{2r+m}}
        \|g\|_{L^{\infty}(\Mcal)}^{\frac{4r+m-2}{2r+m}}}{\rho^2n^{\frac{2r-2}{2r+m}}}
        +
        \sigma^2
        \frac{\|f\|_{\Hcal^{r}(\Mcal)}^{\frac{2m}{2r+m}}
        \|g\|_{L^{\infty}(\Mcal)}^{\frac{8r+3m}{2r+m}}}{\rho^4n^{\frac{2r}{2r+m}}}.
        \end{split}
    \end{equation*}
    In order to take the worst-case behavior on the right side in all variables, note that we have: 
    \begin{align*}
        \sigma^2 &\le (\diam(\Mcal))^2 \\
        n&\ge 1 \\
        \|g\|_{L^\infty(\Mcal)} &\ge (\vol(\Mcal))^{-1} \\
        \|f\|_{\Hcal^r(\Mcal)}&\ge \|f\|_{L^2(\Mcal)}\ge (\vol(\Mcal))^{-1/2} \\
        \rho&\le \min_{x\in\Mcal}f(x)\le (\vol(\Mcal))^{-1},
    \end{align*}
    hence we obtain the upper bound
    \begin{equation*}
        \Ebb\left[\int_{\Mcal}d^2\left(\hat{\delta}_{\T}(x),\delta_{\T}(x)\right)f(x)\diff x\right] \lesssim_{\Mcal,r} \frac{\sigma^{2}\|f\|^{\frac{2m+4}{2r+m}}_{\Hcal^{r}(\Mcal)}\|g\|^{\frac{8r+3m}{2r+m}}_{L^{\infty}(\Mcal)}}{\rho^4n^{\frac{2r-2}{2r+m}}}  \end{equation*}
    as claimed.
\end{proof}

\begin{proof}[Proof of Theorem~\ref{thm:S1-lower-bd}]
    Before the main part of the proof, we briefly introduce, for each $p\in\Nbb$, a collection of functions $\psi_{p,0},\ldots, \psi_{p,2^{p}-1}:\mathbb{S}^1\to\Rbb$ that will feature in our construction.
    The existence of these functions involves standard properties of wavelets on $\mathbb{S}^1$ which can be found in \cite[Theorem~7.2]{WalterShen} or \cite[Section~2]{Balid2009}.
    Specifically, we may construct such functions so that, for all $a_0,\ldots a_{2^p-1}\in\Rbb$, we have
    \begin{equation}\label{eqn:wavelet-L2}
        \left\|\sum_{z=0}^{2^p-1}a_z\psi_{p,z}\right\|_{L^2(\mathbb{S}^1)}^2\asymp\sum_{z=0}^{2^p-1}a_z^2
    \end{equation}
    and, for all $\varepsilon_0,\ldots \varepsilon_{2^p-1}\in\{0,1\}$, we have 
    \begin{equation}\label{eqn:wavelet-Linf}
        \left\|\sum_{z=0}^{2^p-1}\varepsilon_z\psi_{p,z}\right\|_{L^{\infty}(\mathbb{S}^1)}\lesssim 2^{p/2}.
    \end{equation}   
    In fact, we further have the property that, for all $z$ we have $\psi_{p,z}^{\ast}(k)\neq 0$ only if $|k|\asymp 2^p$, which allows us to generalize \eqref{eqn:wavelet-L2} as follows; for all $r\in\Rbb$ and all $a_0,\ldots a_{2^p-1}\in\Rbb$, we have \begin{equation}\label{eqn:wavelet-Hr}
        \begin{split}
        \left\|\sum_{z=0}^{2^p-1}a_z\psi_{p,z}\right\|_{\Hcal^r(\mathbb{S}^1)}^2 = \sum_{k\in\Zbb}(1+k^2)^r\left|\sum_{z=0}^{ 2^p-1}a_z\psi_{p,z}^{\ast}(k)\right|^2 &\asymp_r 2^{2pr}\sum_{k\in\Zbb}\left|\sum_{z=0}^{2^p-1} a_z\psi_{p,z}^{\ast}(k)\right|^2 \\
        &=2^{2pr}\left\|\sum_{z=0}^{2^p-1}a_z\psi_{p,z}\right\|^2_{L^2(\mathbb{S}^1)}\asymp2^{2pr}\sum_{z=0}^{2^p-1}a_z^2,    
        \end{split}
    \end{equation}
    for all $p$.

    Next we define the functions $\zeta_{p,0},\ldots, \zeta_{p,2^p-1}:\mathbb{S}^1\to\Rbb$ as follows, for each $z$:
    \begin{equation*}
        \zeta_{p,z}(x)= \int_{\mathbb{S}^1}p_{\theta,\sigma^2}(x)\psi_{p,z}(\theta)\diff \theta = \int_{\mathbb{S}^1}\exp\left(-\frac{d^2(x,\theta)}{2\sigma^2}-A(\sigma^2)\right)\psi_{p,z}(\theta)\diff \theta.
    \end{equation*}
    Note that $\zeta_{p,z}$ has the interpretation of the ``density'' of $X$ when $\psi_{p,z}$ is the ``density'' of $\Theta$ in \eqref{eqn:pop-model}, although these functions can take negative values so are not bona fide densities.
    As we showed in the proof of Lemma~\ref{lem:Sobolev-bd-S1}, we have $\zeta_{p,z}^{\ast}(k) \asymp_{\sigma^2} (1+k^2)^{-1}\psi_{p,z}^{\ast}(k)$ for all $|k|\asymp 2^p$, from which we can deduce bounds similar to \eqref{eqn:wavelet-Hr}.
    Indeed, just compute
    \begin{equation}\label{eqn:zeta-Hr}
        \left\|\sum_{z=0}^{2^p-1}a_z\zeta_{p,z}\right\|_{L^2(\mathbb{S}^1)}^2 =\left\|\sum_{z=0}^{2^p-1}a_z\zeta_{p,z}\right\|_{\Hcal^{0}(\mathbb{S}^1)}^2 \asymp_{\sigma^2} \left\|\sum_{z=0}^{2^p-1}a_z\psi_{p,z}\right\|_{\Hcal^{-2}(\mathbb{S}^1)}^2 \asymp 2^{-4p}\sum_{z=0}^{2^p-1}a_z^2
    \end{equation}
    and
    \begin{equation}\label{eqn:grad-zeta-Hr}
        \begin{split}
        \left\|\sum_{z=0}^{2^p-1}a_z\nabla\zeta_{p,z}\right\|_{L^2(\mathbb{S}^1)}^2 = \sum_{k\in\Zbb}
        k^2\left|\sum_{z=0}^{2^p-1}a_z\zeta_{p,z}^{\ast}(k)\right|^2 &\asymp 2^{2p}\sum_{k\in\Zbb}\left|\sum_{z=0}^{2^p-1}a_z\zeta_{p,z}^{\ast}(k)\right|^2 \\
        &= 2^{2p}\left\|\sum_{z=0}^{2^p-1}a_z\zeta_{p,z}\right\|^2_{L^2(\mathbb{S}^1)} \asymp 2^{-2p}\sum_{z=0}^{2^p-1}a_z^2
        \end{split}
    \end{equation}
    for all $p$.

    Now we set up the proof using Assouad's method.
    To do this, fix $p\in\Nbb$ such that $2^p\asymp n^{1/(2s+5)}$ and fix $\gamma>0$ such that $\gamma\asymp L2^{-p(s+\frac{1}{2})}=Ln^{-(s+\frac{1}{2})/(2s+5)}$.
    Then, for each $\varepsilon = (\varepsilon_0,\ldots, \varepsilon_{2^p-1})\in\{0,1\}^{2^p}$, set
    \begin{equation*}
        g_{\varepsilon}(\theta):= \frac{1}{2\pi}+\gamma\sum_{z=0}^{2^p-1}\varepsilon_z\psi_{p,z}(\theta).
    \end{equation*}
    In order to make sure that $g_{\varepsilon}\in \mathcal{G}_s(\alpha,\beta,L)$ for all $\varepsilon\in\{0,1\}^{2^p}$, note from \eqref{eqn:wavelet-Linf} and $\gamma\lesssim L2^{-p\left(s+\frac{1}{2}\right)}\lesssim 2^{-p/2}$ that we have $\alpha\le g_{\varepsilon}(\theta)\le \beta$ for all $\theta\in\mathbb{S}^1$, and similarly that from \eqref{eqn:wavelet-Hr} and $\gamma\le L2^{-p\left(s+\frac{1}{2}\right)}$ we have $\|g_{\varepsilon}\|_{\Hcal^s(\mathbb{S}^1)}\le L$.
    Now for convenience, write $f_{\varepsilon}$ for the marginal density of $X$ when $\Theta$ has density $g_{\varepsilon}$ in \eqref{eqn:pop-model}, so that $f_{\varepsilon} = (2 \pi)^{-1}+\gamma\sum_{z=0}^{2^p-1}\varepsilon_z\zeta_{p,z}$.
    The core of the rest of the proof is to show that $\{f_{\varepsilon}:\varepsilon\in \{0,1\}^{2^p}\}$ are close in terms of their KL divergence but far in terms of the displacement of their denoisers.
    Throughout the rest of the proof, we write $d_{\mathrm{H}}(\varepsilon,\varepsilon')$ for the Hamming distance between $\varepsilon,\varepsilon'\in\{0,1\}^{2^p}$.

    For the KL computation, suppose $d_{\mathrm{H}}(\varepsilon,\varepsilon')=1$, hence without loss of generality $f_{\varepsilon}-f_{\varepsilon'} = \gamma\zeta_{p,z_0}$ for some $z_0$.
    Then, use Lemma~\ref{lem:f-lower-bd} to get $f_{\varepsilon'}(x)\gtrsim \alpha$ for all $x\in\mathbb{S}^1$, and use \eqref{eqn:zeta-Hr} to compute
    \begin{equation*}
        \KL(f_{\varepsilon}\,|\,f_{\varepsilon'}) \lesssim_{\alpha}\|f_{\varepsilon}-f_{\varepsilon'}\|^2_{L^2(\mathbb{S}^1)} =  \gamma^2\left\|\zeta_{p,z_0}\right\|^2_{L^2(\mathbb{S}^1)}\asymp\gamma^22^{-4p}
        \end{equation*}
    Further, if we let $P_{\varepsilon}$ be the joint distribution of i.i.d. samples $X_1,\ldots, X_n$ from $f_{\varepsilon}$, and similarly for $P_{\varepsilon'}$, we have
    \begin{equation}\label{eqn:KL-computation}
        \KL(P_{\varepsilon}\,|\,P_{\varepsilon'}) = n\KL(f_{\varepsilon}\,|\,f_{\varepsilon'}) \lesssim_{\alpha}n\gamma^22^{-4p} 
        \end{equation}
    for all $n$.
    (Note that we chose $\gamma\asymp L2^{-p(s+\frac{1}{2})}$ and $2^p\asymp n^{1/(2s+5)}$ precisely to guarantee that these KL divergences on adjacent hypothesis are bounded by a small universal constant).

    For the displacement computation, we claim
    \begin{equation}\label{eqn:displacement-computation}
        \int_{\mathbb{S}^1}d^2(\delta_{\T\varepsilon}(x),\delta_{\T\varepsilon'}(x))\diff x \gtrsim_{\sigma^2,\alpha,\beta}\gamma^22^{-2p}d_{\mathrm{H}}(\varepsilon,\varepsilon')
    \end{equation}
    for all $\varepsilon,\varepsilon'\in\{0,1\}^{2^p}$, where $\delta_{\T\varepsilon}(x)= \Exp_x(\sigma^2\nabla \log f_{\varepsilon}(x))$ and $\delta_{\T\varepsilon'}(x)= \Exp_x(\sigma^2\nabla \log f_{\varepsilon'}(x))$ denote the oracle tangential Bayes denoisers when $\Theta$ has distribution with density given by $g_{\varepsilon}$ and $g_{\varepsilon'}$, respectively.
    To do this, we first show that for all $0<\sigma^2<\pi^2$ there exists a constant $0<C_{\alpha,\sigma^2}<\pi$ such that we have
    \begin{equation}\label{eqn:S1-tan-EB-bd}
        \sup_{\varepsilon\in\{0,1\}^{2^p}}\sup_{x\in\mathbb{S}^1}\|\Log_x(\delta_{\T\varepsilon}(x))\|_x\le C_{\alpha,\sigma^2}
    \end{equation}
    To see this, use the Tweedie-Eddington formula (Lemma~\ref{lem:Tweedie}), Jensen's inequality, and the identity, $\|\Log_x\theta\|_x=d(x,\theta)$ to obtain
    \begin{align*}
        \|\Log_x(\delta_{\T\varepsilon}(x))\|_x = \left\|\Ebb[\Log_x\Theta\,|\,X=x]\right\|_x&\le \Ebb[\|\Log_x\Theta\|_x\,|\,X=x] \\
        &= \Ebb[d(x,\Theta)\,|\,X=x] \\
        &= \frac{\int_{-\pi}^{\pi}d(x,\theta)\exp\left(-\frac{d^2(x,\theta)}{2\sigma^2}\right)g_{\varepsilon}(\theta)\diff\theta}{\int_{-\pi}^{\pi}\exp\left(-\frac{d^2(x,\theta)}{2\sigma^2}\right)g_{\varepsilon}(\theta)\diff\theta} \\
        &\le \pi - (\pi-\sigma)2\alpha\sigma\exp\left(-\frac{\alpha^2}{2\sigma^2}\right) =: C_{\alpha,\sigma^2},
    \end{align*}
    where the last line follows by direct calculation after splitting the integral into $[-\sigma,\sigma]$ and its complement $[-\pi,\pi]\setminus[-\sigma,\sigma]$.
    Next note that~\eqref{eqn:S1-tan-EB-bd} implies that there exists a constant $c_{\alpha,\sigma^2}>0$ such that we have
    \begin{equation*}
        d\left(\delta_{\T\varepsilon}(x),\delta_{\T\varepsilon'}(x)\right) \ge c_{\alpha,\sigma^2}\|\nabla \log f_{\varepsilon}(x)-\nabla \log f_{\varepsilon'}(x)\|_x,
    \end{equation*}
    for all $x\in\mathbb{S}^1$, hence
    \begin{equation*}
        \Ebb\left[\int_{\mathbb{S}^1}d^2(\delta_{\T\varepsilon}(x),\delta_{\T\varepsilon'}(x))\diff x\right] \gtrsim_{\alpha,\sigma^2} \Ebb\left[\int_{\mathbb{S}^1}\|\nabla \log f_{\varepsilon}(x)-\nabla \log f_{\varepsilon'}(x)\|^2_x\diff x\right].
    \end{equation*}
    To further this lower bound, use Lemma~\ref{lem:f-lower-bd} to get $\alpha \lesssim f_{\varepsilon}(x)\lesssim\beta$ for all $x\in\mathbb{S}^1$ and all $\varepsilon,\varepsilon'\in\{0,1\}^{2^p}$, hence
    \begin{align*}
        &\left\|\nabla \log f_{\varepsilon}-\nabla \log f_{\varepsilon'}\right\|_{L^2(\mathbb{S}^1)}^2 \\
        &= \left\|\frac{f_{\varepsilon}\nabla f_{\varepsilon'}-f_{\varepsilon'}\nabla f_{\varepsilon}}{f_{\varepsilon}f_{\varepsilon'}}\right\|_{L^2(\mathbb{S}^1)}^2 \\
        &\gtrsim_{\beta}\left\|f_{\varepsilon}\nabla f_{\varepsilon'}-f_{\varepsilon'}\nabla f_{\varepsilon}\right\|_{L^2(\mathbb{S}^1)}^2 \\
        &= \left\|f_{\varepsilon}\left(\nabla f_{\varepsilon}+\gamma\sum_{z=0}^{2^p-1}(\varepsilon'_z-\varepsilon_z)\nabla \zeta_{p,z}\right)-\left(f_{\varepsilon}+\gamma\sum_{z=0}^{2^p-1}(\varepsilon'_z-\varepsilon_z)\zeta_{p,z}\right)\nabla f_{\varepsilon}\right\|_{L^2(\mathbb{S}^1)}^2 \\
        &= \gamma^2\left\|f_{\varepsilon}\sum_{z=0}^{2^p-1}(\varepsilon'_z-\varepsilon_z)\nabla\zeta_{p,z} -\left(\sum_{z=0}^{2^p-1}(\varepsilon'_z-\varepsilon_z)\zeta_{p,z}\right)\nabla f_{\varepsilon}\right\|_{L^2(\mathbb{S}^1)}^2 \\
        &\ge \gamma^2\left(\left\|f_{\varepsilon}\sum_{z=0}^{2^p-1}(\varepsilon'_z-\varepsilon_z)\nabla\zeta_{p,z}\right\|_{L^2(\mathbb{S}^1)} -\left\|\left(\sum_{z=0}^{2^p-1}(\varepsilon'_z-\varepsilon_z)\zeta_{p,z}\right)\nabla f_{\varepsilon}\right\|_{L^2(\mathbb{S}^1)}\right)^2.
    \end{align*}
    Now bound both of the terms on the right side as follows.
    For the first term, use $f_{\varepsilon}(x)\gtrsim \alpha$ then~\eqref{eqn:grad-zeta-Hr} to obtain
    \begin{equation*}
     \left\|f_{\varepsilon}\sum_{z=0}^{2^p-1}(\varepsilon'_z-\varepsilon_z)\nabla\zeta_{p,z}\right\|_{L^2(\mathbb{S}^1)}\gtrsim_{\alpha} \left\|\sum_{z=0}^{2^p-1}(\varepsilon'_z-\varepsilon_z)\nabla\zeta_{p,z}\right\|_{L^2(\mathbb{S}^1)}\asymp 2^{-2p}\sum_{z=0}^{2^p-1}(\varepsilon_z'-\varepsilon_z)^2.
    \end{equation*}
    For the second term, use Lemma~\ref{lem:f-diff-1} and~\eqref{eqn:zeta-Hr} to get 
    \begin{align*}
        \left\|\left(\sum_{z=0}^{2^p-1}(\varepsilon'_z-\varepsilon_z)\zeta_{p,z}\right)\nabla f_{\varepsilon}\right\|^2_{L^2(\mathbb{S}^1)} &\le \left\|\sum_{z=0}^{2^p-1}(\varepsilon'_z-\varepsilon_z)\zeta_{p,z}\right\|^2_{L^2(\mathbb{S}^1)}\|\nabla f_{\varepsilon}\|^2_{L^{\infty}(\mathbb{S}^1)} \\ &\lesssim_{\sigma^2} \left\|\sum_{z=0}^{2^p-1}(\varepsilon'_z-\varepsilon_z)\zeta_{p,z}\right\|^2_{L^2(\mathbb{S}^1)}\|g_{\varepsilon}\|_{L^{\infty}(\mathbb{S}^1)}^2 \\
        &\asymp_{\beta} 2^{-4p}\sum_{z=0}^{2^p-1}(\varepsilon'_z-\varepsilon_z)^2.
    \end{align*}
    Now combine these displays and note that $p\to\infty$, so the dominant term between $2^{-2p}$ and $2^{-4p}$ is the slower one, i.e., $2^{-2p}$.
    Thus,
    \begin{equation*}
        \left\|\nabla \log f_{\varepsilon}-\nabla \log f_{\varepsilon'}\right\|_{L^2(\mathbb{S}^1)}^2 \gtrsim_{\sigma^2,\alpha,\beta}\gamma^2\sum_{z=0}^{2^p-1}(\varepsilon'_z-\varepsilon_z)^2\left(2^{-2p}-2^{-4p}\right)^2 \gtrsim \gamma^22^{-2p}d_{\mathrm{H}}(\varepsilon,\varepsilon')
    \end{equation*}
    for all $\varepsilon,\varepsilon'\in\{0,1\}^{2^p}$, which is exactly~\eqref{eqn:displacement-computation}.

    To put the pieces together, we use Assouad's lemma (e.g., \cite[Theorem~31.2]{Polyanskiy_Wu_BOOK}) and Pinsker's inequality. Using the bounds in~\eqref{eqn:KL-computation} and~\eqref{eqn:displacement-computation} and the specification     $\gamma\asymp L2^{-p(s+\frac{1}{2})}$ and $2^p\asymp n^{1/(2s+5)}$, we obtain
    \begin{align*}
        &\inf_{\hat{\delta}}\sup_{g\in \mathcal{G}_s(\alpha,\beta,L)}\Ebb\left[\int_{\mathbb{S}^1}d^2(\hat{\delta}(x),\delta_{\T}(x))\diff x\right] \\
        &\gtrsim_{\sigma^2,\alpha,\beta} \gamma^22^{-2p}\cdot 2^p\left(1-\max_{d_{\mathrm{H}}(\varepsilon,\varepsilon')=1}\sqrt{\KL(P_{\varepsilon}|P_{\varepsilon'})}\right)\\
        &\asymp_{L} n^{-\frac{2s+1}{2s+5}}2^{-p}\left(1-n^{1/2}\gamma2^{-2p}\right)\\
        &\asymp n^{-\frac{2s+2}{2s+5}}.
    \end{align*}
    Finally, use Lemma~\ref{lem:f-lower-bd} to get $f(x)\gtrsim \alpha$ for all $x\in\mathbb{S}^1$ and all $g\in \mathcal{G}_s(\alpha,\beta,L)$, hence
    \begin{align*}
        &\inf_{\hat{\delta}}\sup_{g\in \mathcal{G}_s(\alpha,\beta,L)}\Ebb\left[\int_{\mathbb{S}^1}d^2(\hat{\delta}(x),\delta_{\T}(x))f(x)\diff x\right] \\
        &\gtrsim_{\alpha}\inf_{\hat{\delta}}\sup_{g\in \mathcal{G}_s(\alpha,\beta,L)}\Ebb\left[\int_{\mathbb{S}^1}d^2(\hat{\delta}(x),\delta_{\T}(x))\diff x\right] \gtrsim_{\sigma^2,\alpha,\beta,L} n^{-\frac{2s+2}{2s+5}}.
    \end{align*}
    This finishes the proof.
\end{proof}

\begin{proof}[Proof of Lemma~\ref{lem:risk-to-dist}]
    Recall \cite[Lemma~2.3]{EvansJaffeSLLN} that for each $\varepsilon>0$ there exists a constant $C_{\varepsilon}>0$ such that we have $d^2(x,z)\le (1+\varepsilon)d^2(x,y) + C_{\varepsilon}d^2(y,z)$ for all $x,y,z\in\Mcal$.
    Thus, for each $x,\theta\in\Mcal$ we have
    \begin{equation*}
        d^2\left(\delta(x),\theta\right)\le (1+\varepsilon)d^2(\delta_{\T}(x),\theta) + C_{\varepsilon}d^2\left(\delta(x),\delta_{\T}(x)\right).
    \end{equation*}
    Now integrate this over the joint distribution of $(\Theta,X)$ to get the result.
\end{proof}

\section{Implementation Details}\label{app:implementation}

In this section, we describe the details of implementing our denoising procedures.
Specifically, in Subsection~\ref{subsec:compute-oracle} we discuss approximately computing the oracle denoisers $\delta_{\B}$ and $\delta_{\T}$, in Subsection~\ref{subsec:compute-empirical} we discuss computing the empirical tangential Bayes denoiser $\hat{\delta}_{\T}$, and in Subsection~\ref{subsec:score-matching} we discuss the hyperparameter selection via the Riemannian score-matching objective.

\subsection{Approximating the Oracle Denoisers}\label{subsec:compute-oracle}

In the simulations of the main body, both oracle denoisers $\delta_{\B}$ and $\delta_{\T}$ are approximated by utilizing additional oracle samples from $G$.
In this subsection we explain this approximation and its computation in more detail.

\subsubsection{Oracle Bayes Denoiser}

First, we consider computing the oracle Bayes denoiser $\delta_{\B}(x)$ for fixed $x\in\Mcal$, defined implicitly via the optimization problem in \eqref{eqn:def-deltaB}.
To approximate it, we sample i.i.d pairs $(\Theta_1,X_1),\ldots, (\Theta_n,X_n)$ from the model~\eqref{eqn:pop-model}, and we estimate the conditional Fr\'echet mean of $\Theta$ given $\{X=x\}$ via Fr\'echet Nadaraya-Watson regression; that is, we use \texttt{Geomstats} to compute the Fr\'echet mean of the probability distribution
\begin{equation*}
    \sum_{i=1}^{n}w_i(x)\delta_{\Theta_i}
\end{equation*}
where
\begin{equation}
w_i(x)
=
\exp\!\left(-\frac{d^2(X_i,\,x)}{2h^2}\right),
\end{equation}
for some bandwidth parameter $h>0$ selected by cross-validation.

\subsubsection{Oracle Tangential Bayes Denoiser}

Second, we consider computing the oracle tangential Bayes denoiser $\delta_{\T}(x)$ for fixed $x\in\Mcal$, defined in~\eqref{eqn:def-deltaT}.
Since Lemma~\ref{lem:Tweedie} allows us to write
\begin{equation*}
    \delta_{\T}(x) = \Exp_x\left(\sigma^2\frac{\nabla f(x)}{f(x)}\right),
\end{equation*}
it follows that a straightforward approximation of $\delta_{\T}(x)$ is given as follows; sample i.i.d random variables $\Theta_1,\ldots, \Theta_n$ from $G$, and then compute
\begin{equation*}
    \Exp_x\left(\frac{\frac{1}{n}\sum_{i=1}^{n}\Log_x(\Theta_i)p_{\Theta_i,\sigma^2}(x)}{\frac{1}{n}\sum_{i=1}^{n}p_{\Theta_i,\sigma^2}(x)}\right).
\end{equation*}
If $\Mcal$ is homogeneous, then the normalizing constants above all cancel and we can simply compute
\begin{equation*}
    \Exp_x\left(\frac{\frac{1}{n}\sum_{i=1}^{n}\Log_x(\Theta_i)\exp\left(-\frac{d^2(\Theta_i,x)}{2\sigma^2}\right)}{\frac{1}{n}\sum_{i=1}^{n}\exp\left(-\frac{d^2(\Theta_i,x)}{2\sigma^2}\right)}\right)
\end{equation*}
In order to reduce the variance of this approximation in practice (recall that all our applications in practice are homogeneous), we fix a fine grid $S\subseteq\Mcal$ and we compute
\begin{equation*}
    \Exp_x\left(\frac{\sum_{\theta\in S}\Log_x(\theta)w(\theta)}{\sum_{\theta\in S}w(\theta)}\right)
\end{equation*}
where
\begin{equation*}
    w(\theta) := \frac{1}{n}\sum_{i=1}^{n}\exp\left(-\frac{d^2(\theta,x)}{2\sigma^2}\right).
\end{equation*}

\subsection{Computing Empirical Tangential Bayes Denoiser}\label{subsec:compute-empirical}

Even computing the empirical tangential Bayes denoiser $\hat{\delta}_{\T}$, as in the simulations and applications in the main body, requires some work.
The complication is that the expression~\eqref{eq:Est-Denoiser} for $\hat{\delta}_{\T}$ involves gradients of $\hat{f}_n$ which by~\eqref{eq:Density-Est} involve gradients of the eigenfunctions $\{\phi_j\}_{j=0}^{\infty}$ of the Laplace-Beltrami operator $\Delta$.
While there is no closed-form expression for such gradients in general, they can be computed by embedding $\Mcal$ into $\Rbb^k$ for some $k$ and then applying the well-known result (see~\cite[Chapter~3]{Boumal}) that the Riemannian gradient of a function $h:\Mcal\to\Rbb$ at $x\in\Mcal\subseteq\Rbb^k$ is just the projection of the ambient (Euclidean) gradient $\nabla h(x)\in\Rbb^k$ onto the (affine) tangent space $x+\Tan_x(\Mcal)\subseteq\Rbb^k$.

In the case of the circle $\Mcal=\mathbb{S}^1$, we may directly differentiate the Dirichlet kernel \eqref{eqn:Dirichlet-kernel}, which is equivalent to embedding $\mathbb{S}^1\subseteq\Cbb$ as the unit circle.

In the case of the sphere $\Mcal=\mathbb{S}^2$, we embed $\mathbb{S}^2\subseteq\Rbb^3$ via $\mathbb{S}^2=\{x\in\Rbb^3:\|x\|=1\}$.
Then use the summation formula to get that suitable sums of eigenfunctions are equal to suitable Legendre polynomials; we then differentiate these polynomials and multiply the result by the matrix $I_3-xx^{\top}$ which corresponds to the projection onto $\Tan_x(\mathbb{S}^2)$.

\subsection{Hyperparameter Selection}\label{subsec:score-matching}

In the applications of Section~\ref{sec:appl} where we do not have oracle information to guide our choices of $\ell_n$ and $\rho$, we instead selected these hyperparameters with the data-driven method of cross-validation, where the objective is an adaptation of Hyv\"arinen's score matching objective \cite{hyvarinen2005} to the Riemannian setting \cite{de2022riemannian}.
In this subsection we explain this in further detail.

\subsubsection{Derivation of Riemannian Score Matching}

First, we explain how to derive the objective $\hat{\mathcal{J}}$ from~\eqref{eqn:Hyvarinen-Riemannian}.
Note that our empirical tangential Bayes denoiser directly involves the vector field $\nabla \log f(x)$ on $\Mcal$, where $f:\Mcal\to\Rbb$ is the marginal density of $X$, hence a natural loss function for a candidate density $\tilde f:\Mcal\to\Rbb$ is
\begin{equation}
\mathcal{L}(\tilde f)
=
\mathbb{E}\big[
\|\nabla \log \tilde{f}(X) - \nabla \log f(X)\|_X^2
\big].
\end{equation}
Of course, this cannot be computed in practice since it involves the unknown $f$.
Nonetheless, we can expand the square and then use integration by parts to get
\begin{align*}
    \mathcal{L}(\tilde{f}) &= \mathbb{E}\big[\|\nabla \log \tilde f (X)\|_X^2\big] - 2\,\mathbb{E}\big[\langle \nabla \log \tilde f(X), \nabla \log f(X)\rangle_X\big] + \mathbb{E}\big[\|\nabla \log f (X)\|_X^2\big] \\
    &= \mathbb{E}\big[\|\nabla \log \tilde f (X)\|_X^2\big] + 2\,\mathbb{E}\big[\Delta \log \tilde{f}(X)\big] + \mathbb{E}\big[\|\nabla \log f (X)\|_X^2\big].
\end{align*}
Notice that the last term does not depend on $\tilde{f}$, hence optimizing $\mathcal{L}$ is equivalent to optimizing $\mathcal{J}$.

Next note that $\mathcal{J}(\tilde{f})$ still depends on $f$, but only implicitly through the integral over the marginal distribution of $X$.
By estimating this integral via the natural Monte Carlo estimate built from the measurements $X_1,\ldots, X_n$, we are led to consider
\begin{equation*}
    \hat{\mathcal{J}}(\tilde{f}) =
    \frac{1}{n}\sum_{i=1}^n \left(\|\nabla \log \tilde{f}(X_i)\|_{X_i}^2 + 2\Delta \log \tilde{f}(X_i)\right),
\end{equation*}
which is an unbiased estimator of $\mathcal{J}(\tilde{f})$ for each fixed $\tilde{f}$. 
This is precisely the objective we defined in~\eqref{eqn:Hyvarinen-Riemannian} in the main body.

\subsubsection{Computing Riemannian Score Matching Objective}

For general $\tilde{f}:\Mcal\to\Rbb$ we cannot easily compute $\hat{\mathcal{J}}(\tilde{f})$, but it turns out that the specific form~\eqref{eq:Density-Est} of our candidate densities $\hat{f}_n$ allows us to do this computation exactly.

First, we re-write the objective in terms of $\tilde{f}$ instead of $\log\tilde{f}$:
\begin{equation*}
    \hat{\mathcal{J}}(\tilde{f}) =
    \frac{1}{n}\sum_{i=1}^n \left(\|\nabla \log \tilde{f}(X_i)\|_{X_i}^2 + 2\Delta \log \tilde{f}(X_i)\right)=
    \frac{1}{n}\sum_{i=1}^n \left(\frac{\|\nabla \tilde{f}(X_i)\|_{X_i}^2}{\tilde{f}(X_i)} + \frac{2\Delta \tilde{f}(X_i)}{\tilde{f}(X_i)}\right),
\end{equation*}
As such, we need to evaluate derivatives of $\tilde{f}$ of order up to 2.

Now consider $\tilde{f}=\hat{f}_n$.
We know that computing $\hat{f}_n$ only requires computing sums of eigenfunctions of $\Delta$, and we showed in Subsection~\ref{subsec:compute-empirical} that in concrete cases of interest one can also exactly compute gradients of sums of eigenfunctions of $\Delta$.
In order to compute the second derivative, we simply note that $\Delta$ acts diagonally on the span of $\{\phi_j\}_{j=0}^{\infty}$, i.e., we have
\[
\Delta_x K(x,y,\ell)
=
\sum_{j\le \ell} \lambda_j \phi_j(x)\overline{\phi_j(y)},
\]
hence\[
\Delta \hat f(x)
=
\frac{1}{n}
\sum_{i=1}^n
\sum_{j\le \ell}
\lambda_j \phi_j(x)\overline{\phi_j(X_i)}.
\]

\subsubsection{Cross-Validation}

We perform $K$-fold cross-validation of the Riemannian score-matching objective in the usual way, namely by splitting the data into $K$ folds and for each fold computing the density estimate on the resulting training set and evaluating the score-matching objective $\hat{\mathcal{J}}(\hat{f}_n)$ on the complementary test set.
In the applications of Section~\ref{sec:appl} we chose $K=5$.

We also include a penalty in our cross-validation procedure, in order to prevent over-fitting.
We follow the usual Akaike information criterion (AIC) where we penalize proportional to the effective dimension of the model, which in our setting corresponds to the number of eigenfunctions of the Laplace-Beltrami operator. 
More precisely, we minimize
\begin{equation*}
\hat{\mathcal{J}}(\hat{f}_n)+
\frac{2\ell}{n}
\end{equation*}
which encourages more parsimonious models.

Lastly, we comment on the hyperparameter ranges over which we search.
For $\ell$, we set its upper bound for $\ell$ to be equal to the conservative oracle choice from Theorem~\ref{thm:EB-risk}.
For $\rho$, we parameterize it in terms of a percentile of the positive values taken by the density $\hat{f}_n$, and we set its upper bound as the 10 percentile.

\section{Additional Figures}\label{app:additional-figs}

In this section we include some further figures that were omitted from the main body for the sake of brevity.

\subsection{Visualizations from Subsection~\ref{subsec:sim-sphere}}

First, we provide visualizations of the denoising procedure for the examples of Subsection~\ref{subsec:sim-sphere} on the sphere $\Mcal=\mathbb{S}^2$; more precisely we show versions of Figure~\ref{fig:intro} (i.e., the latent variables, their measurements, the oracle and empirical denoisers, and the estimated score field.) when the latent variable distribution $G$ is taken to be uniform, supported on a cap, or supported on the equator.

In Figure~\ref{fig:S2_uniform} we show the uniform example.
As discussed in Example~\ref{ex:oracle-uniform}, the oracle denoisers $\delta_{\N},\delta_{\B}$, and $\delta_{\T}$ are all equal, and the score field $\nabla \log \hat{f}_n$ is approximately identically zero.
In this particular simulation (and some of the simulations below) we see that our approximation of $\delta_{\B}$ achieves a worse risk than all of $\delta_{\N},\delta_{\T}$, and $\hat{\delta}_{\T}$, which is because its numerical optimization is more difficult and can lead to sub-optimal denoising in practice.

    \begin{figure}[h!]
        \centering
        \includegraphics[width=0.9\linewidth]{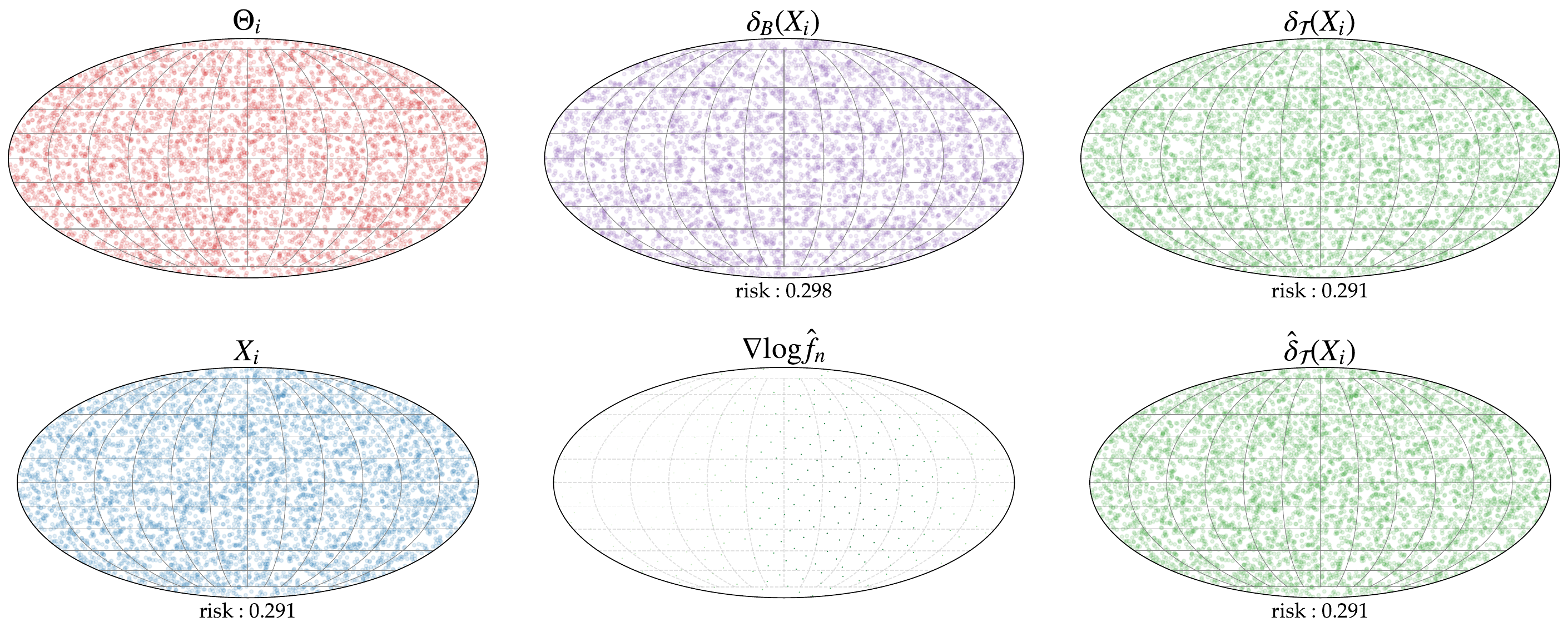}
        \caption{
        Illustration of the `uniform' example in the sphere $\mathbb{S}^2$.}
    \label{fig:S2_uniform}
    \end{figure}

    In Figure~\ref{fig:S2_cap} we show the cap example.
    We observe that the estimated score field $\nabla \log \hat{f}_n$ is  strongest near the boundary of the support of $G$, and that it is approximately identically zero away from the boundary; this is what allows $\hat{\delta}_{\T}$ to exhibit a sharp boundary in the denoised data set.    

    \begin{figure}[h!]
        \centering
        \includegraphics[width=0.9\linewidth]{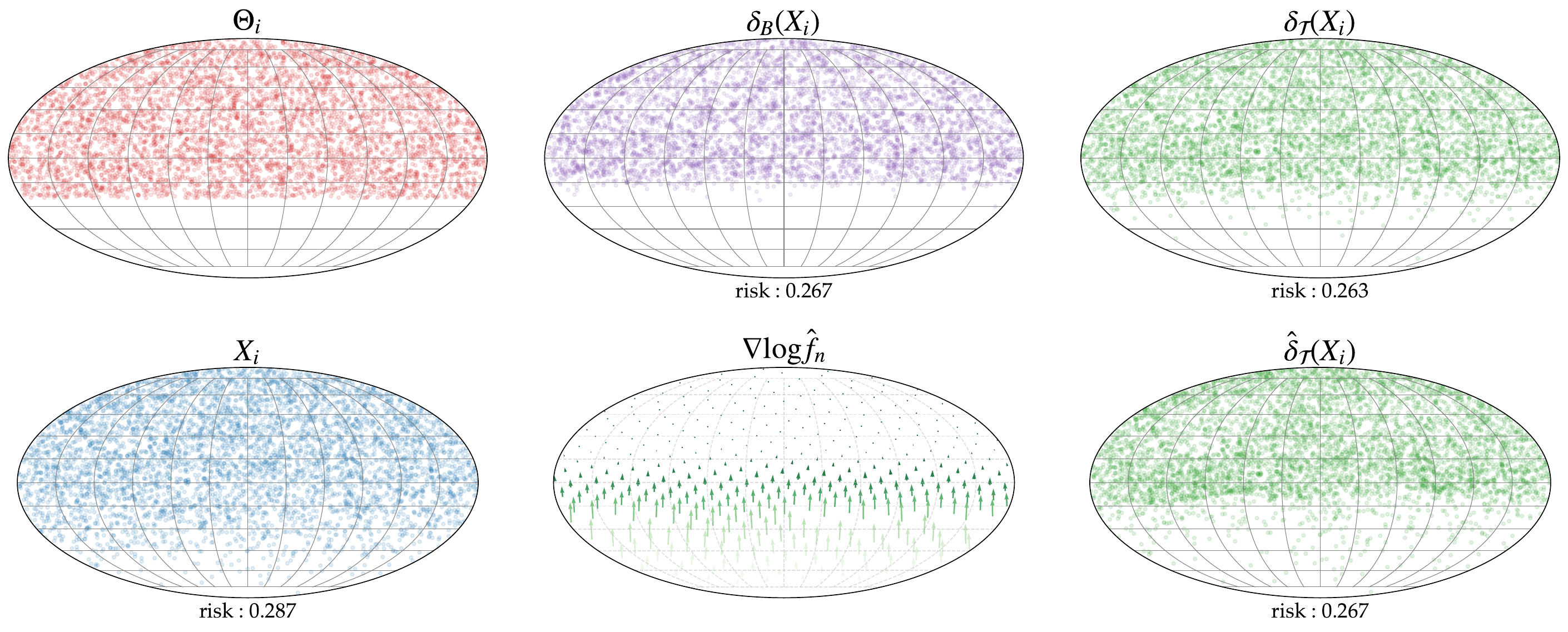}
        \caption{
        Illustration of the `cap' example in the sphere $\mathbb{S}^2$.}
    \label{fig:S2_cap}
    \end{figure}

    In Figure~\ref{fig:S2_equator} we show the equator example.
    Since the support $S\subseteq\mathbb{S}^2$ of $G$ is geodesically convex, we see that $\delta_{\B}(X_i)$ is contained in $S$ almost surely, meaning $\delta_{\B}$ exactly matches this latent one-dimensional sub-manifold structure.
    The denoisers
    $\delta_{\T}$ and $\hat{\delta}_{\T}$ approximately recover this structure, and have comparable risk.

    \begin{figure}[h!]
        \centering
        \includegraphics[width=0.9\linewidth]{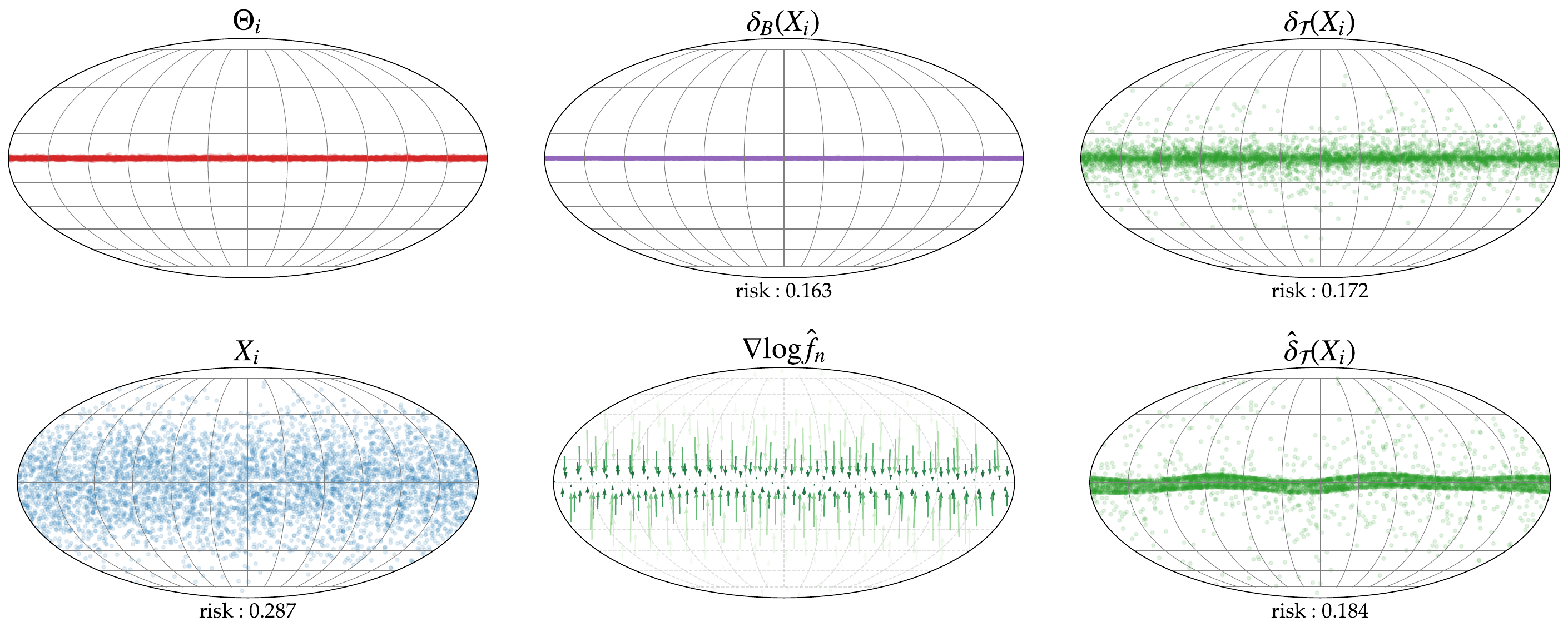}
        \caption{
        Illustration of the `equator' example in the sphere $\mathbb{S}^2$.}
    \label{fig:S2_equator}
    \end{figure}

    In Figure~\ref{fig:S2_10_model} we show the 10-mode example.
    Note that the measurements appear to be nearly uniformly-distributed on the sphere, but the denoisers recover the latent structure.
    Also observe that $\delta_{\B}(X_i)$ is concentrated on the union of the components of $G$ and the geodesics between them; this is reflected in $\delta_{\T}$ and $\hat{\delta}_{\T}$ as well. 

    \begin{figure}[h!]
        \centering
        \includegraphics[width=0.9\linewidth]{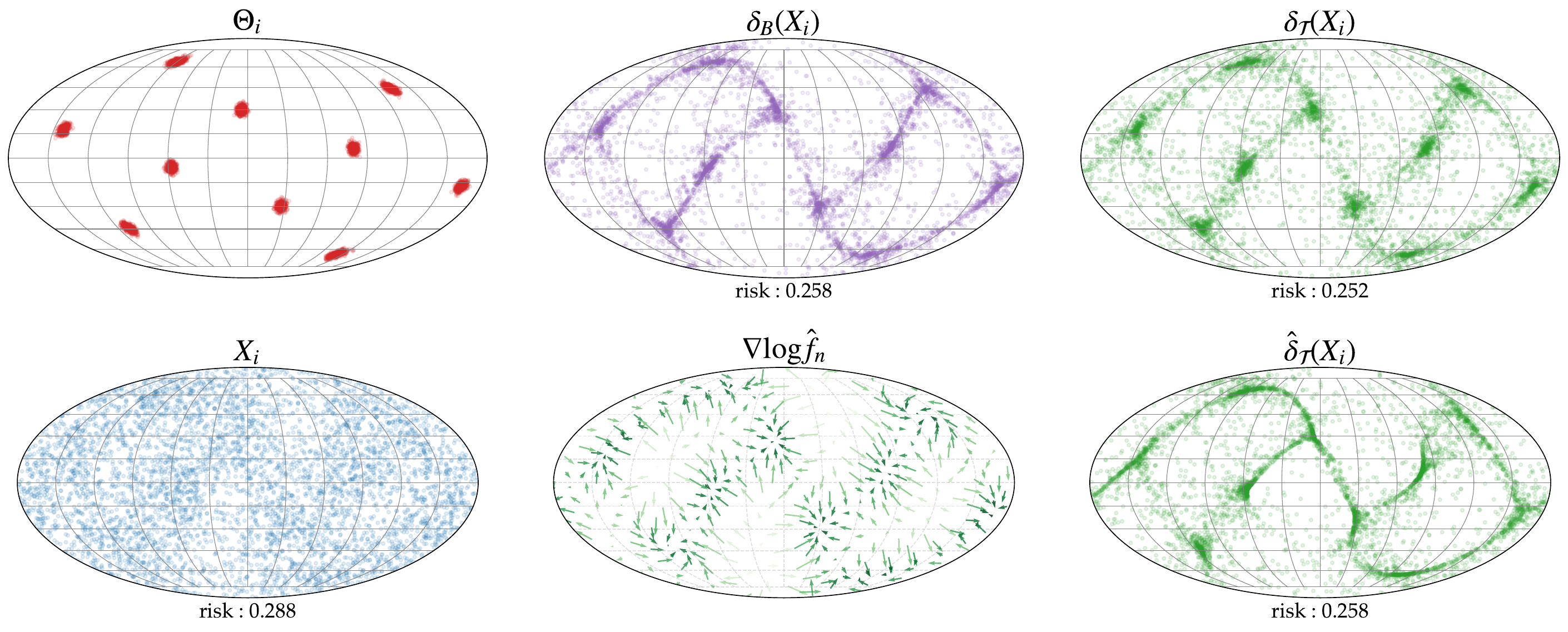}
        \caption{
        Illustration of the `10-modal' example in the sphere $\mathbb{S}^2$.}
    \label{fig:S2_10_model}
    \end{figure}

\subsection{Stratified Ramachandran Plots from Subsection~\ref{subsec:protein}}

Second, we provide versions of Figure~\ref{fig:chemistry} where each enzyme is plotted separately.
We emphasize that the denoising still utilizes the full pooled data set; but, viewing the denoised data in a stratified way allows one to understand the effect of shrinkage on each individual enzyme.

\begin{figure}[h!]
\centering
\includegraphics[width=0.9\linewidth]{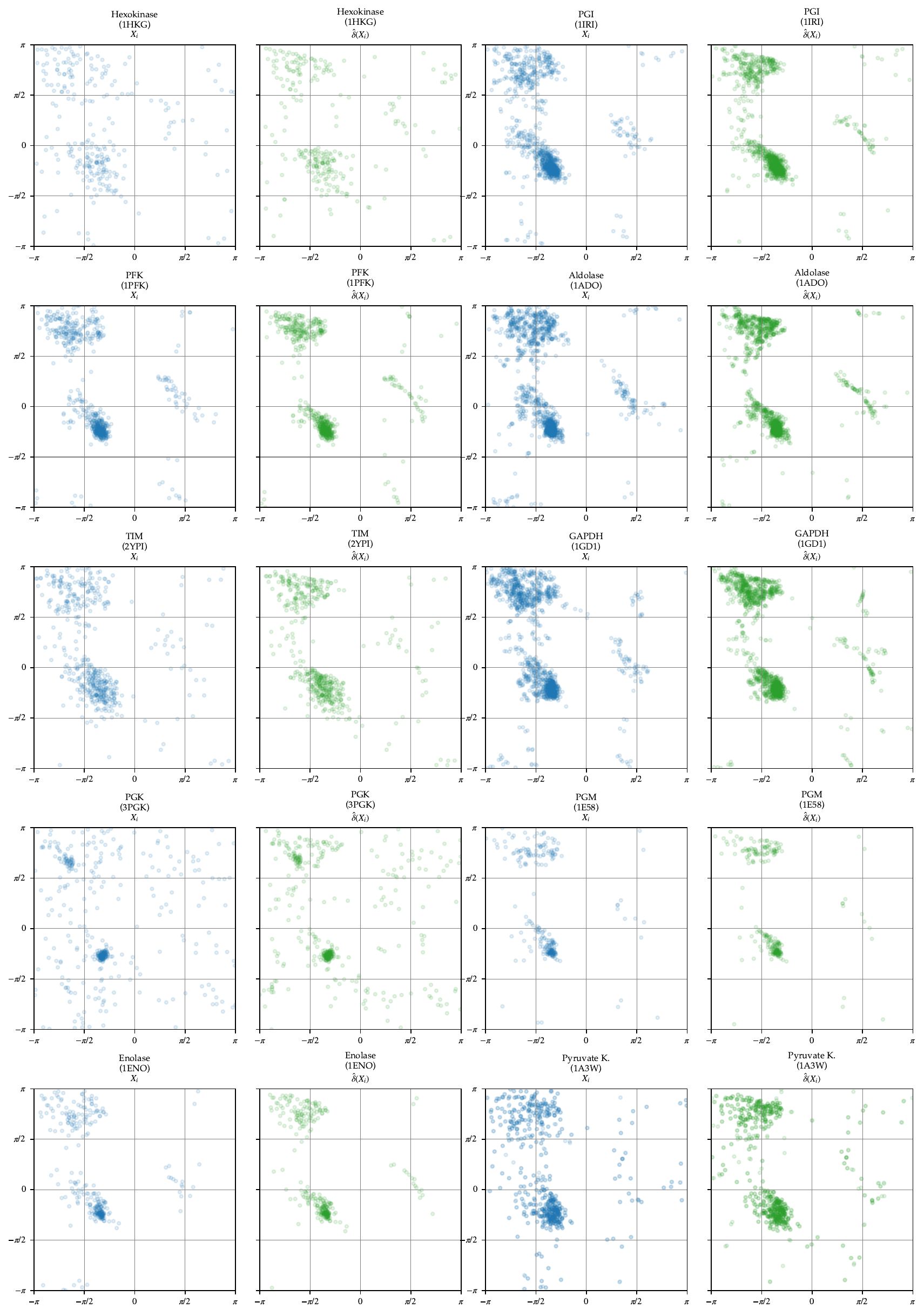}
\caption{Visualizing the denoised Ramachandran plots for the ten glycolytic enzymes of Subsection~\ref{subsec:protein}.
For each enzyme $1\le j\le 10$, we show the measurements $X_1,\ldots, X_{n_j}$ (first column, third column) and the denoised data set $\hat{\delta}_{\T}(X_1),\ldots, \hat{\delta}_{\T}(X_{n_j})$ (second column, fourth column).}
\label{fig:chemistry-perenzyme}
\end{figure}

\end{appendix}

\end{document}